\documentclass[a4paper,reqno,11pt]{amsart}
\usepackage[left=1 in, right=1 in,top=1 in, bottom=1 in]{geometry}
\usepackage{amsfonts}
\usepackage{amssymb} 
\usepackage{amsthm} 
\usepackage{amsmath} 
\usepackage{mathrsfs} 
\usepackage{color}
\usepackage{enumerate}
\usepackage{hyperref}
\usepackage[numbers,sort&compress]{natbib}
\usepackage{pdfsync} 
\usepackage{esint}
\usepackage{graphicx}
\usepackage{float}
\usepackage{epstopdf}
\usepackage{caption}
\usepackage{subfigure}
\usepackage{url}
\usepackage{dsfont}
\usepackage{enumitem}
\usepackage{bm}
\usepackage{setspace}
\usepackage{inputenc} 
\newtheorem{theorem}{Theorem}[section]
\newtheorem{definition}{Definition}[section]
\newtheorem{lemma}{Lemma}[section]
\newtheorem{remark}{Remark}[section]
\newtheorem{proposition}{Proposition}[section]

\numberwithin{equation}{section}

\renewcommand{\d}{\operatorname{d}}

\renewcommand{\u}{\mathbf{u}}
\newcommand{\w}{\mathbf{w}}

\newcommand{\ls}{\leqslant}

\newcommand{\ra}{\rightarrow}

\newcommand{\eps}{\varepsilon}
\renewcommand{\u}{{\bf u}}

\renewcommand{\v}{{\bf{v}}}
\newcommand{\z}{{\bf{z}}}

\usepackage{stmaryrd}
\renewcommand{\O}{\mathbb{T}^3}
\allowdisplaybreaks[4]

\begin{document}

\title[Stationary solution to SEP in Bounded Domain]{Stationary solution to Stochastically Forced Euler-Poisson Equations in Bounded Domain:\\
 Part 2. 1-D Ohmic Contact Boundary}
\author{Yachun Li$^{1}$, Ming Mei$^{2,3,4}$, Lizhen Zhang$^{*,5}$}
\dedicatory{ {\footnotesize\it $^1$School of Mathematical Sciences, CMA-Shanghai, MOE-LSC, Shanghai Jiao Tong University, 200240, China}\\
{\footnotesize\it $^2$School of Mathematics and Statistics, Jiangxi Normal University, Nanchang, 330022, China} \\
{\footnotesize\it $^3$Department of Mathematics, Champlain College Saint-Lambert,
	    Saint-Lambert, Quebec, J4P 3P2, Canada} \\
{\footnotesize\it $^4$Department of Mathematics and Statistics, McGill University,
	      Montreal, Quebec, H3A 2K6, Canada  }\\
{\footnotesize\it $^5$School of Mathematical Sciences, Shanghai Jiao Tong University,
	     Shanghai, 200240, China }\\
{\tt Emails: ycli@sjtu.edu.cn; ming.mei@mcgill.ca; Zhanglizhen@sjtu.edu.cn;}
}

\begin{abstract}
In this paper, we establish the asymptotic stability of the steady-state for the 1-D stochastic Euler-Poisson equations with Ohmic contact boundary conditions forced by the Wiener process. We utilize Banach's fixed point theorem and the {\it a priori} energy estimates uniformly in time to ensure the global existence of solutions around the steady state. In contrast to the deterministic case, the presence of stochastic forces leads to the lack of temporal derivatives of momentum, posing challenges for energy estimates. Furthermore, Ohmic contact boundary conditions pose greater challenges for energy estimates compared to the insulating boundary conditions. To address this issue, we establish asymptotic stability concerning the spatial derivatives through weighted energy estimates for the estimates of stochastic integrals, employing a technique distinct from that of the deterministic case. Furthermore, we demonstrate the existence of an invariant measure based on the {\it a priori} energy estimates. This invariant measure precisely corresponds to the Dirac measure generated by the steady state, due to the time-exponential decay of the perturbed solutions around the steady state.
\end{abstract}

\date{\today}

\maketitle
 {\small \tableofcontents}
 \setcounter{tocdepth}{2}
\section{Introduction}
The model of the stochastic Euler-Poisson equations (SEP for short) provides a precise depiction of physical phenomena for the analysis and design of semiconductor devices \cite{deBisschop2018StochasticEI} since the stochastic effects occasionally lead to undesired defects and pattern roughness in chips.
Thus, we explore the dynamic model of semiconductors perturbed by stochastic forces within mathematical frameworks. In Part 1 \cite{Zhang-Li-Mei2024}, we studied the existence and the symptomatic behavior of the stationary solution to SEP in 3-D bounded domain with the insulating boundary conditions. The Ohmic contact boundary conditions pose more challenges than the insulating boundary conditions. For a multi-dimensional bounded domain with Ohmic contact boundary conditions, there are no results on the well-posedness of deterministic solutions around the steady states. Hence we merely consider the 1-D SEP. In 1-D bounded smooth domain $U$, the 1-D SEP reads as
\begin{equation}
\left\{\begin{array}{l}\label{1-d Euler Poisson}
\rho_{t}+ J_{x} =0, \\
 \d  J + \left( \left(\frac{J^{2}}{\rho}+P(\rho)\right)_{x}+\frac{J}{\tau}\right)\d t = \rho \Phi_{x}\d t + \mathbb{F}\left(\rho,J\right) \d  W, \\
\Phi_{xx} =\rho-b,
\end{array}\right.
\end{equation}
where $``\d "$ is the differential notation with respect to time $t$, $\rho$ is the electron density,  $J$ is the current, $P\left(\rho\right)$ is the pressure, $\Phi$ is the electrostatic potential, and $\tau$ is the velocity relaxation time, which is regarded as a constant, assumed to be $1$ without loss of generality thereafter. $b:=b(x)$ is called the doping profile, which is positive and immobile. We also introduce the electricity field $E$ satisfying $E=\Phi_{x}$.
Stochastic processes $\rho$, $J$, $\Phi$, $E$, and $P\left(\rho\right)$ are actually functions of $\omega$, $t$, and $x$, where $\omega$ is a sample in the probability space $\left(\Omega, \mathbb{P}\right)$. For convenience, we use the simplified notions $\rho$, $J$, $\Phi$, $E$ and $P\left(\rho\right)$ here and hereafter.
 $W$ is an $\mathcal{H}$-valued cylindrical Wiener process defined on the filtrated probability space $\left(\Omega, \mathfrak{F}, \mathbb{P}\right)$, where $\mathcal{H}$ is an auxiliary separable Hilbert space, and $\mathfrak{F}$ is the filtration, referring the definitions of filtration and Wiener process in Appendix \ref{append}. Let $\{e_{k}\}_{k=1}^{+\infty}$ be an orthonormal basis in $\mathcal{H}$. Then the cylindrical Wiener process is
$W=\sum\limits_{k=1}^{+\infty}e_{k}\beta_{k}$, where $\{\beta_{k}(t); k \in \mathds{N}, t\geqslant 0\}$ is a sequence of independent, real-valued standard Brownian motions.
Let $H$ be a Bochner space, $\mathbb{F}\left(\rho,J\right)$ be an $H$-valued operator from $ \mathcal{H}$ to $ \mathcal{H}$. Denoting the inner product in $\mathcal{H}$ as $  \langle\cdot,\cdot \rangle_{\mathcal{H}} $,
the inner product
\begin{equation}
\langle\mathbb{F}\left(\rho,J\right),e_{k}\rangle_{\mathcal{H}} = F_{k}\left(\rho,J\right)
\end{equation}
 is an $H$-valued vector function, which shows the strength of the external stochastic forces,
 \begin{equation}
 \mathbb{F}\left(\rho, J\right)\d  W=\sum\limits_{k=1}^{+\infty}F_{k}\left(\rho,J\right)\d  \beta_{k} e_{k},
  \quad \mathbb{F}\left(\rho, J\right)=\sum\limits_{k=1}^{+\infty}F_{k}\left(\rho,J\right) e_{k}.
 \end{equation}
We will write $F_{k}\left(\rho,J\right)$ as $F_{k}$ for short in the sequel.

We assume that
\begin{align}\label{condition for F}
 \mathbb{F}\left(\rho,J\right)\d  W=&\sum\limits_{k=1}^{+\infty}F_{k}\left(\rho,J\right)\d  \beta_{k} e_{k},\\
F_{k} = a_{k}JY\left(J\right),
 \quad  \sum a_{k}^{2}=1,  \quad  &\left|Y\left(J\right)\right|\ls C J, \quad \left|Y'\left(J\right)\right|\ls C, \quad \left|Y''\left(J\right)\right|\ls C,
\end{align}
where ${a_{k}}$ are positive constants. 
Regarding the Euler-Poisson equations \eqref{1-d Euler Poisson}, the Ohmic contact boundary condition is given by
\begin{align}\label{1-D contact boundary of rho}
\rho\left(\omega, t, 0\right)=\rho_{1},\quad \rho(\omega, t, 1)=\rho_{2},
\end{align}
where $\rho_{1}$ and $\rho_{2}$ are positive constants, and the boundary condition for the electrostatic potential $\Phi$ is required:
\begin{align}\label{1-D contact boundary of Phi}
\Phi\left(\omega, t, 0\right)= \Phi_{1}, \quad \Phi(\omega, t, 1)=\Phi_{2},
\end{align}
where $\Phi_{1}$ and $\Phi_{2}$ are constants.

 The stability problem of \eqref{1-d Euler Poisson} with \eqref{1-D contact boundary of rho}, \eqref{1-D contact boundary of Phi}, along with the following initial data in $\left(\Omega, \mathfrak{F}, \mathbb{P}\right)$
\begin{align}\label{initial conditions}
 \left(\rho_{0}(\omega, x),~J_{0}(\omega, x), ~\Phi_{0}(\omega, x) \right),
\end{align}
 is studied in this paper. We denote $\left(\rho_{0}(\omega, x),~J_{0}(\omega, x), ~\Phi_{0}(\omega, x) \right)$ by $\left(\rho_{0},~J_{0}, ~\Phi_{0}\right)$ for short.

For the multiple dimensional case, Guo-Strauss \cite{Guo2006StabilityOS} considered the deterministic 3-D Euler-Poisson equations in a bounded domain with insulating boundary, and showed the convergence of solutions to the 3-D subsonic steady-states. The stability analysis of semiconductor systems with the contact boundary conditions is more complicated compared to those with the insulating boundary conditions. This complexity arises due to the absence of studies on well-posedness of steady state for 2-D or 3-D deterministic semiconductor systems with the contact boundary conditions. Generally, there are no solutions or no uniqueness of solutions to the elliptic equation with the Dirichlet boundary conditions. Researchers make attempts to get close to this problem. Mei-Wu-Zhang \cite{MeiWu2021Stability} considered the stability of the steady-state for 3-D spherical symmetric semiconductors in a hollow ball, i.e., the radius $r>\eps$ is required, for some small constant $\eps$. Markowich \cite{Markowich1991OnSS} discussed the existence of the subsonic solution in two dimensions with the irrotational currents on the boundary. Further exploration into the stability of steady state solutions for the Cauchy problem was undertaken by Hattori \cite{Hattori1997}. The Cauchy problems to the deterministic Euler-Poisson equations were extensively studied in \cite{DMRS,Huang1,Huang2,Huang2011,Huang2011SIAM,Kawashima1984SystemsOA}. For the ill-posedness of solutions of multidimensional Cauchy problem, one can refer to Ogawa-Suguro-Wakui's work \cite{Ogawa-Suguro-Wakui2022}. For the case of free boundary with vacuum, we refer to \cite{Luo1,Mai,Zeng} and the references therein. For the formulation of singularities in compressible Euler-Poisson equations and the large time behavior of Euler equations with damping, one can refer to \cite{Sideris-Thomases-Wang} and \cite{Wang-Chen1998JDE}, respectively.

For the deterministic case with the 1-D contact boundary condition, Degond-Markowich \cite{Degond1990OnAsym} studied the existence of steady states and the uniqueness for a given current density $\bar{J}$. Li-Markowich-Mei \cite{Li-Markowich-Mei-2002} explored the existence and uniqueness of steady state solution for the flat doping profile. Guo-Strass \cite{Guo2006StabilityOS} also considered the steady state $\left(\bar{\rho}, \bar{J}, \bar{\Phi}\right)$, and constructed a global solution near the steady state. While the current $\bar{J}$ is small, $\bar{\rho}$, $\bar{\Phi}$, and the doping profile $b(x)$ have large variation. Then, the uniqueness is reconsidered by Nishibada-Suzuki \cite{Nishibata2007AsymptoticStability} for a given boundary voltage.

As we addressed in Part 1 \cite{Zhang-Li-Mei2024}, the stochastic forces exhibit at most H\"older-$\frac{1}{2}-$continuity in time $t$, leading to the decreased regularity of velocity over time. From a mathematical perspective, investigating the stochastic problem enables us to examine the long time behavior of solutions to the stochastic Euler-Poisson equations in the absence of strong regularity in time and determine whether the desired properties persist in the presence of specific types of noise, providing valuable theoretical guidance.

The solution to stochastic evolving systems is called the stationary solution provided that the increment of solutions during evolution is time-independent.  Originally, the study of stationary measures dates back to the works of Hopf \cite{Hopf1932TheoryOM}, Doeblin \cite{Doeblin1938}, Doob \cite{Doob1953}, Halmos \cite{Halmos1946, Halmos1947}, Feller \cite{Feller1957}, and Harris and Robbins \cite{Harris-Robbins1953, Harris1956}, for the theory of discrete Markov processes. The study of invariant measure of fluid models dates back to Cruzeiro \cite{Cruzeiro1989} for stochastic incompressible Navier-Stokes equations in 1989, by Galerkin approximation with dimensions $D\geqslant 2$.
Flandoli \cite{Flandoli1994} and Flandoli-Gatarek \cite{Flandoli1-Gatarek1995} proved the existence of an invariant measure for multidimensional incompressible Navier-Stokes equations by different methods with \cite{Cruzeiro1989}.
Mattingly \cite{Mattingly2002} proved the existence of exponentially attracting invariant measure with respect to initial data, for incompressible N-S equations.
Later, Goldys-Maslowski \cite{Goldys-Maslowski2005} showed that transition measures of 2-D stochastic Navier-Stokes equations converge exponentially fast to the corresponding invariant measures in the distance of total variation. Then, for the 3-D case, Da Prato and Debussche \cite{DaPrato-Debussche2003} constructed a transition semigroup for 3-D stochastic Navier-Stokes equations, which allows for rather irregular solutions without the uniqueness. Flandoli-Romito \cite{Flandoli-Romito2008} used the classical Stroock-Varadhantype argument to find the almost-sure Markov selection. 
For stochastic compressible Navier-Stokes equations, Breit-Feireisl-Hofmanov\'a-Maslowski \cite{Breit-Feireisl-Hofmanova-Maslowski2019} proved the existence of stationary solutions. 
Hofmanov\'a-Zhu-Zhu \cite{Hofmanova-Zhu-Zhu2022} selected the dissipative global martingale solutions to the stochastic incompressible Euler system, and obtained the non-uniqueness of strong Markov solutions. Very recently, they \cite{Hofmanova-Zhu-Zhu2024} showed that stationary solution to the Euler equations is a vanishing viscosities limit in law of stationary analytically weak solutions to Navier-Stokes equations.
In terms of the non-uniqueness studies, some scholars believe that a certain stochastic perturbation can provide
a regularizing effect of the underlying PDE dynamics. For instance, Flandoli-Luo \cite{Flandoli-Luo2021} showed that a noise of transport type prevents a vorticity blow-up in the incompressible Navier-Stokes equations. A linear multiplicative noise prevents the blow up of the velocity with high probability for the 3-D Euler system, which was shown by Glatt-Holtz-Vicol \cite{Glatt-Holtz-Vicol2014}. Tang \cite{Tang2023} extended the noise of transport type to pseudo-differential/multiplicative noise for stochastic {E}uler-{P}oincar\'e equations. 
To the best of our knowledge, the stationary solutions of SEP with Ohmic contact boundary have not been explored previously. For our SEP, with better regularity than Euler equations, we could show the existence and uniqueness of invariant measure in more regular space.

To study the asymptotic behavior of solutions to SEP, we first establish the global existence and uniqueness of perturbed solutions around the steady state for the Euler-Poisson equations. Subsequently, we demonstrate the existence of stationary solutions and invariant measure, based on the {\it a priori} energy estimates and weighted energy estimates.

 In the probability space $\left(\Omega, \mathfrak{F}, \mathbb{P}\right)$, the random variables $\left(\bar{\rho}(\omega, x), \bar{J}(\omega, x), \bar{\Phi}(\omega, x)\right)$ are supposed to satisfy the same equations as the steady states equations
\begin{align}
\left\{\begin{array}{l}\label{steady state for 1-D contact}
 \bar{J}\equiv Constant,\\
\left(\frac{\bar{J}^{2}}{\bar{\rho}}+P\left(\bar{\rho}\right)\right)_{x}+\bar{J}=\bar{\rho}\bar{\Phi}_{x},\\
\bar{\Phi}_{xx}=\bar{\rho}-b(x).
 \end{array}\right.
\end{align}
By multiplying $\frac{1}{\bar{\rho}}$ with \eqref{steady state for 1-D contact}, and taking divergence on the resultant equation, we have
\begin{align}\label{equation of bar rho}
\left(P'\left(\bar{\rho}\right)-\frac{\bar{J}^{2}}{\bar{\rho}^{2}}\right)\frac{1}{\bar{\rho}}\bar{\rho}_{xx}
+\left(\frac{4\bar{J}^{2}}{\bar{\rho}^{4}}+\frac{P''\left(\bar{\rho}\right)}{\bar{\rho}}-\frac{P'\left(\bar{\rho}\right)}{\bar{\rho}^{2}}\right)\left|\bar{\rho}_{x}\right|^{2}
+-\frac{\bar{J}}{\bar{\rho}^{2}}\bar{\rho}_{x}=\bar{\rho}-b(x).
\end{align}
We consider the case that $\bar{J}$ is sufficiently small so that the subsonic condition $P'\left(\bar{\rho}\right)-\frac{\bar{J}^{2}}{\bar{\rho}^{2}}>0$ holds. The equation of $\bar{\rho}$ \eqref{equation of bar rho} is uniformly elliptic. Since we consider the small perturbation around the steady state, the solutions $\rho$ and $J$ are subjected to the subsonic condition $P'\left(\rho\right)-\frac{J^{2}}{\rho^{2}}>0$.
The existence and uniqueness of the system \eqref{steady state for 1-D contact} with the contact boundary condition was obtained by \cite{Li-Markowich-Mei-2002}.
\begin{proposition}\label{existence of steady state} \cite{Guo2006StabilityOS}
Let $b(x)>0 $ in $\bar{U}$, $\bar{J}$ be a small constant, and $P:\left(0, \infty\right)\rightarrow \left(0, \infty\right)$ be smooth with $P\left(0\right)=0$. Then there exists $\left(\bar{\rho}(x), \bar{J}, \bar{\Phi}(x)\right)$, a unique smooth steady-state solution, with the Ohmic contact boundary condition
\begin{align}
\bar{\rho}\left(0\right)=\rho_{1}, \quad \bar{\rho}(1)=\rho_{2}, \quad \bar{\Phi}\left( 0\right)=\Phi_{1}, \quad \bar{\Phi}(1)=\Phi_{2},
\end{align}
such that there holds
\begin{align}
\bar{\rho}\left( x\right)>\underline{\rho}>0, \int_{U} \bar{\rho}(x)\d x = \int_{U}b(x)\d x,\quad \forall x\in \bar{U},
\end{align}
where $\underline{\rho}$ is a constant, and $\bar{Q}$ is supposed to satisfy
\begin{align}
 \bar{Q}_{x}(\bar{\rho})=\bar{\Phi}_{x}.
\end{align}
\end{proposition}

In this paper, we make a clear distinction between stationary solution for the stochastically forced system \eqref{1-d Euler Poisson} and the steady-state for the deterministic system \eqref{steady state for 1-D contact}.

For every $\omega\in \Omega$, $\left(\bar{\rho}(x, \omega), \bar{J}(\omega), \bar{\Phi}(x, \omega)\right)= \left(\bar{\rho}(x), \bar{J}, \bar{\Phi}(x)\right)$ holds. We denote by $\left(\bar{\rho}, \bar{J}, \bar{\Phi}\right)$ for convenience. Hence, the law of steady state is Dirac measure $\delta_{\bar{\rho}}\times \delta_{\bar{J}}\times \delta_{\bar{\Phi}}$, see Appendix \ref{append}.

Our result is about the existence and asymptotic stability of solutions near the steady state with the Ohmic contact boundary condition. We denote
\begin{align}
\sigma=\rho-\bar{\rho},&\quad j=J-\bar{J},\\
\phi=\Phi-\bar{\Phi}, &\quad \tilde{e}=\phi_{x}.\notag
\end{align}
We apply Banach's fixed point theorem and uniform energy estimates to establish the global existence of $\left(\sigma, j, \tilde{e} \right)$. Furthermore, we demonstrate the weighted energy estimates to achieve the asymptotic stability for steady states with the Ohmic contact boundary conditions. Based on the {\it a priori} estimates uniformly in $t$, by Krylov-Bogoliubov's theorem, we imply that the approximating measures will converges to an invariant measure. For this case, the invariant measure for \eqref{1-d Euler Poisson} precisely coincides with the law of steady state.

We denote $\left\|\cdot\right\|$ as the $L^{2}(U)$-norm; the $L^{\infty}(U)$-norm is denoted by $\left\|\cdot\right\|_{\infty}$; $\left\|\cdot\right\|_{k}$ means the $H^{k}$-norm. $\mathcal{L}\left(\cdot\right)$ is the law of random variables; see the definition of law in Appendix \ref{append}. $L^{2m}\left(\Omega; C\left([0,T]; H^{k}\left(U\right)\right)\right)$ is the space in which the $2m-$th moment of the $C\left([0,T]; H^{k}\left(U\right)\right)$-norm of random variables is bounded. The following is the statement of our main theorem.

\begin{theorem}
Let $\left(\bar{\rho},\bar{J}, \bar{\Phi}\right)$ be the smooth steady state of \eqref{1-d Euler Poisson} in Proposition \ref{existence of steady state}. If there exists a constant $\eps>0$ such that the initial condition $\left(\rho_{0}, J_{0}, \Phi_{0}\right)$ satisfies
\begin{align}\label{assumption for initial data}
\mathbb{E}\left[\left(\left\| \rho_{0} - \bar{\rho} \right\|_{2}^{2}+\left\| J_{0} -\bar{J} \right\|_{2}^{2}+\left\|\left(\Phi_{0}\right)_{x} -\bar{\Phi}_{x} \right\|^{2}\right)^{m}\right] \ls \eps^{2m},
\end{align}
for any $m\geqslant 2$, and
\begin{align}\label{Initial formula of Phi}
\left|\bar{J}\right|\ls \eps,
\quad \left(\Phi_{0}\right)_{xx} = \rho_{0} - b(x),
\end{align}
then in $\left(\Omega, \mathfrak{F}, \mathbb{P}\right)$, there exists a unique strong global-in-time solution of \eqref{1-d Euler Poisson}: $\rho$, $ J$ are in $L^{2m}\left(\Omega; C\left([0,T]; H^{2}\left([0,1]\right)\right)\right)$, $\Phi\in L^{2m}\left(\Omega; C\left([0,T]; H^{4}\left([0,1]\right)\right)\right)$ up to a modification. Moreover, there hold the asymptotic stability and the existence of invariant measure:
\begin{enumerate}
  \item There are positive constants $C$ and $\zeta$ such that the expectation
\begin{align}\label{small perturbation condition}
&\mathbb{E}\left[\left|\sup\limits_{s\in[0,t]}\left(\left\| \rho -\bar{\rho} \right\|_{2}^{2}+\left\| J -\bar{J} \right\|_{2}^{2} + \left\|\Phi_{x} -\bar{\Phi}_{x} \right\|^{2}\right) \right|^{m}\right]\notag\\
\ls  & C e^{-\zeta m t}\mathbb{E}\left[\left(\left\| \rho_{0} - \bar{\rho} \right\|_{2}^{2}+\left\| J_{0} -\bar{J}\right\|_{2}^{2}+\left\|\left(\Phi_{0}\right)_{x} -\bar{\Phi}_{x} \right\|^{2}\right)^{m}\right],
\end{align}
holds, where $C$ is independent of $t$, and $C$ is the $m$-th power of some constant;
\item The invariant measure generated by $\frac{1}{T}\int_{0}^{T} \mathcal{L}\left(\rho \right)\times\mathcal{L}\left(J \right)\times\mathcal{L}\left(\Phi \right) \d t$ is exactly the Dirac measure of steady state $\left(\bar{\rho} , \bar{J} , \bar{\Phi} \right)$.
\end{enumerate}
\end{theorem}

\begin{remark}
After $t\ra \infty$ in \eqref{1-d Euler Poisson}, the stationary solution coincides with the steady state $\mathbb{P}$ {\rm a.s.}, on account that the $m$-th moment of their difference tends to zero.
\end{remark}

\begin{remark}
For
\begin{align}
\left(\left\|\rho_{0} - \bar{\rho} \right\|_{2}^{2}+\left\| J_{0} -\bar{J} \right\|_{2}^{2}+\left\|E_{0} -\bar{E} \right\|^{2}\right)\ls \eps^{2},
\end{align}
 recalling $E=\Phi_{x}$, we have the $\mathbb{P}$~{\rm a.s} asymptotic stability for some constant $\tilde{C}$:
\begin{align}
 \sup\limits_{s\in[0,t]}\left(\left\| \rho -\bar{\rho} \right\|_{2}^{2}+\left\| J -\bar{J}\right\|_{2}^{2}+\left\|E -\bar{E} \right\|^{2}\right)
\ls   2\tilde{C} e^{-\alpha t} \eps^{2}.
\end{align}
In fact, by Chebyshev's inequality (see Appendix \ref{append}), it holds that
\begin{align}
&\mathbb{P}\left[\left\{\omega\in \Omega|  \left\| \rho -\bar{\rho} \right\|_{2}^{2}+\left\| J -\bar{J}\right\|_{2}^{2}+\left\|E -\bar{E} \right\|^{2}
>  2\tilde{C} e^{-\alpha t} \eps^{2}\right\}\right]\notag\\
\ls &\frac{\mathbb{E}\left[\left|\left\| \rho -\bar{\rho} \right\|_{2}^{2}+\left\| J -\bar{J}\right\|_{2}^{2}+\left\|E -\bar{E} \right\|^{2}\right|^{m}\right]}{\left(2\tilde{C} e^{-\alpha t} \eps^{2}\right)^{m}}\\
\ls &\frac{\mathbb{E}\left[\tilde{C}e^{-\alpha t}\left(\left\| \rho_{0} - \bar{\rho} \right\|_{2}^{2}+\left\|J_{0} -\bar{J}\right\|_{2}^{2}+ \left\|E_{0} -\bar{E} \right\|^{2}\right)^{m}\right]}{\left(2\tilde{C} e^{-\alpha t} \eps^{2}\right)^{m}}=\frac{1}{2^{m}}.\notag
\end{align}
Sending $m\ra \infty$, we have
$$\mathbb{P}\left[\left\{\omega\in \Omega| \left\| \rho -\bar{\rho} \right\|_{2}^{2} + \left\|J -\bar{J}\right\|_{2}^{2}+\left\|E -\bar{E} \right\|^{2}>2\tilde{C} e^{-\alpha t} \eps^{2}\right\}\right]\ra 0,$$
which means that\\
 $\left\|\rho -\bar{\rho} \right\|_{2}^{2}+\left\|J -\bar{J}\right\|_{2}^{2}+\left\|E -\bar{E} \right\|^{2}\ls 2\tilde{C} e^{-\alpha t} \eps^{2}$ holds $\mathbb{P}$ {\rm a.s.} for every $s\in[0, t]$.
\end{remark}


The challenges and their corresponding strategies are analyzed as follows.
\begin{enumerate}
\item \textbf{The a priori estimates with the contact boundary condition.} For the contact boundary condition, the second-order estimates of perturbed solutions involve the boundary of $\left[\begin{array}{l}\sigma \\ j\end{array}\right]_{xx}\cdot \left[\begin{array}{l}\sigma \\ j\end{array}\right]_{xx}$, which can not be directly dealt by the contact boundary condition as presented in previous works such as \cite{Guo2006StabilityOS,MeiWu2021Stability}. We solve this problem by substituting $\left[\begin{array}{l}\sigma \\ j\end{array}\right]_{xx}$ with its formula in the first-order derivative of the system. Subsequently, we focus on estimating other lower-order derivatives in the first-order derivative of the system. Further, for the second-order estimates, the symmetrzing matrix in the first-order estimates is not applicable. We use another symmetrizing matrix to multiply the system, enabling us to handle the integral of $\left[\begin{array}{l}\sigma \\ j\end{array}\right]_{xxx}\cdot \left[\begin{array}{l}\sigma \\ j\end{array}\right]_{xx}$. For the zeroth-order estimates, due to the Ohmic boundary conditions, we use the method of error analysis compared to the symmetrizing method for the insulating boundary case.

\item \textbf{No temporal solutions due to the stochastic term.} Since the Wiener process is at most H\"oder-$\frac{1}{2}-$ continuous with respect to $t$ and nowhere differentiable, we lack $\frac{\partial W}{\partial t}$ or $\frac{\partial (\rho u)}{\partial t}$. The temporal derivatives are not involved in the norm of solutions. Therefore, the estimates bounded by the norm of the temporal derivatives, which are applicable in the deterministic cases \cite{Guo2006StabilityOS,MeiWu2021Stability}, do not hold in this case. As a result, different energy estimates are necessary in this paper. For instance, for the second-order estimates, we use another symmetrizing matrix to multiply the system, enabling us to handle the integral of $\left[\begin{array}{l}\sigma \\ j\end{array}\right]_{xxx}\cdot \left[\begin{array}{l}\sigma \\ j\end{array}\right]_{xx}$ compared to the first-order estimates.

\item \textbf{Weighted energy estimates on account of the estimates of the stochastic integral.} In the deterministic case, it is common practice to multiply the ordinary differential equation of solutions by an exponential function of $t$ to facilitate stability analysis. However, in this paper, the {\it a priori} estimates are already in the form of time integrals rather than a differential inequality, owing to Burkholder-Davis-Gundy's inequality. Consequently, direct acquisition of asymptotic stability becomes challenging. To overcome this obstacle, we employ the weighted energy estimates. Furthermore, besides the role of ensuring the global existence, the {\it a priori} estimates play a crucial role in the weighted energy estimates.
\end{enumerate}

The paper is structured as follows. In \S 2, we demonstrate the {\it a priori} estimates up to second-order. In \S 3, we establish the global existence of solutions around the steady state. In \S 4, we investigate the asymptotic stability of solutions. \S 5 is about the global existence of invariant measures. \S 6 is the Appendix, in which we outline some stochastic analysis theories utilized in this study.
\smallskip
\smallskip

\section{The a priori estimates up to second-order}

In terms of $\left(\sigma, j, \tilde{e}\right)$, the perturbed Euler-Poisson system becomes
\begin{align}
\left\{\begin{array}{l}\label{1-D sto semiconductor around teady state}
\sigma_{t}+ j_{x}=0,\\
\d j + \left(\left(\frac{\left(\bar{J}+j\right)^{2}}{\bar{\rho}+\sigma}-\frac{\bar{J}^{2}}{\bar{\rho}}+P\left(\bar{\rho}+\sigma\right)-P\left(\bar{\rho}\right)\right)_{x}
  +j-\bar{\rho}\tilde{e}-\sigma \bar{\Phi}_{x}\right)\d t =  \sigma \tilde{e}\d t + \mathbb{F}\d W,\\
\tilde{e}_{x}=\sigma.
 \end{array}\right.
\end{align}
The compatible conditions are $\left.\rho_{1}\right|_{t=0}=\bar{\rho}(0), \quad \left.\rho_{2}\right|_{t=0}=\bar{\rho}(1)$. Since $\sigma_{t}+j_{x}=0$, the boundary conditions for the triple of perturbations $\left(\sigma,\u,\phi\right)$ are
\begin{align}\label{boundary condtion for perturbations}
\sigma (t,0)\equiv\sigma (t,1)\equiv j_{x}(t,0)\equiv j_{x} (t,1)\equiv\phi(t,0)\equiv\phi(t,1)\equiv0.
\end{align}
Furthermore, from $\left(\tilde{e}_{t}+j\right)_{x}=\sigma_{t}+j_{x}=0$, as in \cite{Guo2006StabilityOS} we still have
\begin{align}
\tilde{e}_{t}=-j+\int_{0}^{1} j \d x.
\end{align}
Denoting $\w=\left[\begin{array}{l}\sigma \\ j\end{array}\right]$, we write the system as
\begin{align}\label{1-d system with contact boundary}
\d  \w+\left(\mathcal{A}\w_{x}+\mathcal{B}\w + \mathcal{C}\right) \d t= \mathcal{N}\d t+ \left[\begin{array}{c}0\\\mathbb{F}\d  W\end{array}\right],
\end{align}
where
\begin{align}\label{mathcal A}
\mathcal{A}=\left[\begin{array}{cc}
0 & 1 \\
  -\left(\frac{\bar{J}+j}{\bar{\rho}+\sigma}\right)^{2}+P'\left(\bar{\rho}+\sigma\right) & 2\frac{\bar{J}+j}{\bar{\rho}+\sigma}
\end{array}\right],
\end{align}
\begin{align}\label{mathcal B}
\mathcal{B}=\left[\begin{array}{cc}
0 & 0 \\
\frac{-2\bar{J}^{2}}{\bar{\rho}^{3}}\bar{\rho}_{x} + P''\left(\bar{\rho}\right)\bar{\rho}_{x}-\bar{E} & \frac{-2\bar{J}}{\bar{\rho}^{2}}\bar{\rho}_{x}+1
\end{array}\right],
\end{align}
\begin{align}\label{mathcal C}
\mathcal{C}=\left[\begin{array}{c}
0 \\
-\bar{\rho}\tilde{e}
\end{array}\right],
\end{align}
\begin{align}\label{mathcal N}
\mathcal{N}=\left[\begin{array}{c}0\\O\left(\left|\w\right|^{2}+\left|\tilde{e}\right|^{2}+\left|\bar{J}\right|\left|\w\right|\right)\end{array}\right].
\end{align}
Under the assumption that
\begin{align}
 \left\|\w\right\|_{2}^{2}(t)+\left\|\tilde{e}\right\|^2
\end{align}
is small, we obtain the {\it a priori} estimates. Motivated by \cite{Guo2006StabilityOS}, we perform the zeroth-order estimates through error analysis. Furthermore, the symmetrizing matrix for the second-order derivative of the system differs from that for the first-order, serving different purposes.
In the subsequent subsections, we provide the {\it a priori} estimates up to second-order.

\subsection{Zeroth-order estimates on $\w$}\label{Zeroth-order estimates on w}

We calculate $\d  \left(\frac{J^{2}}{2\rho}\right)$ by It\^o's formula (see Appendix \ref{append}), taking $X=\frac{J}{2\rho}$, $Y=J$ in \eqref{Ito for XY},
\begin{align}\label{zero order balance}
& \d  \left(\frac{J^{2}}{2\rho}\right)=J\d \left(\frac{J}{2\rho}\right)+\frac{J}{2\rho}\d J+\langle \d \frac{J}{2\rho}, \d J\rangle\notag\\
=&\frac{J}{\rho}\d  J + \frac{-J^{2}}{2\rho^{2}}\d  \rho + \frac{1}{2\rho}\left\langle\d  J,\d  J\right\rangle \\
=& \frac{J}{\rho}\left(-\left(\frac{J^{2}}{\rho}+P(\rho)\right)_{x}-J+ \rho \Phi_{x}\right)\d t+\frac{J^{2}}{2\rho^{2}} J_{x} \d t + \frac{J}{\rho}\mathbb{F}\d  W  + \frac{1}{2\rho} \left|\mathbb{F}\right|^{2}\d t,
\end{align}
in which we used the property of cross quadratic variations $\langle \d \rho, \d J\rangle =0$ and $\frac{1}{2\rho}\left\langle\d  J,\d  J\right\rangle =\frac{1}{2\rho} \left|\mathbb{F}\right|^{2}\d t$.
We calculate
\begin{align}
\frac{J}{\rho}\left(-\frac{J^{2}}{\rho}\right)_{x} = \frac{J^{3}}{\rho^{3}}\rho_{x}-2\frac{J^{2}}{\rho^{2}}J_{x}.
\end{align}
Setting $G''\left(\rho\right)=\frac{ P'\left(\rho\right)}{\rho}$, we have
\begin{align}
&P\left(\rho\right)_{x}\frac{J}{\rho} = P'\left(\rho\right)\rho_{x}\frac{J}{\rho}
=\frac{ P'\left(\rho\right)}{\rho}\rho_{x}J=\left(G'\left(\rho\right)\right)_{x}J \\
=&-G'\left(\rho\right)J_{x}+ \left(G'\left(\rho\right)J\right)_{x}=G\left(\rho\right)_{t}+\left(G'\left(\rho\right)J\right)_{x}.\notag
\end{align}
Since $\tilde{e}_{t}=-j+\int_{0}^{1}j\d x$ and $\bar{J}$ being a constant, there holds
\begin{align}
\tilde{e}_{t}=-J+\int_{0}^{1}J\d x,\\
J=-\tilde{e}_{t}+\int_{0}^{1}J\d x.
\end{align}
Together with $\tilde{e}=\Phi_{x}-\bar{\Phi}_{x}=\phi_{x}$, it holds that
\begin{align}
\frac{J}{\rho}\rho\Phi_{x}=J\Phi_{x}=\Phi_{x}\left(-E_{t}+\int_{0}^{1}J\d x\right)=-\Phi_{x}\Phi_{xt} + \Phi_{x}\int_{0}^{1}J\d x.
\end{align}
Combining the above terms, we obtain
\begin{align}\label{equation of mathcal F}
&\d  \left(\frac{J^{2}}{2\rho}+G\left(\rho\right)+\frac{\Phi_{x}^{2}}{2}\right)+\left(\left(\frac{J^{3}}{2\rho^{2}}+G'\left(\rho\right)J\right)_{x}-\Phi_{x}\int_{0}^{1}J\d x+\frac{J^{2}}{\rho}\right)\d t \\
=& \frac{J\mathbb{F}}{\rho}\d  W + \frac{1}{2\rho}\left|\mathbb{F}\right|^{2}\d t. \notag
\end{align}
We denote the equation \eqref{equation of mathcal F} as
\begin{align}
\mathcal{G}\left(\rho, J, \Phi_{x}\right)=&\d  \left(\frac{J^{2}}{2\rho}+G\left(\rho\right)+\frac{\Phi_{x}^{2}}{2}\right)+\left(\left(\frac{J^{3}}{2\rho^{2}}+G'\left(\rho\right)J\right)_{x}-\Phi_{x}\int_{0}^{1}J\d x+\frac{J^{2}}{\rho}\right)\d t\\
& -\frac{J\mathbb{F}}{\rho}\d  W-\frac{1}{2\rho}\left|\mathbb{F}\right|^{2}\d t. \notag
\end{align}
From the process of deriving the expression of $\mathcal{G}$, we notice that
\begin{align}
&\mathcal{G}\left(\rho, J, \Phi_{x}\right) = \mathcal{F}\left(\rho, J, \Phi_{x}\right)\frac{J}{\rho},
\end{align}
where $\mathcal{F}$ satisfies
\begin{align}
&\mathcal{F}\left(\rho, J, \Phi_{x}\right)=\d  J + \frac{-J}{2\rho}\d \rho
- \left(-\left(\frac{J^{2}}{\rho}+P(\rho)\right)_{x}-J+ \rho \Phi_{x}\right)\d t - \mathbb{F}\d  W -\frac{J}{2\rho} J_{x} \d t\notag\\
&\qquad \qquad \quad =0,\\
&\mathcal{F}\left(\bar{\rho}, \bar{J}, \bar{\Phi}_{x}\right)=-\mathbb{F}\left(\bar{J}\right)\d  W.
\end{align}
Let $\hat{\w}=\bar{\w}+\eps \w$ be the solution of the continuity equation, where $\bar{\w}=\left(\bar{\rho},\bar{J}, \bar{\Phi}\right)$. It also holds
\begin{align}
&\mathcal{G}\left(\hat{\w}, \bar{\Phi}_{x}+\eps \phi_{x}\right) = \mathcal{F}\left(\hat{\w}, \bar{\Phi}_{x}+\eps \phi_{x}\right)\frac{\bar{J}+\eps j}{\bar{\rho}+\eps \sigma},
\end{align}
since we just use the mass equation and Poisson equation to do the substitution in the above deformative equalities.
We denote
\begin{align}
\mathscr{E}\left(\hat{\w},\bar{\Phi}_{x}+\eps \phi_{x}\right)=&\frac{\left(\bar{J}+\eps j\right)^{2}}{2\left(\bar{\rho}+\eps\rho\right)}+G\left(\bar{\rho}+\eps \rho\right)+\frac{\left(\bar{\Phi}_{x}+\eps \phi_{x}\right)^{2}}{2}, \\
\mathscr{D}\left(\hat{\w}\right)=&\frac{\left(\bar{J}+\eps j\right)^{3}}{2\left(\bar{\rho}+\eps\rho\right)^{2}}+G'\left(\bar{\rho}+\eps\rho\right)\left(\bar{J}+\eps j\right),\\
\mathscr{R}\left(\hat{\w},\bar{\Phi}_{x}+\eps \phi_{x}\right)=&\left(-\left(\bar{\Phi}+\eps \phi\right)_{x}\int_{0}^{1}\left(\bar{J}+\eps j\right)\d x+\frac{\left(\bar{J}+ \eps j\right)^{2}}{\bar{\rho}+\eps \rho}\right).
\end{align}
Then we have
\begin{align}
\mathcal{G}\left(\hat{\w},\bar{\Phi}_{x}+\phi_{x}\right)
= & \d  \mathscr{E}\left(\hat{\w},\bar{\Phi}+\eps \phi\right)+\mathscr{D}_{x}\left(\hat{\w}\right)\d t+\mathscr{R}\left(\hat{\w},\bar{\Phi}_{x}+\eps \phi_{x}\right)\d t\\
  &-\frac{1}{2\left(\bar{\rho}+\eps \rho\right)}\left|\mathbb{F}\left(\bar{J}+\eps j\right)\right|^{2}\d t-\frac{\left(\bar{J}+\eps j\right)\mathbb{F}\left(\bar{J}+\eps j\right)}{\left(\bar{\rho}+\eps \rho\right)}\d  W,\notag
\end{align}
and
\begin{align}
&\mathcal{G}\left(\bar{\w},\bar{\Phi}_{x}\right)\\
=&\d  \mathscr{E}\left(\bar{\w},\bar{\Phi}_{x}\right)+\mathscr{D}_{x}\left(\bar{\w}\right)\d t+\mathscr{R}_{x}\left(\bar{\w},\bar{\Phi}_{x}\right)\d t-\frac{1}{2\bar{\rho}}\left|\mathbb{F}\left(\bar{J}\right)\right|^{2}\d t-\frac{\bar{J}\mathbb{F}\left(\bar{J}\right)}{\bar{\rho}}\d  W. \notag
\end{align}
By Taylor's expansion, it holds that
\begin{align}
&\mathcal{G}\left(\bar{\w}+\w,\bar{\Phi}_{x}+ \phi_{x}\right)\\
=&\mathcal{G}\left(\bar{\w},\bar{\Phi}_{x}\right)+\mathcal{G}_{\eps}'\left(\bar{\w},\bar{\Phi}_{x}\right)\cdot\left[\begin{array}{c}\w\\ \phi_{x}\end{array}\right]+ \frac{1}{2}\mathcal{G}_{\eps\eps}''\left(\bar{\w},\bar{\Phi}_{x}\right)\left[\begin{array}{c}\w\\ \phi_{x}\end{array}\right]\cdot\left[\begin{array}{c}\w\\ \phi_{x}\end{array}\right]+O\left(|\w|^{3}+\left|\phi_{x}\right|^{3}\right).\notag
\end{align}
There holds
\begin{align}
&\left.\frac{\d^{2} }{\d \eps^{2}}\left(\mathcal{G}\left(\hat{\w},\bar{\Phi}_{x}+\eps \phi_{x}\right)\right)\right|_{\eps=0}= \left.\frac{\d^{2} }{\d \eps^{2}}\left(\mathcal{F}\left(\hat{\w},\bar{\Phi}_{x}+\eps \phi_{x}\right)\frac{\hat{J}}{\hat{\rho}}\right)\right|_{\eps=0}\\
=&\mathcal{F}''\left(\bar{\w},\bar{\Phi}_{x}\right) \left[\begin{array}{c}\w\\ \phi_{x}\end{array}\right]\left(\w\quad \phi_{x}\right)
\frac{\bar{J}}{\bar{\rho}}+2\left[\begin{array}{c}\w\\ \phi_{x}\end{array}\right] \mathcal{F}'\left(\bar{\w},\bar{\Phi}_{x}\right) \left(\frac{j}{\bar{\rho}}-\frac{\bar{J}\sigma}{\bar{\rho}^{2}}\right)+\mathcal{F}\left(\bar{\w},\bar{\Phi}_{x}\right)O\left(\left|\w\right|^{2}\right),\notag
\end{align}
By Taylor's expansion, it holds that
\begin{align}
&-\mathcal{F}'\left(\bar{\w},\bar{\Phi}_{x}\right)\left[\begin{array}{c}\w\\ \phi_{x}\end{array}\right] =\mathcal{F}\left(\bar{\w}+\w,\bar{\Phi}_{x}+ \phi_{x}\right)-\mathcal{F}\left(\bar{\w},\bar{\Phi}_{x}\right)-\mathcal{F}'\left(\bar{\w},\bar{\Phi}_{x}\right)\cdot \left[\begin{array}{c}\w\\ \phi_{x}\end{array}\right]\\
 = &\mathcal{F}\left(\bar{\w}, \bar{\Phi}_{x}\right)+ \frac{1}{2}\left(\w\quad \phi_{x}\right)\mathcal{F}''\left(\tilde{\w},\tilde{\phi}_{x}\right)\left[\begin{array}{c}\w\\ \phi_{x}\end{array}\right] ,\notag
\end{align}
where $\left(\tilde{\w},\tilde{\phi}_{x}\right)$ lies between $\left(\bar{\w},\bar{\Phi}_{x}\right)$ and $\left(\bar{\w}+\w,\bar{\Phi}_{x}+ \phi_{x}\right)$.
The error terms
\begin{align}
\left| \int_{0}^{1}\left[\begin{array}{c}\w\\ \phi_{x}\end{array}\right]\mathcal{F}''\left(\bar{\w},\bar{\Phi}_{x}\right)\left[\w\quad \phi_{x}\right]
\frac{\bar{J}}{\bar{\rho}}\left[\begin{array}{c}\w\\ \phi_{x}\end{array}\right]\cdot\left[\begin{array}{c}\w\\ \phi_{x}\end{array}\right]\d x\right| \ls C \left\|\w\right\|_{2}^{4}\ls C \left\|\w\right\|_{2}^{3},
\end{align}
\begin{align}
&\left| 2\left[\begin{array}{c}\w\\ \phi_{x}\end{array}\right] \mathcal{F}'\left(\bar{\w},\bar{\Phi}_{x}\right) \left(\frac{j}{\bar{\rho}}-\frac{\bar{J}\sigma}{\bar{\rho}^{2}}\right)\left[\begin{array}{c}\w\\ \phi_{x}\end{array}\right]\cdot\left[\begin{array}{c}\w\\\phi_{x}\end{array}\right]\right|\\
= & \left|\left(\mathcal{F}\left(\bar{\w}, \bar{\Phi}_{x}\right) + \frac{1}{2}\left[\w\quad \phi_{x}\right]\mathcal{F}''\left(\tilde{\w},\tilde{\phi}_{x}\right)\left[\begin{array}{c}\w\\ \phi_{x}\end{array}\right]\right) \left(\frac{j}{\bar{\rho}}-\frac{\bar{J}\sigma}{\bar{\rho}^{2}}\right)\left[\begin{array}{c}\w\\ \phi_{x}\end{array}\right]\cdot\left[\begin{array}{c}\w \\ \phi_{x}\end{array}\right]\right|= O\left(\left|\w\right|^{3}\right), \notag
\end{align}
and
\begin{align}
&\left| O\left(\left|\w\right|^{2}\right)\mathcal{F}\left(\bar{\w},\bar{\Phi}_{x}\right)\d W \left[\begin{array}{c}\w\\ \phi_{x}\end{array}\right]\cdot\left[\begin{array}{c}\w\\ \phi_{x}\end{array}\right]\right| = O\left(\left|\w\right|^{4}\right)\d W,
\end{align}
are small for $\left\|\w\right\|_{2}$ small enough. 
We calculate term by term:
\begin{align}
\d \left(\frac{\left(\bar{J}+j\right)^{2}}{2\left(\bar{\rho}+\sigma\right)}-\left(\frac{\bar{J}^{2}}{2\bar{\rho}}+\frac{\bar{J}}{\bar{\rho}}j-\frac{\bar{J}^{2}}{2\bar{\rho}^{2}}\sigma\right)\right)
= \d \left(\frac{\left(\bar{\rho}j-\bar{J}\sigma\right)^{2}}{2\left(\bar{\rho}+\sigma\right)\bar{\rho}^{2}}\right),
\end{align}
\begin{align}
\d \left(\frac{\left(\bar{\Phi}_{x}+\phi_{x}\right)^{2}}{2}-\left(\frac{\bar{\Phi}_{x}^{2}}{2}-\bar{\Phi}_{x}\phi_{x}\right)\right)
= \d \left(\frac{\phi_{x}^{2}}{2}\right),
\end{align}
\begin{align}
 &\left(\frac{\left(\bar{J}+j\right)^{3}}{2\left(\bar{\rho}+\sigma\right)^{2}}-\left(\frac{\bar{J}^{3}}{2\left(\bar{\rho}\right)^{2}}
 -\frac{\bar{J}^{3}}{\bar{\rho}^{3}}\sigma+\frac{3\bar{J}^{2}}{2\bar{\rho}^{2}}j\right)\right)_{x}\\
=&\left(\frac{3\bar{J}j^{2}}{\bar{\rho}^{2}}-\frac{3\bar{J}^{2}j\sigma}{\bar{\rho}^{3}}+\frac{3 \bar{J}^{2} \sigma^{2}}{\bar{\rho}^{4}}+O\left(\left(\left|j\right|+|\sigma|\right)^{3}\right)\right)_{x},\notag
\end{align}
\begin{align}
G'\left(\rho\right)J-G'\left(\bar{\rho}\right)\bar{J}-G''\left(\bar{\rho}\right)\bar{J}\sigma-G'\left(\bar{\rho}\right)j=O\left(\left|\bar{J}\right|\left| \w \right|^{2} \right),
\end{align}
\begin{align}
& \left(\bar{\Phi}_{x}+\phi_{x}\right)\int_{0}^{1}\left(\bar{J}+j\right) \d x-\left(\bar{\Phi}_{x}\right)\int_{0}^{1}\bar{J} \d x-\phi_{x}\int_{0}^{1}\bar{J}\d x-\bar{\Phi}_{x}\int_{0}^{1}j \d x\\
=& 2\phi_{x}\int_{0}^{1}j \d x \left|\phi_{x}\right|^{2}+ o\left(\left|\phi_{x}\right|^{3}\right)= O\left(\left|\phi_{x}\right|^{3}\right)=O\left(\left|\tilde{e}\right|^{3}\right),\notag
\end{align}

\begin{align}
& \frac{J \mathbb{F}\d  W}{\rho}-\frac{\bar{J}\mathbb{F}\left(\bar{J}\right)\d  W}{\bar{\rho}}-\frac{j\mathbb{F}\left(\bar{J}\right)\d  W}{\bar{\rho}}-\frac{\bar{J}j\mathbb{F}'\left(\bar{J}\right)\d  W}{\bar{\rho}}+\frac{\bar{J}\mathbb{F}\left(\bar{J}\right)\sigma\d  W}{\bar{\rho}^{2}}\\
=& -\left(\frac{2\mathbb{F}'\left(\bar{J}\right)}{\bar{\rho}}+\frac{\bar{J}\mathbb{F}''\left(\bar{J}\right)}{\bar{\rho}}\right)\frac{1}{2}j^{2}\d  W-2\frac{\bar{J}\mathbb{F}\left(\bar{J}\right)}{\bar{\rho}^{3}}\frac{1}{2}\sigma^{2}\d  W
-\frac{\bar{J}\mathbb{F}'\left(\bar{J}\right)+\mathbb{F}\left(\bar{J}\right)}{\bar{\rho}^{2}}j\sigma\d  W \notag\\
&+ O\left(\left(\left|j\right|^{3}+\left|\sigma\right|^{3}\right)\d  W\right),\notag
\end{align}
and
\begin{align}
  & \frac{1}{2\rho}\left|\mathbb{F}\right|^{2}\d t-\frac{1}{2\bar{\rho}}\left|\mathbb{F}\left(\bar{J}\right)\right|^{2}\d t- \frac{\mathbb{F}\left(\bar{J}\right)\mathbb{F}'\left(\bar{J}\right)}{\bar{\rho}}j\d t-\left((-1)\frac{\left|\mathbb{F}\left(\bar{J}\right)\right|^{2}}{\bar{\rho}^{2}}\right)\sigma\d t \notag\\
= &\frac{\mathbb{F}'\left(\bar{J}\right)^{2}+\mathbb{F}\left(\bar{J}\right)\mathbb{F}''\left(\bar{J}\right) }{2\bar{\rho}}j^{2} \d t- \frac{\mathbb{F}\left(\bar{J}\right)\mathbb{F}''\left(\bar{J}\right)}{\bar{\rho}^{2}} j\sigma \d t+ 2\frac{\left|\mathbb{F}\left(\bar{J}\right)\right|^{2}}{\bar{\rho}^{3}}\sigma^{2}\d t\\
&+O\left(\left(\left|j\right|^{3}+\left|\sigma\right|^{3}\right)\left|\bar{J}\right|\right)\d t. \notag
\end{align}
Combining the above terms and integrating over $x$, noticing that the two ways to describe the errors should be equivalent, we obtain
\begin{align}\label{equality for 1-D zeroth-order}
& \d \int_{0}^{1}\left(\frac{\left(\bar{\rho}j-\bar{J}\sigma\right)^{2}}{2\left(\bar{\rho}+\sigma\right)\bar{\rho}^{2}}+G\left(\bar{\rho}+\sigma\right)
 -G\left(\bar{\rho}\right)-G'\left(\bar{\rho}\right)\sigma+\frac{\tilde{e}^{2}}{2}\right)\d x \notag\\
& + \int_{0}^{1} \frac{\left(\bar{\rho}j-\bar{J}\sigma\right)^{2}}{2\left(\bar{\rho}+\sigma\right)\bar{\rho}^{2}} \d x \d t +\left.\left(\frac{3\bar{J}j^2}{\bar{\rho}^{2}}+O\left(\left(\left|j\right|^{3}+\left|\sigma\right|^{3}\right)\right)\right) \right|_{0}^{1}\d t + O\left(\left|\bar{J}\right|\left\|\w\right\|_{2}^{2} \right)\d t \notag\\
&+\int_{0}^{1}\frac{\mathbb{F}'\left(\bar{J}\right)^{2}+\mathbb{F}\left(\bar{J}\right)\mathbb{F}''\left(\bar{J}\right) }{2\bar{\rho}}j^{2} \d x \d t -\int_{0}^{1}\frac{\mathbb{F}\left(\bar{J}\right)\mathbb{F}''\left(\bar{J}\right)}{\bar{\rho}^{2}} j\sigma \d x \d t\\
& - \int_{0}^{1}2\frac{\left|\mathbb{F}\left(\bar{J}\right)\right|^{2}}{\bar{\rho}^{3}}\sigma^{2} \d x \d t + \int_{0}^{1}O\left(\left(\left|j\right|^{3}+\left|\sigma\right|^{3}\right)\left|\bar{J}\right|\right)\d x \d t+O\left(\left\|\w\right\|_{2}^{3}\right)\d t\notag\\
&+\int_{0}^{1}\left(\frac{2\mathbb{F}'\left(\bar{J}\right)}{\bar{\rho}}+\frac{\bar{J}\mathbb{F}''\left(\bar{J}\right)}{\bar{\rho}}\right)\frac{1}{2}j^{2}\d x \d  W+\int_{0}^{1}\frac{\bar{J}\mathbb{F}\left(\bar{J}\right)}{\bar{\rho}^{3}}\sigma^{2}\d x \d  W\notag\\
&+\frac{\bar{J}\mathbb{F}'\left(\bar{J}\right)+\mathbb{F}\left(\bar{J}\right)}{\bar{\rho}^{2}}j\sigma\d  W  +\int_{0}^{1} O\left(\left(\left|j\right|^{3}+\left|\sigma\right|^{3}\right)\right)\d x\d  W =0. \notag
\end{align}
We estimate the stochastic integral $\int_{0}^{t}\int_{0}^{1}\left(\frac{2\mathbb{F}'\left(\bar{J}\right)}{\bar{\rho}}+\frac{\bar{J}\mathbb{F}''\left(\bar{J}\right)}{\bar{\rho}}\right)\frac{1}{2}j^{2}\d x\d  W $ by Burkh\"older-Davis-Gundy's inequality (see Appendix \ref{append}):
\begin{align}
&\mathbb{E}\left[\left|\int_{0}^{t}\int_{0}^{1}\left(\frac{2\mathbb{F}'\left(\bar{J}\right)}{\bar{\rho}}+\frac{\bar{J}\mathbb{F}''\left(\bar{J}\right)}{\bar{\rho}}\right)\frac{1}{2}j^{2}\d x\d  W \right|^{m}\right]\notag\\
\ls & C \mathbb{E}\left[\left(\int_{0}^{t}\left|\int_{0}^{1}\left(\frac{2\mathbb{F}'\left(\bar{J}\right)}{\bar{\rho}}+\frac{\bar{J}\mathbb{F}''\left(\bar{J}\right)}{\bar{\rho}}\right)\frac{1}{2}j^{2}\d x\right|^{2}\d s\right)^{\frac{m}{2}}\right]\notag\\
\ls &C \mathbb{E}\left[\left(\int_{0}^{t}\left|\bar{J}\right|^{2}\left\| j \right\|_{2}^{4} \d s\right)^{\frac{m}{2}}\right] \\
\ls & C \mathbb{E}\left[\left(\sup_{s\in [0,t]}\left|\bar{J}\right| \left\| j \right\|_{2}^{2} \right)^{\frac{m}{2}}\left(\int_{0}^{t}\left|\bar{J}\right|\left\| j \right\|_{2}^{2} \d s\right)^{\frac{m}{2}}\right]\notag\\
\ls & \vartheta_{1} \mathbb{E}\left[\left(\sup_{s\in [0,t]}\left|\bar{J}\right| \left\| j \right\|_{2}^{2} \right)^{m}\right]+C_{\vartheta_{1}} \mathbb{E}\left[\left(\int_{0}^{t}\left|\bar{J}\right|\left\| j \right\|_{2}^{2} \d s\right)^{m}\right], \notag
\end{align}
where $C$ is a $m$-th power of some constant.
Similarly, we have
\begin{align}
& \mathbb{E}\left[\left|\int_{0}^{t}\int_{0}^{1}\frac{\bar{J}\mathbb{F}\left(\bar{J}\right)}{\bar{\rho}^{3}}\sigma^{2}\d x \d  W \right|^{m}\right]\\
 \ls &\frac{\vartheta_{2}}{2}  \mathbb{E}\left[\left(\sup_{s\in [0,t]}\left|\bar{J}\right| \left\| \sigma \right\|_{2}^{2} \right)^{m}\right]+C_{\vartheta_{2}} \mathbb{E}\left[\left(\int_{0}^{t}\left|\bar{J}\right|\left\| \sigma \right\|_{2}^{2} \d s \right)^{m}\right],\notag
\end{align}
\begin{align}
& \mathbb{E}\left[\left|\int_{0}^{t}\int_{0}^{1}\frac{\bar{J}\mathbb{F}'\left(\bar{J}\right)+\mathbb{F}\left(\bar{J}\right)}{\bar{\rho}^{2}}j\sigma\d  W\right|^{m}\right]\\
 \ls &\frac{\vartheta_{2}}{2}  \mathbb{E}\left[\left(\sup_{s\in [0,t]}\left|\bar{J}\right| \left\| \sigma \right\|_{2}^{2} \right)^{m}\right]+C_{\vartheta_{2}} \mathbb{E}\left[\left(\int_{0}^{t}\left|\bar{J}\right|\left\| j \right\|_{2}^{2} \d s \right)^{m}\right],\notag
\end{align}
and
\begin{align}
&\mathbb{E}\left[\left|\int_{0}^{t}\int_{0}^{1} O\left(\left(\left|j\right|^{3}+\left|\sigma\right|^{3}\right)\right)\d x\d  W \right|^{m}\right] \\
\ls &\vartheta_{3}\mathbb{E}\left[\left( \sup\limits_{s\in [0, t]}\left\|\w\right\|_{2}^{3}\right)^{m}\right]+C_{\vartheta_{3}} \mathbb{E}\left[\left( \int_{0}^{t}\left\|\w\right\|_{2}^{3} \d s\right)^{m}\right], \notag
\end{align}
where $\vartheta_{3}$ is sufficiently small such that $\vartheta_{3} \sup\limits_{s\in [0, t]}\left\|\w\right\|_{2}^{3}$ can be controlled by the error term in the left side.
In summary, we have
\begin{align}\label{zero order estimate of j}
& \mathbb{E}\left[\left|\int_{0}^{t}\d \int_{0}^{1}\left(\frac{\left(\bar{\rho}j-\bar{J}\sigma\right)^{2}}{2\left(\bar{\rho}+\sigma\right)\bar{\rho}^{2}}+G\left(\bar{\rho}+\sigma\right)
 -G\left(\bar{\rho}\right)-G'\left(\bar{\rho}\right)\sigma+\frac{\tilde{e}^{2}}{2}\right)(s)\d x \right|^{m}\right]\notag\\
& + \mathbb{E}\left[\left|\int_{0}^{t}\int_{0}^{1} \frac{\left(\bar{\rho}j-\bar{J}\sigma\right)^{2}}{2\left(\bar{\rho}+\sigma\right)\bar{\rho}^{2}} \d x \d s\right|^{m}\right] \notag \\
\ls &\mathbb{E}\left[\left|\int_{0}^{t}\left.\left(\frac{3\bar{J}j^2}{2\bar{\rho}^{2}}+ O\left(\left(\left|j\right|^{3}+\left|\sigma\right|^{3}\right)\right)\right) \right|_{0}^{1}\d s   \right|^{m}\right]+ \mathbb{E}\left[\left|\int_{0}^{t} O\left(\left|\bar{J}\right|\left\|\w\right\|_{2}^{2} \right)\d s \right|^{m}\right]\\
& +\mathbb{E}\left[\left|\int_{0}^{t}\int_{0}^{1} O\left(\left(\left|j\right|^{3}+\left|\sigma\right|^{3}\right)\left|\bar{J}\right|\right)\d x\d s\right|^{m}\right]+ C\mathbb{E}\left[\left|\int_{0}^{t} \left\|\w\right\|_{2}^{3} \d s\right|^{m}\right] \notag\\
&+\mathbb{E}\left[\left|\vartheta_{1} \sup\limits_{s\in [0, t]}\left|\bar{J}\right| \left\| j \right\|_{2}^{2}+\vartheta_{2} \sup\limits_{s\in [0, t]}\left|\bar{J}\right| \left\| \sigma \right\|_{2}^{2} \right|^{m}\right]\notag\\
\ls & C \mathbb{E}\left[\left|\int_{0}^{t}\left|\bar{J}\right|\left\|\w\right\|_{2}^{2} \d s \right|^{m}\right]+ C \mathbb{E}\left[\left|\int_{0}^{t} \left\|\w\right\|_{2}^{3} \d s\right|^{m}\right]\notag\\
&+\mathbb{E}\left[\left|\vartheta_{1} \sup\limits_{s\in [0, t]}\left|\bar{J}\right| \left\| j \right\|_{2}^{2}+\vartheta_{2} \sup\limits_{s\in [0, t]}\left|\bar{J}\right| \left\| \sigma \right\|_{2}^{2} \right|^{m}\right], \notag
\end{align}
where $C$ is a $m$-th power of some constant.
Since $G''\left(\rho\right)>0$, $G$ is strictly convex, the steady state is smooth, so the first integral above is at least $\mathbb{E}\left[\left|\int_{0}^{t}\d \left( 2\zeta_{1}\int_{0}^{1}\left(j^{2}+\sigma^{2}+\tilde{e}^{2}\right)\d x \right)\right|^{m}\right]$ for small enough $\left\|\w\right\|_{2}$ and $\left|\bar{J}\right|$.  $\vartheta_{1}$ and $\vartheta_{2} $ are small enough such that
\begin{align}
\mathbb{E}\left[\left|\vartheta_{1}\sup\limits_{s\in [0, t]}\left|\bar{J}\right| \left\| j \right\|_{2}^{2}+\vartheta_{2} \sup\limits_{s\in [0, t]}\left|\bar{J}\right| \left\| \sigma \right\|_{2}^{2} \right|^{m}\right]
\end{align}
is bounded by
\begin{align}
\mathbb{E}\left[\left|\sup\limits_{s\in[0,t]}\int_{0}^{s}\d_{\tau}\left( \zeta_{1}\int_{0}^{1}\left(j^{2}+\sigma^{2}+\tilde{e}^{2}\right)\d x \right)\right|^{m}\right].
 \end{align}
 Therefore, there holds
\begin{align}\label{zero order estimate of j summary}
& \mathbb{E}\left[\left|\zeta_{1}\int_{0}^{t}\d \int_{0}^{1}\left(j^{2}+\sigma^{2}+\tilde{e}^{2}\right)\d x\right|^{m}\right]+ \mathbb{E}\left[\left|\int_{0}^{t}\int_{0}^{1} \frac{\left(\bar{\rho}j-\bar{J}\sigma\right)^{2}}{2\left(\bar{\rho}+\sigma\right)\bar{\rho}^{2}} \d x \d s\right|^{m}\right] \notag \\
\ls & \mathbb{E}\left[\left|\int_{0}^{t} O\left(\left|\bar{J}\right|\left\|\w\right\|_{2}^{2} \right)\d s \right|^{m}\right]\\
  & +\mathbb{E}\left[\left|\int_{0}^{t}\int_{0}^{1} O\left(\left(\left|j\right|^{3}+\left|\sigma\right|^{3}\right)\left|\bar{J}\right|\right)\d x\d s\right|^{m}\right]+ C\mathbb{E}\left[\left|\int_{0}^{t} \left\|\w\right\|_{2}^{3} \d s\right|^{m}\right] \notag\\
\ls & C \mathbb{E}\left[\left|\int_{0}^{t}\left|\bar{J}\right|\left\|\w\right\|_{2}^{2} \d s \right|^{m}\right]+C \mathbb{E}\left[\left| \int_{0}^{t} \left\|\w\right\|_{2}^{3} \d s\right|^{m}\right],\notag
\end{align}
where $C$ is a $m$-th power of some constant.
Additionally, the estimate for the dissipation in terms of $\int_{0}^{t}\int_{0}^{1}\left(\sigma^{2}+\tilde{e}^{2}\right)\d x \d s$ is needed. We divide the momentum equation by $\bar{\rho}$, and use the steady state equation to get
\begin{align}
& \frac{\d  j}{\bar{\rho}} + \left(\frac{1}{\bar{\rho}}\left(P\left(\bar{\rho}+\sigma\right)-P\left(\bar{\rho}\right)\right)_{x}-\tilde{e}-\frac{\sigma}{\bar{\rho}}\bar{\Phi}_{x}-\frac{\sigma \tilde{e}}{\bar{\rho}} + \frac{j}{\bar{\rho}}\right) \d t \\
=& \left(-\frac{2\left(\bar{J}+j\right)}{\bar{\rho}\left(\bar{\rho}+\sigma\right)}j_{x} + \frac{1}{\bar{\rho}}\left(\left(\frac{\bar{J}+j}{\bar{\rho}+\sigma}\right)^{2}-\left(\frac{\bar{J}}{\bar{\rho}}\right)^{2}\right)\bar{\rho}_{x} + \frac{1}{\bar{\rho}}\left(\frac{\bar{J}+j}{\bar{\rho}+\sigma}\right)^{2}\sigma_{x}\right)\d t + \frac{\mathbb{F}}{\bar{\rho}}\d  W. \notag
\end{align}
As in \cite{Guo2006StabilityOS}, we have
\begin{align}
\frac{1}{\bar{\rho}}\left(P\left(\bar{\rho}+\sigma\right)-P\left(\bar{\rho}\right)\right)_{x}=\left(\frac{P'\left(\bar{\rho}\right)}{\bar{\rho}}\sigma\right)_{x}+ \frac{\bar{\rho}_{x}}{\bar{\rho}^{2}}P'\left(\bar{\rho}\right)\sigma+ \left(O\left(\sigma^{2}\right)\right)_{x}+O\left(\sigma^{2}\right) .
\end{align}
By \eqref{steady state for 1-D contact}, it holds 
\begin{align}
\bar{\Phi}_{x}= \frac{1}{\bar{\rho}}\left(\frac{\bar{J}^{2}}{ \bar{\rho}}\right)_{x} + \frac{1}{\bar{\rho}}\left( P\left(\bar{\rho}\right)\right)_{x} + \frac{\bar{J}}{\bar{\rho}},
\end{align}
and
\begin{align}
-\frac{\sigma \bar{\Phi}_{x}}{\bar{\rho}}=-\frac{P'\left(\bar{\rho}\right)\bar{\rho}_{x}\sigma}{\bar{\rho}^{2}}+ O\left(\bar{J}\right)\sigma.
\end{align}
Combining the terms, we have
\begin{align}\label{equation of E}
&-\left(\frac{P'\left(\bar{\rho}\right)}{\bar{\rho}}\sigma\right)_{x} \d t + \tilde{e}\d t \\
= &\frac{\d  j+ j\d t + \mathbb{F} \d  W}{\bar{\rho}}+ O\left(\bar{J}\right)\left(\left|\w\right|+ \left|\w_{x}\right|\right)\d t + O\left(\left|\w\right|^{2}+ \left|\w_{x}\right|^{2} + \tilde{e}^{2} \right) \d t. \notag
\end{align}
By It\^o's formula, we calculate
\begin{align}
\d \left(\frac{j\tilde{e}}{\bar{\rho}}\right)=\frac{\left(\d j\right) \tilde{e}}{\bar{\rho}}+\frac{j \left(\d \tilde{e}\right)}{\bar{\rho}},
\end{align}
due to $\left\langle \d \tilde{e}, \d j\right\rangle=0$. We multiple \eqref{equation of E} with $\tilde{e}$, and then integrate it over $x$. Since $\tilde{e}_{x}=\sigma=0$ at $x=0$ and $x=1$, we have
\begin{align}\label{balance of E}
&\int_{0}^{1} \frac{P'\left(\bar{\rho}\right)}{\bar{\rho}} \tilde{e}_{x}^{2} \d x \d t + \int_{0}^{1}\tilde{e}^{2} \d x \d t \notag\\
=&\d  \left(\int_{0}^{1}\frac{j\tilde{e}}{\bar{\rho}} \d x\right) + \int_{0}^{1}\frac{j^{2}}{\bar{\rho}} \d x \d t-\int_{0}^{1}j\d x \int_{0}^{1}\frac{j}{\bar{\rho}} \d x \d t+\int_{0}^{1}\frac{j\tilde{e}}{\bar{\rho}} \d x \d t  \\
& + \int_{0}^{1}\frac{\mathbb{F} \tilde{e}}{\bar{\rho}} \d x \d  W + \int_{0}^{1}O\left(\left|\bar{J}\right|\tilde{e} \left(\left|\w\right|+\left|\w_{x}\right|\right)\right)\d x \d t +\int_{0}^{1}O\left(\left|\w\right|^{2}+ \left|\w_{x}\right|^{2} + \tilde{e}^{2} \right)\tilde{e} \d x \d t. \notag
\end{align}
We deal with the right-hand side by H\"older's inequality:
\begin{align}
& -\int_{0}^{1}j\d x \int_{0}^{1}\frac{j}{\bar{\rho}} \d x \d t+\int_{0}^{1}\frac{j^{2}}{\bar{\rho}} \d x \d t + \int_{0}^{1}\frac{j\tilde{e}}{\bar{\rho}} \d x \d t \notag\\
\ls &C \int_{0}^{1}\frac{j^{2}}{\bar{\rho}} \d x \d t+\int_{0}^{1}\frac{\tilde{e}^{2}}{2} \d x \d t+\int_{0}^{1}\frac{j^{2}}{2\bar{\rho}^{2}} \d x \d t \\
\ls & C\int_{0}^{1}j^{2}\d x \d t + \int_{0}^{1}\frac{\tilde{e}^{2}}{2} \d x \d t,\notag
\end{align}
where $C$ depends on $\bar{\rho}$.
For the stochastic term, we estimate it as follows
\begin{align}
\mathbb{E}\left[\left|\int_{0}^{t}\int_{0}^{1}\frac{\mathbb{F}\tilde{e}}{\bar{\rho}}\d x \d  W\right|^{m}\right] \ls C \mathbb{E}\left[\left| \int_{0}^{t} \left|\bar{J}\right|\left\|\w\right\|_{2}^{2} \d s\right|^{m}\right] + C\mathbb{E}\left[\left|\int_{0}^{t}\left\|\w\right\|_{2}^{3}\d s\right|^{m}\right],
\end{align}
where $C$ is the $m$-th power of some constant.
Multiplying \eqref{balance of E} with some small constant, adding \eqref{zero order estimate of j summary}, we have
\begin{align}\label{conclusion of zero order estimate}
&\mathbb{E}\left[\left|\int_{0}^{1} \left(j^{2}+\sigma^{2}+ \tilde{e}^{2}\right)(t)\d x + \zeta_{2}\int_{0}^{t}\int_{0}^{1} \left(j^{2}+\sigma^{2}+ \tilde{e}^{2}\right) \d x \d s \right|^{m}\right] \notag\\
\ls &C_{1}^{m}\mathbb{E}\left[\left|\int_{0}^{t} \left|\bar{J}\right|\left\|\w\right\|_{2}^{2} \d s  +\int_{0}^{t}\left\|\w\right\|_{2}^{3}\d s+\int_{0}^{t}\int_{0}^{1}\left|\bar{J}\right|\tilde{e} \left(\left|\w\right|+\left|\w_{x}\right|\right)\d x \d s \right|^{m}\right] \\
&+C_{1}^{m}\mathbb{E}\left[\left|\int_{0}^{t}\int_{0}^{1}\left(\left|\w\right|^{2}+ \left|\w_{x}\right|^{2} + \tilde{e}^{2} \right)\tilde{e} \d x \d s\right|^{m}\right]\notag\\
\ls & \mathbb{E}\left[\left| C_{2}\int_{0}^{t} \left|\bar{J}\right|\left\|\w\right\|_{2}^{2} \d s\right|^{m}\right]  +\mathbb{E}\left[\left|C_{2}\int_{0}^{t}\left\|\w\right\|_{2}^{3}\d s\right|^{m}\right],\notag
\end{align}
where $\zeta_{2}$, $C_{1}$, and $C_{2}$ are positive constants.
\subsection{First-order estimates on $\w$}

Inspired by \cite{Guo2006StabilityOS}, we choose the approximately symmetrizing matrix:
\begin{align}\label{symmetrizing matrix of first order estim}
\mathcal{D}=\left[\begin{array}{cc}
s & 0 \\
0 & r
\end{array}\right],
\end{align}
where $s=\left\{-\frac{\bar{J}^{2}}{\bar{\rho}^{2}}+P'\left(\bar{\rho}\right)\right\} r$. $r$ satisfies
\begin{align}\label{r satisfies in first order estimates}
\left\{\frac{\bar{J}^{2}}{\bar{\rho}^{2}}-P'\left(\bar{\rho}\right)\right\}r_{x}+\left\{\frac{3\bar{J}^{2}}{\bar{\rho}^{3}}\bar{\rho}_{x}+P''\left(\bar{\rho}\right)\bar{\rho}_{x}-\frac{P'\left(\bar{\rho}\right)}{\bar{\rho}}\bar{\rho}_{x}-\frac{\bar{J}}{\bar{\rho}}\right\} r = 0.
\end{align}
We multiply \eqref{1-d system with contact boundary} on the left by $\mathcal{D}$ to obtain
\begin{align}\label{symmetrized sto Euler-Poisson system contact boundary}
\d  \left(\mathcal{D}\w\right)+\left(\bar{\mathcal{A}}\w_{x}+\bar{\mathcal{B}}\w + \bar{\mathcal{C}}\right) \d t= \bar{\mathcal{N}} \d t + \mathcal{D}\left[\begin{array}{c}0\\ \mathbb{F}\d  W\end{array}\right],
\end{align}
$\bar{\mathcal{A}}=\mathcal{D}\mathcal{A}$, $\bar{\mathcal{B}}=\mathcal{D}\mathcal{B}$, $ \bar{\mathcal{C}}=\mathcal{D}\mathcal{C}$, $\bar{\mathcal{N}} =\mathcal{D}\mathcal{N}$.
Differentiating \eqref{symmetrized sto Euler-Poisson system contact boundary} with respect to $x$, we get
\begin{align}\label{1 order derivative of symmetrized system 1}
&\d  \left(\mathcal{D} \w_{x}\right) - \left(\mathcal{D}_{x} \left(\mathcal{A}\w_{x}+\mathcal{B}\w + \mathcal{C}- \mathcal{N}\right)\right)\d t + \left(\bar{\mathcal{A}}\w_{xx}+ \left(\bar{\mathcal{A}}_{x}+\bar{\mathcal{B}}\right)\w_{x}+\bar{\mathcal{B}}_{x}\w + \bar{\mathcal{C}}_{x}\right) \d t \notag \\
=& \bar{\mathcal{N}}_{x} \d t + \mathcal{D}_{x}\left[\begin{array}{c}0\\\mathbb{F}\d  W\end{array}\right]  + \mathcal{D} \left[\begin{array}{c}0\\ \mathbb{F}_{x} \d  W\end{array}\right].
\end{align}
By It\^o's formula, it holds that
\begin{align}\label{equality of fiist order estim after Ito}
 & \d  \left( \mathcal{D}\frac{\left|\w_{x}\right|^{2}}{2}\right)
 = \mathcal{D} \d  \w_{x}\cdot \w_{x} + \frac{r^{-1}\left|r\mathbb{F}_{x}+r_{x}\mathbb{F}\right|^{2}}{2}\d t  \notag\\
=& \w_{x}\mathcal{D}_{x}\left(\mathcal{A}\w_{x}+\mathcal{B}\w + \mathcal{C}- \mathcal{N}\right)\d t- \bar{\mathcal{A}} \w_{xx}\w_{x}\d t  - \left(\bar{\mathcal{A}}_{x} + \bar{\mathcal{B}}\right) \w_{x}\cdot \w_{x} \d t \\
 & - \bar{\mathcal{B}}_{x}\w \w_{x}\d t - \bar{\mathcal{C}}_{x} \w_{x}\d t+ \bar{\mathcal{N}}_{x}\w_{x}\d t + r_{x} \mathbb{F} j_{x} \d  W + r \mathbb{F}_{x}j_{x} \d  W+\frac{r^{-1}\left|r\mathbb{F}_{x}+r_{x}\mathbb{F}\right|^{2}}{2}\d t.\notag
\end{align}
As in \cite{Guo2006StabilityOS}, by separating $\bar{\mathcal{A}}$ into symmetric matrix and an antisymmetric matrix, we have
\begin{align}
\int_{0}^{1} \bar{\mathcal{A}}\w_{xx} \cdot \w_{x} \d x = \int_{0}^{1} O\left(\left(\left|\w\right|+\left|\sigma_{x}\right|+\left|j_{x}\right|\right)^{3}\right)\d x.
\end{align}
Now we collect all the terms containing $\w_{x}^{2}$:
\begin{align}
\int_{0}^{1} \left(-\mathcal{D}_{x}\mathcal{A}+\bar{\mathcal{A}}_{x}+\bar{\mathcal{B}}-\frac{1}{2}\bar{\mathcal{A}}_{x}\right)\w_{x} \cdot \w_{x} \d x = \int_{0}^{1} \mathcal{Q} \w_{x}\cdot \w_{x} \d x,
\end{align}
where $ \mathcal{Q}=\frac{1}{2}\mathcal{D}\mathcal{A}_{x}-\frac{1}{2}\mathcal{D}_{x}\mathcal{A}+\mathcal{D}\mathcal{B}=\left[\begin{array}{cc}
q_{11} & q_{12} \\
 q_{21} & q_{22}
\end{array}\right]$.
$r$ is chosen such that
\begin{align}
\left|q_{12}+q_{21}\right|=&O\left(\left|\w\right|+\left|\w_{x}\right|\right), \\
q_{22}=&\frac{1}{2}r\left\{2\frac{\bar{J}+j}{\bar{\rho}+\sigma}\right\}_{x}-\frac{1}{2}r_{x}\left(2\frac{\bar{J}+j}{\bar{\rho}+\sigma}\right) +r \left(1-2\frac{\bar{J}}{\bar{\rho}^{2}}\bar{\rho}_{x}\right)\\
=& r-\frac{\bar{J}}{\bar{\rho}}r_{x}-3\frac{\bar{J}}{\bar{\rho}^{2}}\bar{\rho}_{x}r+O\left(|\w|+|\w_{x}|\right).\notag
\end{align}
\eqref{r satisfies in first order estimates} admits a positive solution $r(x)$.
Thus, for small $\bar{J}$, there is a constant $\zeta_{3}>0$, such that
\begin{align}
\mathcal{Q}\w_{x}\cdot\w_{x}\geqslant \zeta_{3}j_{x}^{2} + O\left(\left|\w\right|+\left|\w_{x}\right|\right)\left|\w_{x}\right|^{2}.
\end{align}
For the terms containing $\w\cdot \w_{x}$, we estimate
\begin{align}
\left(-\mathcal{D}_{x}\mathcal{B}+\bar{\mathcal{B}}_{x}\right)\w \cdot \w_{x}=\mathcal{D}\mathcal{B}_{x} \w \cdot \w_{x} = O\left(\left|\w\right|\left|\w_{x}\right|\right)\ls \vartheta_{3} \left|\w_{x}\right|^{2}+C_{\vartheta_{3}} \left|\w\right|^{2},
\end{align}
where $\vartheta_{3}$ is small enough, balanced by the good term $\zeta_{4}j_{x}^{2}$ and $\zeta_{4}\sigma_{x}^{2}$ in the later estimates.
For terms concerning $\w_{x}$, we have
\begin{align}
 & \left(-\mathcal{D}_{x}\mathcal{C}+\bar{\mathcal{C}}_{x}\right) \cdot \w_{x}\d t= \mathcal{D}\mathcal{C}_{x}\cdot \w_{x}\d t= - r\left(\bar{\rho} \tilde{e} \right)_{x}j_{x}\d t= r\left(\bar{\rho}\sigma+\bar{\rho}_{x}\tilde{e}\right)\sigma_{t}\d t \\
=& \d \left(\frac{1}{2}r\bar{\rho}\sigma^{2} + r\bar{\rho}_{x}\tilde{e}\sigma \right)-r\bar{\rho}_{x}\tilde{e}_{t}\sigma\d t
= \d \left(\frac{1}{2}r\bar{\rho}\sigma^{2} + r\bar{\rho}_{x}\tilde{e}\sigma \right) + O\left(\sigma^{2}+j^{2}\right)\d t, \notag
\end{align}
and
\begin{align}
\bar{\mathcal{N}}_{x}\w_{x}=O\left(\left|\w\right|+\left|\tilde{e}\right|\right)\left(\left|\w\right|^{2}+\left|\w_{x}\right|^{2}\right).
\end{align}
We estimate
\begin{align}
\int_{0}^{1}  \frac{r^{-1}\left|r\mathbb{F}_{x}+r_{x}\mathbb{F}\right|^{2}}{2}\d x \d t\ls C\int_{0}^{1}\left(\left|\w\right|^{2}+\left|\w_{x}\right|^{2}\right)^{2}\d x \d t,
\end{align}
\begin{align}
\int_{0}^{1}\frac{r^{-1}\left|r_{x}\mathbb{F}\right|^{2}}{2}\d x \d t \ls C\int_{0}^{1}\left|\w\right|^{3}\d x \d t.
\end{align}
For the stochastic term, with the assumption that $|\bar{J}|^{3}\ls \mathbb{E}\left[\left\|\w_{0} \right\|^{2}\right]$ is sufficiently small, we have
\begin{align}
& \mathbb{E} \left[\left| \int_{0}^{t}\int_{0}^{1} r_{x} \mathbb{F} j_{x} \d x \d  W \right|^{m}\right] \notag\\
\ls & \mathbb{E} \left[ \left(C\int_{0}^{t}\left|\int_{0}^{1} r_{x} \mathbb{F} j_{x} \d x \right|^{2} \d s \right)^{\frac{m}{2}} \right]
= \mathbb{E} \left[\left(C\int_{0}^{t}\left|\int_{0}^{1} r_{x} \left|\bar{J}+j\right|^{2} j_{x} \d x \right|^{2} \d s \right)^{\frac{m}{2}} \right]\notag\\
= &\mathbb{E} \left[\left(C\int_{0}^{t}\left|\int_{0}^{1} r_{x} \left(\bar{J}^{2}+2\bar{J}j+j^{2}\right) j_{x} \d x \right|^{2} \d s \right)^{\frac{m}{2}} \right] \\
\ls & \mathbb{E} \left[\left\|\w_{0} \right\|^{2m} \right]+\mathbb{E} \left[\left( C \int_{0}^{t} \left|\bar{J}\right| \left\|\w\right\|_{1}^{2} \d s \right)^{m}\right]\notag\\
  &+ \mathbb{E} \left[ \left( \frac{\vartheta_{4}}{4}\sup\limits_{s\in [0,t]}\int_{0}^{1}\left|j\right|^{2}\d x\right)^{m} \right] +\mathbb{E} \left[\left(C_{\vartheta_{4}} \int_{0}^{t} \left\|\w\right\|_{1}^{3} \d s \right)^{m}\right],\notag
\end{align}
where $\vartheta_{4}$ is taken that $\mathbb{E} \left[\left(\vartheta_{4} \sup\limits_{s\in [0,t]}\int_{0}^{1}\left|j\right|^{2}\d x\right)^{m} \right]$ can be balanced by the zeroth-order estimates and
 $\mathbb{E}\left[ \left(\frac{\vartheta_{4}}{4} \sup\limits_{s\in [0,t]}\int_{0}^{1}\left|j_{x}\right|^{2}\d x\right)^{m}\right]$ can be balanced by the left hand-side of this first-order estimates.
Combining the above terms, we obtain
\begin{align}
& \mathbb{E} \left[\left| \sup\limits_{s\in[0,t]} \int_{0}^{1} \left(j_{x}^{2}+\sigma_{x}^{2}\right)(s)\d x +\zeta_{5} \int_{0}^{t}\int_{0}^{1}j_{x}^{2} \d x \d s  \right|^{m}\right] \notag\\
\ls &  \mathbb{E} \left[\left| \int_{0}^{1}\left(\frac{1}{2}\bar{\rho}\sigma^{2}+ r\bar{\rho}_{x}\tilde{e}\sigma \right)\d x+C_{3}\int_{0}^{t}\int_{0}^{1}\left(\left|\w\right|^{2}+\left(\left|\w_{x}\right|+\left|\w\right|+\left|\tilde{e}\right|\right)^{3}\right)\d x \d s \right|^{m}\right] \\
\ls & \mathbb{E} \left[\left\|\w_{0} \right\|^{2m} \right] + \mathbb{E} \left[\left| \int_{0}^{1}\left(\frac{1}{2}\bar{\rho}\sigma^{2}+ \bar{\rho}_{x}\tilde{e}\sigma \right)\d x+C_{4}\int_{0}^{t} \left|\bar{J}\right|\left\|\w\right\|_{2}^{2} \d s  +C_{4}\int_{0}^{t}\left\|\w\right\|_{2}^{3}\d s \right|^{m}\right]. \notag
\end{align}
Next we give the estimate of $\int_{0}^{t}\int_{0}^{1}\sigma_{x}^{2} \d x \d s$ by multiplying \eqref{1 order derivative of symmetrized system 1} with $\left[\begin{array}{c}0  \\ \sigma\end{array}\right]$. By It\^o's formula, we know
\begin{align}
\mathcal{D}\d \left(\w_{x}\cdot\left[\begin{array}{c}0  \\ \sigma\end{array}\right]\right)= \mathcal{D}\d \w_{x}\cdot\left[\begin{array}{c}0  \\ \sigma\end{array}\right]+\mathcal{D}\w_{x} \cdot\left[\begin{array}{c}0  \\ -j_{x}\end{array}\right]\d t.
\end{align}
Then, we have
\begin{align}\label{equality for estimate of sigma x}
&-\mathcal{D}\d  \left(\w_{x}\cdot\left[\begin{array}{c}0  \\ \sigma\end{array}\right]\right) + \mathcal{D} \w_{x} \cdot\left[\begin{array}{c}0  \\ -j_{x}\end{array}\right]\d t + \left(\mathcal{D}_{x} \left(\mathcal{A}\w_{x}+\mathcal{B}\w + \mathcal{C}- \mathcal{N}\right)\right)\cdot\left[\begin{array}{c}0  \\ \sigma\end{array}\right]\d t \notag\\
 &-\left(\bar{\mathcal{A}}\w_{xx}+ \left(\bar{\mathcal{A}}_{x}+\bar{\mathcal{B}}\right)\w_{x}+\bar{\mathcal{B}}_{x}\w +\bar{\mathcal{C}}\right)\cdot\left[\begin{array}{c}0  \\ \sigma\end{array}\right] \d t  \\
=& -\left(\bar{\mathcal{N}}_{x} \d t + \mathcal{D}_{x}\left[\begin{array}{c}0\\\mathbb{F}\d  W\end{array}\right]  + \mathcal{D}\left[\begin{array}{c}0\\ \mathbb{F}_{x} \d  W\end{array}\right]\right)\cdot \left[\begin{array}{c}0  \\ \sigma\end{array}\right],\notag
\end{align}
where the first term in the left hand-side in \ref{equality for estimate of sigma x} holds
\begin{align}
-\mathcal{D}\d  \left(\w_{x}\cdot\left[\begin{array}{c}0  \\ \sigma\end{array}\right]\right)=-\d \left(rj_{x}\sigma\right),
 \end{align}
and the second term in the left hand-side in \ref{equality for estimate of sigma x} holds
 \begin{align}
\mathcal{D}\w_{x} \cdot\left[\begin{array}{c}0  \\ -j_{x}\end{array}\right]\d t=-r j_{x}^{2}\d t.
  \end{align}
We consider the stochastic term.
For $ \mathcal{D}\left[\begin{array}{c}0\\ \mathbb{F}_{x} \d  W\end{array}\right]\cdot \left[\begin{array}{c}0  \\ \sigma\end{array}\right]=r \mathbb{F}_{x}\sigma \d  W$, there holds
\begin{align}
& \mathbb{E} \left[\left| \int_{0}^{t}\int_{0}^{1} r \mathbb{F}_{x}\sigma \d x \d  W \right|^{m}\right] \notag\\
\ls & \mathbb{E} \left[ \left(\int_{0}^{t}\left|\int_{0}^{1} r \mathbb{F}_{x}\sigma  \d x \right|^{2} \d s \right)^{\frac{m}{2}} \right] \\
\ls &  \mathbb{E} \left[\left\|\w_{0} \right\|^{2m} \right]+\mathbb{E} \left[\left( C \int_{0}^{t} \left|\bar{J}\right| \left\|\w\right\|_{1}^{2} \d s \right)^{m}\right] \notag\\
&+\mathbb{E} \left[ \frac{\vartheta_{4}}{4} \left(\sup\limits_{s\in [0,t]}\int_{0}^{1}\left|\sigma\right|^{2}\d x\right)^{m} \right]+C_{\vartheta_{4}}\mathbb{E} \left[\left( \int_{0}^{t} \int_{0}^{1} C\left\|\w\right\|_{2}^{3} \d s \right)^{m}\right].\notag
\end{align}
The estimate of $\mathcal{D}_{x}\left[\begin{array}{c}0\\\mathbb{F}\d  W\end{array}\right]$ is obtained similarly and upper bounded by the same quantity as the above formula.
By substitution, the key term holds
 \begin{align}
  & -\bar{\mathcal{A}}\w_{xx}\cdot\left[\begin{array}{c}0  \\ \sigma \end{array}\right]\d t\notag\\
  = &-\left(\bar{\mathcal{A}}\w_{x}\left[\begin{array}{c} 0 \\ \sigma \end{array}\right]\right)_{x}\d t+ \bar{\mathcal{A}}\w_{x}\cdot\left[\begin{array}{c} 0 \\ \sigma_{x} \end{array}\right]\d t +\bar{\mathcal{A}}_{x} \w_{x}\cdot\left[\begin{array}{c} 0 \\ \sigma \end{array}\right]\d t\\
 = & -\left(\bar{\mathcal{A}}\w_{x}\left[\begin{array}{c} 0 \\ \sigma \end{array}\right]\right)_{x}\d t +  \left( -\left(\frac{\bar{J}+j}{\bar{\rho}+\sigma}\right)^{2}+P'\left(\bar{\rho}+\sigma\right)\right)r \sigma_{x}^{2}\d t+  2\frac{\bar{J}+j}{\bar{\rho}+\sigma}rj_{x}\sigma_{x}\d t\notag\\
 & +\left(\left( -\left(\frac{\bar{J}+j}{\bar{\rho}+\sigma}\right)^{2}+P'\left(\bar{\rho}+\sigma\right)\right)r\right)_{x}\sigma_{x}\sigma\d t + 2\left(\frac{\bar{J}+j}{\bar{\rho}+\sigma}r\right)_{x} j_{x}\sigma\d t. \notag
 \end{align}
 We integrate the above formula in $[0,1]$ with respect to $x$. Then, the first term holds
  \begin{align}
  -\int_{0}^{1}\left(\bar{\mathcal{A}}\w_{x}\left[\begin{array}{c} 0 \\ \sigma \end{array}\right]\right)_{x}\d x \d t =0;
   \end{align}
 the second term is at leat $\zeta_{4}\sigma_{x}^{2}$;
 the other three terms are bounded by
 \begin{align}
 C\left(\bar{J}+j\right)\w_{x}^{2}+C\left| \w_{x}\cdot \w \right|\ls  C\left(\bar{J}+j\right)\w_{x}^{2}+\vartheta_{5}\w_{x}^{2}+C_{\vartheta_{5}}\w^{2},\quad \vartheta_{5}\ls \min\left\{\frac{\zeta_{5}}{4},\frac{\zeta_{4}}{4}\right\}.
\end{align}
To summarize, we have
\begin{align}
 &\mathbb{E} \left[\left|\int_{0}^{t}\int_{0}^{1}\zeta_{5}\sigma_{x}^{2}\d x \d s \right|^{m}\right]\notag\\
  \ls & \mathbb{E} \left[\left|\int_{0}^{t}\int_{0}^{1} C\left(\bar{J}+j\right)\w_{x}^{2}+C_{\vartheta_{5}}\w^{2}+C|j_{x}|^{2}\d x \d s \right|^{m}\right]+ \vartheta_{4}^{m} \mathbb{E}\left[ \left(\sup\limits_{s\in [0,t]}\int_{0}^{1}\left|\sigma\right|^{2}\d x\right)^{m} \right]\notag\\
 &+ \mathbb{E} \left[\left( \int_{0}^{t} \int_{0}^{1} C_{\vartheta_{4}}\left\|\w\right\|_{2}^{3} \d s \right)^{m}\right]+ \vartheta_{5}^{m}\mathbb{E} \left[\left|\int_{0}^{t}\int_{0}^{1} \w_{x}^{2}\d x \d s \right|^{m}\right]-\mathbb{E} \left[\left|\int_{0}^{t}\d \int_{0}^{1}rj_{x}\sigma\d x\right|^{m}\right]\notag\\
\ls & \vartheta_{4}^{m} \mathbb{E}\left[ \left(\sup\limits_{s\in [0,t]}\int_{0}^{1}\left|\sigma\right|^{2}\d x\right)^{m} \right]+ \mathbb{E} \left[\left|\int_{0}^{t}\int_{0}^{1} C_{\vartheta_{5}}\w^{2}\d x \d s \right|^{m}\right]\\
 &+\vartheta_{5}^{m}\mathbb{E} \left[\left|\int_{0}^{t}\int_{0}^{1} \w_{x}^{2}\d x \d s \right|^{m}\right]+ \mathbb{E} \left[\left( \int_{0}^{t} \int_{0}^{1} C\left|\bar{J}\right|\left\|\w\right\|_{2}^{2} \d s \right)^{m}\right]\notag\\
&+ \mathbb{E} \left[\left( \int_{0}^{t} \int_{0}^{1} C_{\vartheta_{4}}\left\|\w\right\|_{2}^{3} \d s \right)^{m}\right]
-\mathbb{E} \left[\left|\int_{0}^{t}\d \int_{0}^{1}rj_{x}\sigma\d x\right|^{m}\right].\notag
\end{align}
Multiplying the above formula with some small constants such that the coefficent constant is less than $\frac{\zeta_{5}}{4}$, $\zeta_{6}=\min\left\{\frac{\zeta_{5}}{2},\frac{\zeta_{4}}{2}\right\}$, we obtain
  \begin{align}
&\mathbb{E} \left[\left| \int_{0}^{1} \left(j_{x}^{2}+\sigma_{x}^{2}\right)(t)\d x \right|^{m}\right] + \mathbb{E} \left[\left|\zeta_{6} \int_{0}^{t}\int_{0}^{1}\left(j_{x}^{2} +\sigma_{x}^{2}\right) \d x \d s  \right|^{m}\right]\notag\\
\ls & \mathbb{E} \left[\left| \int_{0}^{1}\left(\frac{1}{2}\bar{\rho}\sigma^{2}+ r\bar{\rho}_{x}\tilde{e}\sigma -rj_{x}\sigma\right)\d x \right|^{m}\right]+\vartheta_{4}^{m} \mathbb{E}\left[ \left(\sup\limits_{s\in [0,t]}\int_{0}^{1}\left|\w\right|^{2}\d x\right)^{m} \right]\\
&+ \mathbb{E} \left[\left\|\w_{0} \right\|^{2m} \right] +\mathbb{E} \left[\left|C_{4}\int_{0}^{t} \left|\bar{J}\right|\left\|\w\right\|_{2}^{2} \d s +C_{4}\int_{0}^{t}\left\|\w\right\|_{2}^{3}\d s\right|^{m}\right]\notag\\
&+ \mathbb{E} \left[\left|\int_{0}^{t}\int_{0}^{1} C_{\vartheta_{5}} \w^{2}\d x \d t \right|^{m}\right],\notag
\end{align}
 Together with the zeroth-order estimates \eqref{conclusion of zero order estimate}, since $\sup\limits_{s\in[0,t]}\left| \int_{0}^{1}\left(\frac{1}{2}\bar{\rho}\sigma^{2}+ r\bar{\rho}_{x}\tilde{e}\sigma -rj_{x}\sigma\right)\d x \right|\ls \sup\limits_{s\in[0,t]}\vartheta_{6}\left\|\w_{x}\right\|^{2}+\sup\limits_{s\in[0,t]}C_{\vartheta_{6}}\left\|\w\right\|^{2}$, we have
\begin{align}
&\mathbb{E} \left[\left| \sup\limits_{s\in[0,t]} \int_{0}^{1} \left(j_{x}^{2}+\sigma_{x}^{2}\right)(s)\d x \right|^{m}\right] + \mathbb{E} \left[\left|\zeta_{7} \int_{0}^{t}\int_{0}^{1}\left(j_{x}^{2} +\sigma_{x}^{2}\right) \d x \d s  \right|^{m}\right] \notag \\
\ls & \mathbb{E} \left[\left|  \int_{0}^{1} \left(j_{x}^{2}+\sigma_{x}^{2}\right)(0)\d x \right|^{m}\right]+ C \mathbb{E} \left[\left|\int_{0}^{1} \left|\w_{0}\right|^{2}\d x \right|^{m}\right]\\
&+\mathbb{E} \left[\left|C_{5}\int_{0}^{t} \left|\bar{J}\right|\left\|\w\right\|_{2}^{2} \d s +C_{5}\int_{0}^{t}\left\|\w\right\|_{2}^{3}\d s \right|^{m}\right]. \notag
\end{align}

\subsection{Second-order estimates on $\w$}\label{second order estimates subsec 1-D}

For the second-order estimates, we use a different symmetrizing matrix, which involves $t$.
Let $\tilde{\mathcal{D}}=\left[\begin{array}{cc}\tilde{s} &0 \\ 0& \tilde{r}\end{array}\right]$, $\tilde{s}=\left(P'\left(\bar{\rho}+\sigma\right)-\frac{\left(\bar{J}+j\right)^{2}}{\left(\bar{\rho}+\sigma\right)^{2}}\right)\tilde{r}$, where $ \tilde{r}$ is determined later. Then the system becomes
\begin{align}\label{2order estim symmetrized sto Euler-Poisson system contact boundary}
 \tilde{\mathcal{D}}\d  \w+\tilde{\mathcal{D}}\left(\mathcal{A}\w_{x}+\mathcal{B}\w + \mathcal{C}\right) \d t = \tilde{\mathcal{D}} \mathcal{N} \d t + \tilde{\mathcal{D}} \left[\begin{array}{c}0  \\ \mathbb{F}\d  W\end{array}\right].
\end{align}
Differentiating \eqref{2order estim symmetrized sto Euler-Poisson system contact boundary} with respect to $x$, we have
\begin{align}\label{1-Deriv 2order estim symmetrized sto Euler-Poisson system contact boundary}
 &\tilde{\mathcal{D}} \d  \w_{x} + \tilde{\mathcal{D}}_{x} \d  \w + \left(\tilde{\mathcal{A}}\w_{xx}+ \left(\tilde{\mathcal{A}}_{x}+\tilde{\mathcal{B}}\right)\w_{x}+\tilde{\mathcal{B}}_{x}\w + \tilde{\mathcal{C}}_{x}\right) \d t \notag \\
=& \tilde{\mathcal{N}}_{x} \d t + \tilde{\mathcal{D}}_{x}\left[\begin{array}{c}0  \\ \mathbb{F}\d  W\end{array}\right] + \tilde{\mathcal{D}} \left[\begin{array}{c}0  \\ \mathbb{F}_{x}\d  W\end{array}\right],
\end{align}
where $\tilde{\mathcal{A}}=\tilde{\mathcal{D}}\mathcal{A}$, $\tilde{\mathcal{B}}=\tilde{\mathcal{D}}\mathcal{B}$, $\tilde{\mathcal{C}}=\tilde{\mathcal{D}}\mathcal{C}$, and $\tilde{\mathcal{N}}=\tilde{\mathcal{D}}\mathcal{N}$.
Differentiating \eqref{1-Deriv 2order estim symmetrized sto Euler-Poisson system contact boundary} with respect to $x$, we get
\begin{align}\label{2-deriv 2order estim symmetrized sto Euler-Poisson system contact boundary}
 & \tilde{\mathcal{D}} \d  \w_{xx} + 2 \tilde{\mathcal{D}}_{x} \d  \w_{x} + \tilde{\mathcal{D}}_{xx} \d  \w \notag\\
 &+ \left(\tilde{\mathcal{A}}_{x}\w_{xx}+\tilde{\mathcal{A}}\w_{xxx}+ \left(\tilde{\mathcal{A}}_{xx}+\tilde{\mathcal{B}}_{x}\right)\w_{x}+\left(\tilde{\mathcal{A}}_{x}+\tilde{\mathcal{B}}\right)\w_{xx}+\tilde{\mathcal{B}}_{xx}\w + \tilde{\mathcal{B}}_{x}\w_{x} +\tilde{\mathcal{C}}_{xx}\right) \d t \notag \\
=& N_{xx} \d t + \tilde{\mathcal{D}}_{xx} \left[\begin{array}{c}0  \\ \mathbb{F}\d  W\end{array}\right] + 2\tilde{\mathcal{D}}_{x} \left[\begin{array}{c}0  \\ \mathbb{F}_{x}\d  W\end{array}\right]+ \tilde{\mathcal{D}}\left[\begin{array}{c}0  \\ \mathbb{F}_{xx}\d  W\end{array}\right].
\end{align}
Substituting \eqref{2order estim symmetrized sto Euler-Poisson system contact boundary} and \eqref{1-Deriv 2order estim symmetrized sto Euler-Poisson system contact boundary} into \eqref{2-deriv 2order estim symmetrized sto Euler-Poisson system contact boundary}, we have
\begin{align}\label{substituted 2order estim symmetrized sto Euler-Poisson system contact boundary}
 & \tilde{\mathcal{D}} \d  \w_{xx} + 2 \tilde{\mathcal{D}}_{x}\tilde{\mathcal{D}}^{-1}\tilde{\mathcal{D}}_{x}\left(\mathcal{A}\w_{x}+\mathcal{B}\w + \mathcal{C}\right) \d t-2\tilde{\mathcal{D}}_{x}\tilde{\mathcal{D}}^{-1}\tilde{\mathcal{D}}_{x}\left(\mathcal{N} \d t + \left[\begin{array}{c}0  \\ \mathbb{F}\d  W\end{array}\right]\right) \notag\\
 &-2\tilde{\mathcal{D}}_{x}\tilde{\mathcal{D}}^{-1}\tilde{\mathcal{A}}\w_{xx}\d t-2\tilde{\mathcal{D}}_{x}\tilde{\mathcal{D}}^{-1}\left(\left(\tilde{\mathcal{A}}_{x}+\tilde{\mathcal{B}}\right)\w_{x}+\tilde{\mathcal{B}}_{x}\w + \tilde{\mathcal{C}}_{x}\right) \d t+ \tilde{\mathcal{N}}_{x} \d t \notag\\
& + 2\tilde{\mathcal{D}}_{x}\tilde{\mathcal{D}}^{-1}\left( \tilde{\mathcal{D}}_{x}\left[\begin{array}{c}0  \\ \mathbb{F}\d  W\end{array}\right] + \tilde{\mathcal{D}} \left[\begin{array}{c}0  \\ \mathbb{F}_{x}\d  W\end{array}\right]\right)\\
 &+ \tilde{\mathcal{D}}_{xx} \left(-\mathcal{A}\w_{x}-\mathcal{B}\w - \mathcal{C} + \mathcal{N}\right) \d t + \tilde{\mathcal{D}}_{xx}\left[\begin{array}{c}0  \\ \mathbb{F}\d  W\end{array}\right]  \notag \\
 &+\left(\tilde{\mathcal{A}}_{x}\w_{xx}+\tilde{\mathcal{A}}\w_{xxx} + \left(\tilde{\mathcal{A}}_{xx}+\tilde{\mathcal{B}}_{x}\right)\w_{x} +\left(\tilde{\mathcal{A}}_{x}+\tilde{\mathcal{B}}\right)\w_{xx}+\tilde{\mathcal{B}}_{xx}\w + \tilde{\mathcal{B}}_{x}\w_{x} +\tilde{\mathcal{C}}_{xx}\right) \d t \notag \\
=& \tilde{\mathcal{N}}_{xx} \d t  + \tilde{\mathcal{D}}_{xx} \left[\begin{array}{c}0  \\ \mathbb{F}\d  W\end{array}\right] + 2\tilde{\mathcal{D}}_{x} \left[\begin{array}{c}0  \\ \mathbb{F}_{x}\d  W\end{array}\right]+ \tilde{\mathcal{D}}\left[\begin{array}{c}0  \\ \mathbb{F}_{xx}\d  W\end{array}\right]. \notag
\end{align}
Since
\begin{align}
 &\d \left(\tilde{\mathcal{D}}\w_{xx}\right)= \w_{xx} \d \tilde{\mathcal{D}} + \tilde{\mathcal{D}} \d  \w_{xx} + \langle \d  \tilde{\mathcal{D}}, \d  \w_{xx} \rangle \\
= &\left[\begin{array}{c} \left(\d  \tilde{s}\right)\sigma_{xx}\\ 0 \end{array}\right] + \tilde{\mathcal{D}} \d  \w_{xx}+\frac{-2\left(\bar{J}+j\right)}{\left(\bar{\rho}+\sigma\right)^{2}}\mathbb{F}\mathbb{F}_{x}\d t,\notag
 \end{align}
by It\^o's formula, it holds that
 \begin{align}\label{2 order balance}
& \d \left(\tilde{\mathcal{D}}\w_{xx}\cdot \w_{xx}\right) =  \w_{xx}\d \left(\tilde{\mathcal{D}} \w_{xx}\right)+2\tilde{\mathcal{D}} \w_{xx} \d  \w_{xx}+\left\langle \d  \left(\tilde{\mathcal{D}} \w_{xx}\right), \d  \w_{xx} \right\rangle\notag\\
= & \d \tilde{\mathcal{D}} \left|\w_{xx}\right|^{2}+ 2 \w_{xx}\tilde{\mathcal{D}} \d  \w_{xx}+ \frac{-2\left(\bar{J}+j\right)}{\left(\bar{\rho}+\sigma\right)^{2}}\mathbb{F}\cdot\mathbb{F}_{x}\d t + \left\langle \d  \left(\tilde{\mathcal{D}} \w_{xx}\right), \d  \w_{xx}\right\rangle \notag\\
=& \d \tilde{\mathcal{D}} \left|\w_{xx}\right|^{2} + 2 \w_{xx}\tilde{\mathcal{D}} \d  \w_{xx}+ \frac{-2\left(\bar{J}+j\right)}{\left(\bar{\rho}+\sigma\right)^{2}}\mathbb{F}\cdot\mathbb{F}_{x}\d t \notag\\
&+\tilde{r}^{-1}\left\langle 2\tilde{r}_{x}\tilde{r}^{-1}\left(\tilde{r}_{x}\mathbb{F}+\tilde{r}\mathbb{F}_{x} \right)\d  W-2\tilde{r}_{x}\mathbb{F}_{x}\d  W -\tilde{r}\mathbb{F}_{xx}\d  W \right\rangle \\
=& \d \tilde{\mathcal{D}} \left|\w_{xx}\right|^{2} \notag\\
&+ 2 \w_{xx}\left(- 2 \tilde{\mathcal{D}}_{x}\tilde{\mathcal{D}}^{-1}\tilde{\mathcal{D}}_{x}\left(\mathcal{A}\w_{x}+\mathcal{B}\w + \mathcal{C}\right) \d t+2\tilde{\mathcal{D}}_{x}\tilde{\mathcal{D}}^{-1}\left(\tilde{\mathcal{D}}_{x}\mathcal{N} \d t + \tilde{\mathcal{D}}\left[\begin{array}{c}0  \\ \mathbb{F}\d  W\end{array}\right]\right)\right. \notag\\
 &\left.\qquad +2\tilde{\mathcal{D}}_{x}\tilde{\mathcal{D}}^{-1}\tilde{\mathcal{A}}\w_{xx}\d t+2\tilde{\mathcal{D}}_{x}\tilde{\mathcal{D}}^{-1}\left(\left(\tilde{\mathcal{A}}_{x}+\tilde{\mathcal{B}}\right)\w_{x}+\tilde{\mathcal{B}}_{x}\w + \tilde{\mathcal{C}}_{x}\right) \d t- \tilde{\mathcal{N}}_{x} \d t \right.\notag\\
&\left. \qquad - 2\tilde{r}_{x}\tilde{r}^{-1}\left(\tilde{r}_{x}\mathbb{F}\d  W + \tilde{r} \mathbb{F}_{x}\d  W\right) - \tilde{\mathcal{D}}_{xx} \left(-\mathcal{A}\w_{x}-\mathcal{B}\w - \mathcal{C} + \mathcal{N}\right) \d t  \right.\notag \\
 &\left.\qquad -\left(\tilde{\mathcal{A}}_{x}\w_{xx}+\tilde{\mathcal{A}}\w_{xxx} + \left(\tilde{\mathcal{A}}_{xx}+\tilde{\mathcal{B}}_{x}\right)\w_{x} +\left(\tilde{\mathcal{A}}_{x}+\tilde{\mathcal{B}}\right)\w_{xx}+\tilde{\mathcal{B}}_{xx}\w + \tilde{\mathcal{B}}_{x}\w_{x} +\tilde{\mathcal{C}}_{xx}\right) \d t \right.\notag \\
& \left.\qquad +\tilde{\mathcal{N}}_{xx} \d t + \tilde{r}_{xx} \mathbb{F}\d  W + 2\tilde{r}_{x} \mathbb{F}_{x}\d  W  + 2\tilde{r} \mathbb{F}_{xx} \d  W\right)\notag\\
& + \frac{-2\left(\bar{J}+j\right)}{\left(\bar{\rho}+\sigma\right)^{2}}\mathbb{F}\cdot\mathbb{F}_{x}\d t +\tilde{r}^{-1}\left\langle 2\tilde{r}_{x}\tilde{r}^{-1}\left(\tilde{r}_{x}\mathbb{F}+\tilde{r}\mathbb{F}_{x} \right)\d  W-2\tilde{r}_{x}\mathbb{F}_{x}\d  W -\tilde{r}\mathbb{F}_{xx}\d  W \right\rangle. \notag
\end{align}
To estimate the terms in form of $\int_{0}^{1} (\cdot)\left|\w_{xx}\right|^{2}\d x $ suffices to consider $\int_{0}^{1}\w_{xx} \tilde{\mathcal{Q}} \w_{xx}\d x $, where  $\tilde{\mathcal{Q}}=-4\tilde{\mathcal{D}}_{x}\tilde{\mathcal{D}}_{-1}\tilde{\mathcal{A}}+4\tilde{\mathcal{A}}_{x}+2\tilde{\mathcal{B}}+\tilde{\mathcal{A}}_{x}$. Actually, we give the reasons as the followings. The identity
\begin{align}
\int_{0}^{1}\left(\frac{1}{2}\tilde{\mathcal{A}}_{x}\w_{xx}+\tilde{\mathcal{A}}\w_{xxx}\right)\w_{xx}\d x = \left.\frac{1}{2}\tilde{\mathcal{A}}\w_{xx}\cdot \w_{xx} \right|_{0}^{1},
\end{align}
holds because
\begin{align}
\tilde{\mathcal{A}}=\tilde{\mathcal{D}}\mathcal{A}=\left[\begin{array}{cc} 0 &\tilde{s} \\ \tilde{s} &\tilde{r}\frac{2\left(\bar{J}+j\right)}{\bar{\rho}+\sigma}\end{array}\right]
\end{align}
is symmetric. We claim that
\begin{align}\label{important claim on boundary}
\left.\tilde{\mathcal{A}}\w_{xx}\cdot \w_{xx} \d t \right|_{0}^{1}=0.
\end{align}
Actually, we use two different ways to represent $\left.\frac{1}{2}\tilde{\mathcal{A}}\w_{xx}\cdot \w_{xx} \right|_{0}^{1}$. We directly calculate
\begin{align}\label{the one hand}
&\left.\tilde{\mathcal{A}}\w_{xx}\cdot \w_{xx} \right|_{0}^{1} = \left.2 \tilde{s}\sigma_{xx} j_{xx}\right|_{0}^{1}+ \left.\tilde{r}\frac{2\left(\bar{J}+j\right)}{\bar{\rho}+\sigma}j_{xx}^{2}\right|_{0}^{1}\\
=& 2 \tilde{s}\sigma_{xx} j_{xx}(1)-2 \tilde{s}\sigma_{xx} j_{xx}(0)+\tilde{r}\frac{2\left(\bar{J}+j\right)}{\bar{\rho}+\sigma}j_{xx}^{2}(1)-\tilde{r}\frac{2\left(\bar{J}+j\right)}{\bar{\rho}+\sigma}j_{xx}^{2}(0).\notag
\end{align}
On the one hand, from the system \eqref{2-deriv 2order estim symmetrized sto Euler-Poisson system contact boundary}, there hold
\begin{align}\label{formula for barbarA wxx}
&\left.\tilde{\mathcal{A}}\w_{xx}\d t\right|_{0}^{1} \notag\\
=&\left.-\tilde{\mathcal{D}}\d \w_{x}\right|_{0}^{1}+ \left.\left(\tilde{\mathcal{D}}_{x}\left(\mathcal{A}\w_{x}+\mathcal{B}\w +\mathcal{C} -\mathcal{N}\right)-\left(\tilde{\mathcal{A}}_{x}+\tilde{\mathcal{B}}\right)\w_{x}-\tilde{\mathcal{C}}_{x}+\tilde{\mathcal{N}}_{x}\right)\d t\right|_{0}^{1}\\
&+\left.\left( \left[\begin{array}{c} 0\\\tilde{r}_{x} \mathbb{F}\d  W\end{array}\right] + \left[\begin{array}{c} 0\\\tilde{r} \mathbb{F}_{x} \d  W
\end{array}\right]\right)\right|_{0}^{1}, \notag
\end{align}
and
\begin{align}\label{the other hand}
&\left.\tilde{\mathcal{A}}\w_{xx}\cdot\w_{xx}\d t\right|_{0}^{1}\notag\\
=& \left.\left(\tilde{\mathcal{D}}_{x}\left(\mathcal{A}\w_{x}+\mathcal{B}\w +\mathcal{C} -\mathcal{N}\right)-\left(\tilde{\mathcal{A}}_{x}+\tilde{\mathcal{B}}\right)\w_{x}-\tilde{\mathcal{C}}_{x}+\tilde{\mathcal{N}}_{x}\right)\cdot\w_{xx}\d t\right|_{0}^{1}\\
&\left.-\tilde{\mathcal{D}}\d \w_{x}\cdot\w_{xx}\right|_{0}^{1}+\left.\left( \left[\begin{array}{c} 0\\\tilde{r}_{x} \mathbb{F}\d  W\end{array}\right] + \left[\begin{array}{c} 0\\\tilde{r} \mathbb{F}_{x} \d  W
\end{array}\right]\cdot\w_{xx}\right)\right|_{0}^{1}.\notag
\end{align}
We deal with the right-hand side in the above formula term by term.
Fortunately, since $\left.\d j_{x}\right|_{0}^{1}=\left.\d \sigma_{t}\right|_{0}^{1}=0$, we have
\begin{align}
\left.-\tilde{\mathcal{D}}\d \w_{x}\cdot\w_{xx}\right|_{0}^{1}=-\left.\tilde{s}\d \sigma_{x}\sigma_{xx}\right|_{0}^{1}-\left.\tilde{r}\d j_{x}j_{xx}\right|_{0}^{1}
=-\left.\tilde{s}\d \sigma_{x}\sigma_{xx}\right|_{0}^{1}=\left.\tilde{s}j_{xx}\sigma_{xx}\d t\right|_{0}^{1}.
\end{align}
By direct calculation, we have
\begin{align}
&\left.\left(\tilde{\mathcal{D}}_{x}\mathcal{A}\w_{x}-\tilde{\mathcal{A}}_{x}\w_{x}\right)\cdot\w_{xx}\d t\right|_{0}^{1}
=-\left.\tilde{\mathcal{D}}\mathcal{A}_{x}\w_{x}\cdot\w_{xx}\d t\right|_{0}^{1}\notag\\
=&-\left. \tilde{\mathcal{D}}\mathcal{A}_{x}\w_{x}\cdot \tilde{\mathcal{A}}^{-1}\left(-\tilde{\mathcal{D}}\d \w_{x}\right)\right|_{0}^{1}\\
 &-\left.\tilde{\mathcal{D}}\mathcal{A}_{x}\w_{x}\cdot \tilde{\mathcal{A}}^{-1}\left(\tilde{\mathcal{D}}_{x}\left(\mathcal{A}\w_{x}+\mathcal{B}\w +\mathcal{C} -\mathcal{N}\right)-\left(\tilde{\mathcal{A}}_{x}+\tilde{\mathcal{B}}\right)\w_{x}-\tilde{\mathcal{C}}_{x}+\tilde{\mathcal{N}}_{x}\right)\d t \right|_{0}^{1}\notag\\
 &\quad \left.-\tilde{\mathcal{D}}\mathcal{A}_{x}\w_{x}\cdot \tilde{\mathcal{A}}^{-1}\left( \left[\begin{array}{c} 0\\\tilde{r}_{x} \mathbb{F}\d  W\end{array}\right] + \left[\begin{array}{c} 0\\\tilde{r} \mathbb{F}_{x} \d  W
\end{array}\right]\right)\right|_{0}^{1},\notag
\end{align}
By direct calculation, since $\left.\d j_{x}\right|_{0}^{1} = \left. \d \sigma_{t}\right|_{0}^{1}=0$, we have
\begin{align}
-\left. \tilde{\mathcal{D}}\mathcal{A}_{x}\w_{x}\cdot \tilde{\mathcal{A}}^{-1}\left(-\tilde{\mathcal{D}}\d \w_{x}\right)\right|_{0}^{1} = 0.
\end{align}
We calculate
\begin{align}
&-\left.\tilde{\mathcal{D}}\mathcal{A}_{x}\w_{x}\cdot \tilde{\mathcal{A}}^{-1}\left(\tilde{\mathcal{D}}_{x}\left(\mathcal{A}\w_{x}+\mathcal{B}\w +\mathcal{C} -\mathcal{N}\right)-\left(\tilde{\mathcal{A}}_{x}+\tilde{\mathcal{B}}\right)\w_{x}-\tilde{\mathcal{C}}_{x}+\tilde{\mathcal{N}}_{x}\right)\d t \right|_{0}^{1}\notag\\
=&\left.\tilde{\mathcal{D}}\mathcal{A}_{x}\w_{x}\cdot \left(\mathcal{A}^{-1}\mathcal{A}_{x}\w_{x}\right)\d t\right|_{0}^{1}+\left.\tilde{\mathcal{D}}\mathcal{A}_{x}\w_{x}\cdot \left(\tilde{A}^{-1}\tilde{\mathcal{B}}\w_{x}\right)\d t\right|_{0}^{1}\\
&+\left.\tilde{\mathcal{D}}\mathcal{A}_{x}\w_{x}\cdot \tilde{\mathcal{A}}^{-1}\left(\tilde{\mathcal{D}}_{x}\left(\mathcal{B}\w +\mathcal{C} -\mathcal{N}\right)-\tilde{\mathcal{C}}_{x}+\tilde{\mathcal{N}}_{x}\right)\d t\right|_{0}^{1},\notag
\end{align}
\begin{align}
\left.\tilde{\mathcal{D}}\mathcal{A}_{x}\w_{x}\cdot \left(\mathcal{A}^{-1}\mathcal{A}_{x}\w_{x}\right)\d t\right|_{0}^{1}=0,
\end{align}
\begin{align}
\left.\tilde{\mathcal{D}}\mathcal{A}_{x}\w_{x}\cdot \left(\mathcal{A}^{-1}\tilde{\mathcal{B}}\w_{x}\right)\d t\right|_{0}^{1}=0,
\end{align}
and
\begin{align}
&-\left.\tilde{\mathcal{D}}\mathcal{A}_{x}\w_{x}\cdot \tilde{\mathcal{A}}^{-1}\left(\tilde{\mathcal{D}}_{x}\left(\mathcal{B}\w +\mathcal{C} -\mathcal{N}\right)-\tilde{\mathcal{C}}_{x}+\tilde{\mathcal{N}}_{x}\right)\d t\right|_{0}^{1}\\
=&-\left.\tilde{\mathcal{D}}\mathcal{A}_{x}\w_{x}\cdot \tilde{\mathcal{A}}^{-1}\left(\tilde{\mathcal{D}}_{x}\mathcal{B}\w-\tilde{\mathcal{D}}\mathcal{C}_{x}+\tilde{\mathcal{D}}\mathcal{N}_{x}\right)\d t\right|_{0}^{1}=0.\notag
\end{align}
Similarly, we have
\begin{align}
&-\left.\tilde{\mathcal{D}}\mathcal{A}_{x}\w_{x}\cdot \tilde{\mathcal{A}}^{-1}\left( \left[\begin{array}{c} 0\\\tilde{r}_{x} \mathbb{F}\d  W\end{array}\right] + \left[\begin{array}{c} 0\\ \tilde{r} \mathbb{F}_{x} \d  W
\end{array}\right]\right)\right|_{0}^{1}=0.
\end{align}
Therefore, we have
\begin{align}
\left.\left(\tilde{\mathcal{D}}_{x}\mathcal{A}\w_{x}-\tilde{\mathcal{A}}_{x}\w_{x}\right)\cdot\w_{xx}\d t\right|_{0}^{1}=0.
\end{align}
We calculate
\begin{align}
&\left.\left(\tilde{\mathcal{D}}_{x}\mathcal{B}\w-\tilde{\mathcal{D}}\mathcal{B}\w_{x}\right)\cdot\w_{xx}\d t\right|_{0}^{1}\\
=&\left.\tilde{\mathcal{D}}_{x}\mathcal{B}\w\cdot\w_{xx}\d t\right|_{0}^{1}-\left.\tilde{\mathcal{D}}\mathcal{B}\w_{x}\cdot\w_{xx}\right|_{0}^{1}=\left.\tilde{\mathcal{D}}_{x}\mathcal{B}\w\cdot\w_{xx}\d t\right|_{0}^{1},\notag
\end{align}
since $\left.j_{x}\right|_{0}^{1}=0$. Similarly, the following holds:
\begin{align}
 &\left.\tilde{\mathcal{D}}_{x}\mathcal{B}\w\cdot\w_{xx}\d t\right|_{0}^{1}\notag\\
=&-\left. \tilde{\mathcal{D}}_{x}\mathcal{B}\w\cdot \tilde{\mathcal{A}}^{-1}\left(-\tilde{\mathcal{D}}\d \w_{x}\right)\d t\right|_{0}^{1}\notag\\
 &-\left.\tilde{\mathcal{D}}_{x}\mathcal{B}\w\cdot \tilde{\mathcal{A}}^{-1}\left(\tilde{\mathcal{D}}_{x}\left(\mathcal{A}\w_{x}+\mathcal{B}\w +\mathcal{C} -\mathcal{N}\right)-\left(\tilde{\mathcal{A}}_{x}+\tilde{\mathcal{B}}\right)\w_{x}-\tilde{\mathcal{C}}_{x}+\tilde{\mathcal{N}}_{x}\right)\d t \right|_{0}^{1}\\
 &\quad \left.-\tilde{\mathcal{D}}_{x}\mathcal{B}\w\cdot \tilde{\mathcal{A}}^{-1}\left( \left[\begin{array}{c} 0\\\tilde{r}_{x} \mathbb{F}\d  W\end{array}\right] + \left[\begin{array}{c} 0\\\tilde{r} \mathbb{F}_{x} \d  W
\end{array}\right]\right)\right|_{0}^{1}. \notag
\end{align}
Since $\left.j_{x}\right|_{0}^{1}=0$, we have
\begin{align}
-\left. \tilde{\mathcal{D}}_{x}\mathcal{B}\w\cdot \tilde{\mathcal{A}}^{-1}\left(-\tilde{\mathcal{D}}\d \w_{x}\right)\d t\right|_{0}^{1}= 0,\\
-\left. \tilde{\mathcal{D}}_{x}\mathcal{B}\w\cdot \tilde{\mathcal{A}}^{-1}\left(\tilde{\mathcal{D}}_{x}\mathcal{A}\w_{x}-\tilde{\mathcal{A}}_{x}\w_{x}\right)\d t\right|_{0}^{1}= 0,\\
-\left. \tilde{\mathcal{D}}_{x}\mathcal{B}\w\cdot \tilde{\mathcal{A}}^{-1}\tilde{\mathcal{B}}\w_{x}\d t\right|_{0}^{1}=0,
\end{align}
and
\begin{align}
&-\left.\tilde{\mathcal{D}}_{x}\mathcal{B}\w\cdot \tilde{\mathcal{A}}^{-1}\left(\tilde{\mathcal{D}}_{x}\left(\mathcal{B}\w +\mathcal{C} -\mathcal{N}\right)-\tilde{\mathcal{C}}_{x}+\tilde{\mathcal{N}}_{x}\right)\d t\right|_{0}^{1}\notag\\
=&-\left.\tilde{\mathcal{D}}_{x}\mathcal{B}\w\cdot \tilde{\mathcal{A}}^{-1}\left(\tilde{\mathcal{D}}_{x}\mathcal{B}\w-\tilde{\mathcal{D}}\mathcal{C}_{x}+\tilde{\mathcal{D}}\mathcal{N}_{x}\right)\d t\right|_{0}^{1}=0.
\end{align}
Similarly, it holds that
\begin{align}
\left.-\int_{0}^{t}\tilde{\mathcal{D}}_{x}\mathcal{B}\w\cdot \tilde{\mathcal{A}}^{-1}\left( \left[\begin{array}{c} 0\\\tilde{r}_{x} \mathbb{F}\d  W\end{array}\right] + \left[\begin{array}{c} 0\\\tilde{r} \mathbb{F}_{x} \d  W
\end{array}\right]\right)\right|_{0}^{1}=0.
\end{align}
Hence, we have
\begin{align}
\left.\left(\tilde{\mathcal{D}}_{x}\mathcal{B}\w\right)\cdot\w_{xx}\d t\right|_{0}^{1}=0,
\end{align}
\begin{align}
\left.\left(\tilde{\mathcal{D}}_{x}\mathcal{C}-\tilde{\mathcal{C}}_{x}\right)\cdot\w_{xx}\d t\right|_{0}^{1}=0,\\
\left.\left(-\tilde{\mathcal{D}}_{x}\mathcal{N}+\tilde{\mathcal{N}}_{x}\right)\cdot\w_{xx}\d t\right|_{0}^{1}=0,
\end{align}
and
\begin{align}\label{sto times wxx is zero}
&\left.\left( \left[\begin{array}{c} 0\\\tilde{r}_{x} \mathbb{F}\d  W\end{array}\right] + \left[\begin{array}{c} 0\\\tilde{r} \mathbb{F}_{x} \d  W
\end{array}\right]\cdot\w_{xx}\right)\right|_{0}^{1} \\
=&\left.\left( \left[\begin{array}{c} 0\\\tilde{r}_{x} \mathbb{F}\d  W\end{array}\right] + \left[\begin{array}{c} 0\\\tilde{r} \mathbb{F}_{x} \d  W
\end{array}\right]\right)\cdot \tilde{\mathcal{A}}^{-1}\left(\tilde{\mathcal{D}}_{x}\mathcal{B}\w-\tilde{\mathcal{D}}\mathcal{C}_{x}+\tilde{\mathcal{D}}\mathcal{N}_{x}\right)\right|_{0}^{1}=0. \notag
\end{align}
Therefore, we have
\begin{align}\label{one hand relation}
\left. \tilde{s}\sigma_{xx} j_{xx}\d t\right|_{0}^{1}+ \left.  \tilde{r}\frac{2\left(\bar{J}+j\right)}{\bar{\rho}+\sigma}j_{xx}^{2}\d t\right|_{0}^{1}=0.
\end{align}
On the other hand,
we deal with $\left.-\tilde{\mathcal{D}}\d \w_{x}\cdot\w_{xx}\right|_{0}^{1}$ in \eqref{the other hand} by another way, and other terms in \eqref{the other hand} are shown as before, which are proved equal to $0$. 
From \eqref{formula for barbarA wxx}, we have
\begin{align}
&-\tilde{\mathcal{D}}\d \w_{x}\cdot\w_{xx}= -\tilde{\mathcal{D}}\d \w_{x}\cdot\frac{\w_{xx}\d t}{\d t}\notag\\
=&-\tilde{\mathcal{D}}\d \w_{x}\cdot\frac{-\tilde{\mathcal{A}}^{-1}\tilde{\mathcal{D}}\d \w_{x} }{\d t}\\
&-\tilde{\mathcal{D}}\d \w_{x}\cdot\frac{-\tilde{\mathcal{A}}^{-1}\left(\tilde{\mathcal{D}}_{x}\left(\mathcal{A}\w_{x}+\mathcal{B}\w +\mathcal{C} -\mathcal{N}\right)-\left(\tilde{\mathcal{A}}_{x}+\tilde{\mathcal{B}}\right)\w_{x}-\tilde{\mathcal{C}}_{x}
+\tilde{\mathcal{N}}_{x}\right)\d t}{\d t}\notag\\
&-\tilde{\mathcal{D}}\d \w_{x}\cdot\frac{\tilde{\mathcal{A}}^{-1}\left( \left[\begin{array}{c} 0\\\tilde{r}_{x} \mathbb{F}\d  W\end{array}\right] + \left[\begin{array}{c} 0\\\tilde{r} \mathbb{F}_{x} \d  W
\end{array}\right]\right) }{\d t}.\notag
\end{align}
Since $\left.j_{x}\right|_{0}^{1}=0$ and \eqref{sto times wxx is zero}, we have
\begin{align}
&\left.-\tilde{\mathcal{D}}\d \w_{x}\right|_{0}^{1}\cdot\frac{\left.\tilde{\mathcal{A}}^{-1}\left( \left[\begin{array}{c} 0\\\tilde{r}_{x} \mathbb{F}\d  W\end{array}\right] + \left[\begin{array}{c} 0\\\tilde{r} \mathbb{F}_{x} \d  W
\end{array}\right]\right)\right|_{0}^{1}}{\d t} \notag \\
=& \left. \frac{-\tilde{s} \d \sigma_{x} \frac{1}{\tilde{s}}\left(\tilde{r}_{x} \mathbb{F}\d  W+\tilde{r} \mathbb{F}_{x} \d  W\right)}{\d t}\right|_{0}^{1}=\left. j_{xx}\left(\tilde{r}_{x} \mathbb{F}\d  W+\tilde{r} \mathbb{F}_{x} \d  W\right)\right|_{0}^{1}\\
=&\left.\left( \left[\begin{array}{c} 0\\\tilde{r}_{x} \mathbb{F}\d  W\end{array}\right] + \left[\begin{array}{c} 0\\\tilde{r} \mathbb{F}_{x} \d  W
\end{array}\right]\cdot\w_{xx}\right)\right|_{0}^{1}=0 \notag .
\end{align}
By the previous calculation, there holds
\begin{align}
&\left.-\tilde{\mathcal{D}}\d \w_{x}\cdot\frac{-\tilde{\mathcal{A}}^{-1}\left(\tilde{\mathcal{D}}_{x}\left(\mathcal{A}\w_{x}+\mathcal{B}\w +\mathcal{C} -\mathcal{N}\right)-\left(\tilde{\mathcal{A}}_{x}+\tilde{\mathcal{B}}\right)\w_{x}-\tilde{\mathcal{C}}_{x}+\tilde{\mathcal{N}}_{x}\right)\d t}{\d t}\right|_{0}^{1}\notag\\
=&\left. -\tilde{s} \d \sigma_{x} \frac{1}{\tilde{s}}\tilde{r}_{x}\left(\frac{-2\bar{J}}{\bar{\rho}^{2}}\bar{\rho}_{x}+1\right) j+ \tilde{s}\d \w_{x}\left(-\mathcal{C}_{x}+\mathcal{N}_{x}\right) \right|_{0}^{1}\\
=&\left.\tilde{\mathcal{D}}_{x}\mathcal{B}\w\cdot \w_{xx}\right|_{0}^{1}+\left.\left(-\mathcal{C}_{x}+\mathcal{N}_{x}\right)\cdot \w_{xx} \right|_{0}^{1}=0.\notag
\end{align}
By the mass conservation equation, it holds that
\begin{align}
\left.-\tilde{\mathcal{D}}\d \w_{x}\cdot\frac{-\tilde{\mathcal{A}}^{-1}\tilde{\mathcal{D}}\d \w_{x} }{\d t}\right|_{0}^{1}=\left.\frac{2\tilde{r}\d \sigma_{x}\frac{\bar{J}+j}{\bar{\rho}+\sigma}\d \sigma_{x}}{\d t}\right|_{0}^{1}=2\left.\tilde{r}\frac{\bar{J}+j}{\bar{\rho}+\sigma}j_{xx}^{2}\d t \right|_{0}^{1}.
\end{align}
Hence, from \eqref{the one hand} there holds
\begin{align}
\left. 2\tilde{s}\sigma_{xx} j_{xx}\d t\right|_{0}^{1}\d t =0.
\end{align}
Together with \eqref{one hand relation}, we obtain that
\begin{align}
\left.  \tilde{r}\frac{2\left(\bar{J}+j\right)}{\bar{\rho}+\sigma}j_{xx}^{2}\d t\right|_{0}^{1}=0,\\
\left. 2\tilde{s}\sigma_{xx} j_{xx}\d t\right|_{0}^{1}\d t =0.
\end{align}
We proved the claim.

For $\tilde{\mathcal{Q}}=-4\tilde{\mathcal{D}}_{x}\tilde{\mathcal{D}}_{-1}\tilde{\mathcal{A}}+4\tilde{\mathcal{A}}_{x}+2\tilde{\mathcal{B}}+\tilde{\mathcal{A}}_{x}$, we directly calculate
\begin{align}
\tilde{\mathcal{Q}}=\left[\begin{array}{cc}0 & q_{12}\\ q_{21} & q_{22} \end{array}\right],
\end{align}
where $q_{22}=2\tilde{r}_{x}\frac{\bar{J}+j}{\bar{\rho}+\sigma}+2\tilde{r}\left(1-\frac{2\bar{J}}{\bar{\rho}^{2}}\bar{\rho}_{x}+ 5 \left(\frac{\bar{J}+j}{\bar{\rho}+\sigma}\right)_{x}\right)$, and
\begin{align}
q_{12}+q_{21}=&5\tilde{s}_{x}-4+\tilde{r}_{x}\left(P'\left(\bar{\rho}+\sigma\right)-\frac{\left(\bar{J}+j\right)^{2}}{\left(\bar{\rho}+\sigma\right)^{2}}\right)
              +5\tilde{r}\left(P'\left(\bar{\rho}+\sigma\right)-\frac{\left(\bar{J}+j\right)^{2}}{\left(\bar{\rho}+\sigma\right)^{2}}\right)_{x}\\
              &+2\tilde{r}\left(\frac{-2\bar{J}^{2}}{\bar{\rho}^{3}}\bar{\rho}_{x}+P''\left(\bar{\rho}\right)\bar{\rho}_{x}-\bar{E}\right).\notag
\end{align}
We do Taylor's expectation to the above formula,
\begin{align}
q_{12}+q_{21}=&6\tilde{r}_{x}\left(P'\left(\bar{\rho}\right)-\left(\frac{\bar{J}}{\bar{\rho}}\right)^{2}\right)\notag\\
&+\tilde{r}\left(10\left(P'\left(\bar{\rho}\right)-\left(\frac{\bar{J}}{\bar{\rho}}\right)^{2}\right)_{x}-\frac{4\bar{J}^{2}}{\bar{\rho}^{3}}\bar{\rho}_{x}+2P''\left(\bar{\rho}\right)\bar{\rho}_{x}-2\bar{E}\right)\\
&-4+ O\left(\left|\w\right|+\left|\w_{x}\right|\right).\notag
\end{align}
We need $q_{12}+q_{21}=O\left(\left|\w\right|+\left|\w_{x}\right|\right)$, so we let $\tilde{r}$ satisfy that
\begin{align}
&6\tilde{r}_{x}\left(P'\left(\bar{\rho}\right)-\left(\frac{\bar{J}}{\bar{\rho}}\right)^{2}\right)\\
&+\tilde{r}\left(10\left(P'\left(\bar{\rho}\right)-\left(\frac{\bar{J}}{\bar{\rho}}\right)^{2}\right)_{x}
-\frac{4\bar{J}^{2}}{\bar{\rho}^{3}}\bar{\rho}_{x}+2P''\left(\bar{\rho}\right)\bar{\rho}_{x}-2\bar{E}\right)-4=0.\notag
\end{align}
We denote the above equation as
\begin{align}
\tilde{r}_{x}+G\tilde{r}+M=0,
\end{align}
where
\begin{align}
G=\frac{1}{3}\left(P'\left(\bar{\rho}\right)-\left(\frac{\bar{J}}{\bar{\rho}}\right)^{2}\right)^{-1}\left(5\left(P'\left(\bar{\rho}\right)-\left(\frac{\bar{J}}{\bar{\rho}}\right)^{2}\right)_{x}
-\frac{2\bar{J}^{2}}{\bar{\rho}^{3}}\bar{\rho}_{x}+P''\left(\bar{\rho}\right)\bar{\rho}_{x}-\bar{E}\right),
\end{align}
\begin{align}
M=-\frac{2}{3}\left(P'\left(\bar{\rho}\right)-\left(\frac{\bar{J}}{\bar{\rho}}\right)^{2}\right)^{-1}.
\end{align}
Solving the ordinary differential equation, we get
\begin{align}
\tilde{r}=\tilde{r}(0) e^{-\int_{0}^{x}G\d x}-\int_{0}^{x}M\d x.
\end{align}
where $-M>0$. With $\tilde{r}(0)$ being chosen as some positive constant, $\tilde{r}$ has positive lower bound. So does $\tilde{s}$. Hence, we have
\begin{align}
\tilde{\mathcal{Q}}\w_{xx}\cdot\w_{xx}\geqslant \zeta_{9}j_{xx}^{2}+ O\left(\left|\w\right|+\left|\w_{x}\right|\right)\left|\w_{xx}\right|^{2}.
\end{align}
As in the deterministic case, we estimate
\begin{align}
&\int_{0}^{1}\tilde{\mathcal{D}}_{x}\tilde{\mathcal{D}}^{-1}\tilde{\mathcal{D}}_{x}\mathcal{B}\w\cdot \w_{xx}\d x-\int_{0}^{1}\tilde{\mathcal{D}}_{xx}\mathcal{B}\w\cdot \w_{xx}\d x-\int_{0}^{1}\tilde{\mathcal{D}}_{x}\mathcal{B}_{x}\w\cdot \w_{xx}\d x \\
\ls &  C\left(\int_{0}^{1}\left|j\right|^{2}\d x+\int_{0}^{1}\left|j_{x}\right|^{2}\d x\right).\notag
\end{align}
For the terms involving $\mathcal{C}=\left[\begin{array}{c}0\\-\bar{\rho} \tilde{e}\end{array}\right]$, it holds that
\begin{align}
&\int_{0}^{1}2\tilde{\mathcal{D}}_{x}\left(\tilde{\mathcal{D}}^{-1}\tilde{\mathcal{D}}_{x}\mathcal{C}-\tilde{\mathcal{D}}^{-1}\tilde{\mathcal{C}}_{x}\right)\w_{xx}\d x-\int_{0}^{1}\tilde{\mathcal{D}}_{xx}\mathcal{C}\w_{xx}\d x + \int_{0}^{1}\tilde{\mathcal{C}}_{xx}\w_{xx}\d x\\
=&\frac{\partial }{\partial t}\int_{0}^{1}\left(-r \left(\bar{\rho}_{xx}\tilde{e}+2 \bar{\rho}_{x}\sigma\right)\sigma_{x}-\frac{1}{2}r\left|\sigma_{x}\right|^{2}\right)\d x+ \int_{0}^{1}r\left(\bar{\rho}_{xx}\left(-j\right)+2 \bar{\rho}_{x}\left(-j_{x}\right)\right)\sigma_{x}\d x,\notag
\end{align}
in which $\int_{0}^{1}r\left(\bar{\rho}_{xx}\left(-j\right)+2 \bar{\rho}_{x}\left(-j_{x}\right)\right)\sigma_{x}\d x \ls C\left(\int_{0}^{1}j^{2}\d x+\int_{0}^{1}\sigma_{x}^{2}\d x+\int_{0}^{1}j_{x}^{2}\d x\right)$.
For the terms with $\mathcal{N}=\left[\begin{array}{c}0\\ O\left(\left|\w\right|^{2}+\left|\tilde{e}\right|^{2}+\left|\bar{J}\right|\left|\w\right|\right)\end{array}\right]$, we estimate
\begin{align}
&\int_{0}^{1}\tilde{\mathcal{D}}_{x}\tilde{\mathcal{D}}^{-1}\tilde{\mathcal{N}}_{x}\w_{xx}\d x- \int_{0}^{1}\tilde{\mathcal{D}}_{xx} \mathcal{N} \w_{xx}\d x- \int_{0}^{1}\tilde{\mathcal{D}}_{x} \mathcal{N}_{x} \w_{xx}\d x-\int_{0}^{1}\tilde{\mathcal{N}}_{xx}\w_{xx}\d x\notag\\
\ls &\left\|\w\right\|_{2}^{3}+\left| \bar{J} \right| \left\|\w\right\|_{2}^{2}.
\end{align}
The first term in \eqref{2 order balance} is
\begin{align}
& \int_{0}^{1} \left(( \d \tilde{s}) \left|\sigma_{xx}\right|^{2}+ (\d  \tilde{r}) \left|j_{xx}\right|^{2}\right)\d x \notag\\
=& \int_{0}^{1} \left(P''\left(\bar{\rho}+\sigma\right)\d  \sigma - \frac{2\left(\bar{J}+j\right)}{\left(\bar{\rho}+\sigma\right)^{2}}\d  j+2\frac{\left(\bar{J}+j\right)^{2}}{\left(\bar{\rho}+\sigma\right)^{3}}\d  \sigma\right)\tilde{r} \left|\sigma_{xx}\right|^{2}\d x \d t \\
=& \int_{0}^{1} -\left(P''\left(\bar{\rho}+\sigma\right)+\frac{2\left(\bar{J}+j\right)^{2}}{\left(\bar{\rho}+\sigma\right)^{3}}\right)\left(\bar{J}+j\right)_{x})\tilde{r} \left|\sigma_{xx}\right|^{2}\d x \d t-\int_{0}^{1}\frac{2\left(\bar{J}+j\right)}{\left(\bar{\rho}+\sigma\right)^{2}}\left(\d  j\right)\tilde{r} \left|\sigma_{xx}\right|^{2}\d x\notag\\
\ls & O\left(\left\|\w\right\|_{2}^{3}\right)\d t- \int_{0}^{1}\frac{2\left(\bar{J}+j\right)}{\left(\bar{\rho}+\sigma\right)^{2}}\left(\d  j\right)\tilde{r} \left|\sigma_{xx}\right|^{2}\d x.\notag
\end{align}
By \eqref{1-d system with contact boundary}, we know that
\begin{align}
\d  j+\left(\mathcal{A}_{21}\sigma_{x}+\mathcal{A}_{22} j_{x}+\mathcal{B}_{21}\sigma+\mathcal{B}_{22}j-\bar{\rho}\tilde{e}-O\left(\left|\w\right|^{2}+\left|\tilde{e}\right|^{2}+\left|\bar{J}\right|\left|\w\right|\right)\right)\d t =\mathbb{F}\d  W.
\end{align}
So we estimate
\begin{align}
- \int_{0}^{1}\frac{2\left(\bar{J}+j\right)}{\left(\bar{\rho}+\sigma\right)^{2}}\left(\d  j\right)\tilde{r} \left|\sigma_{xx}\right|^{2}\d x \ls  C \left\|\w\right\|_{2}^{3} +C\left|\int_{0}^{1}\frac{\left(\bar{J}+j\right)}{\left(\bar{\rho}+\sigma\right)^{2}}J^{2}\tilde{r} \left|\w_{xx}\right|^{2}\d x \d  W \right|.
\end{align}
By Burkholder-Davis-Gundy's inequality, we have
\begin{align}
&\mathbb{E}\left[\left|\int_{0}^{t}\int_{0}^{1}\frac{\left(\bar{J}+j\right)}{\left(\bar{\rho}+\sigma\right)^{2}}J^{2}\tilde{r} \left|\w_{xx}\right|^{2}\d x \d  W \right|^{m}\right]\notag\\
\ls & \mathbb{E}\left[\left(C\int_{0}^{t}\int_{0}^{1}\left|\frac{\left(\bar{J}+j\right)}{\left(\bar{\rho}+\sigma\right)^{2}}J^{2}\tilde{r} \left|\w_{xx}\right|^{2}\d x \right|^{2}\d s\right)^{\frac{m}{2}}\right] \\
=& \mathbb{E}\left[\left(C\int_{0}^{t}\int_{0}^{1}\left|\frac{\left(\bar{J}^{3}+3\bar{J}^{2}j+3 \bar{J}j^{2}+ +j^{3}\right)}{\left(\bar{\rho}+\sigma\right)^{2}}\tilde{r} \left|\w_{xx}\right|^{2}\d x \right|^{2}\d s\right)^{\frac{m}{2}}\right]. \notag
\end{align}
We estimate term by term as follows:
\begin{align}
  & \mathbb{E}\left[\left(\int_{0}^{t}\left|\int_{0}^{1}\frac{ \bar{J}^{3}}{\left(\bar{\rho}+\sigma\right)^{2}}\tilde{r} \left|\w_{xx}\right|^{2}\d x\right|^{2}\d s\right)^{\frac{m}{2}} \right]\notag\\
\ls &\mathbb{E}\left[\left(C\int_{0}^{t}\left| \bar{J} \right|^{6}\left\|\w_{xx}\right\|^{4}\d s\right)^{\frac{m}{2}} \right]
\ls  \mathbb{E}\left[\left(\left| \bar{J} \right|^{5}\sup\limits_{s\in[0,t]} \left\|\w_{xx}\right\|^{2}
   C\int_{0}^{t}\left| \bar{J} \right|\left\|\w_{xx}\right\|^{2}\d s\right)^{\frac{m}{2}} \right]\\
\ls & \mathbb{E}\left[\left(\left| \bar{J} \right|^{5}\sup\limits_{s\in[0,t]} \left\|\w_{xx}\right\|^{2}\right)^{m}\right]+ \mathbb{E}\left[\left(C\int_{0}^{t}\left| \bar{J} \right|\left\|\w\right\|_{2}^{2} \d s\right)^{m}\right],\notag
\end{align}
where $ \left|\bar{J} \right|^{5}$ is small enough such that the first term can be balanced by the left side of this second-order estimates; there hold
\begin{align}
  & \mathbb{E}\left[\left(\int_{0}^{t}\left|\int_{0}^{1}\frac{ \bar{J}^{2}j}{\left(\bar{\rho}+\sigma\right)^{2}}\tilde{r} \left|\w_{xx}\right|^{2}\d x\right|^{2}\d s\right)^{\frac{m}{2}}\right]\notag\\
\ls &\mathbb{E}\left[\left(C\int_{0}^{t}\left| \bar{J} \right|^{4}\left\|j\right\|_{\infty}^{2}\left\|\w_{xx}\right\|^{4}\d s\right)^{\frac{m}{2}} \right]\\
\ls & \mathbb{E}\left[\left(C\sup_{s\in[0, t]}\left| \bar{J} \right\|^{4}\left\|j\right\|_{\infty}^{2}\right)^{\frac{m}{2}}
     \left(\int_{0}^{t}\left\|\w_{xx}\right\|^{4}\d s\right)^{\frac{m}{2}} \right]\notag\\
\ls &\left(\frac{\vartheta_{7}}{3}\right)^{m}\mathbb{E}\left[\sup_{s\in[0, t]}\left\|\w_{x}\right\|^{2m}\right] +\mathbb{E}\left[\left(C_{\vartheta_{7}}\int_{0}^{t}\left\|\w\right\|_{2}^{3} \d s\right)^{m}\right],\notag
\end{align}
\begin{align}
  & \mathbb{E}\left[\left(\int_{0}^{t}\left|\int_{0}^{1}\frac{\bar{J}j^{2}}{\left(\bar{\rho}+\sigma\right)^{2}}\tilde{r} \left|\w_{xx}\right|^{2}\d x\right|^{2}\d s\right)^{\frac{m}{2}}\right]\notag\\
\ls & \mathbb{E}\left[\left(C\int_{0}^{t}\left| \bar{J} \right|^{2}\left\|j\right\|_{\infty}^{4}\left\|\w_{xx}\right\|^{4}\d s\right)^{\frac{m}{2}} \right]\\
\ls &\mathbb{E}\left[\left(C\sup_{s\in[0, t]}\left| \bar{J} \right|^{2}\left\|j\right\|_{\infty}^{4}\right)^{\frac{m}{2}}
     \left(\int_{0}^{t}\left\|\w_{xx}\right\|^{4}\d s\right)^{\frac{m}{2}} \right]\notag\\
\ls & \left(\frac{\vartheta_{7}}{3}\right)^{m}\mathbb{E}\left[\sup_{s\in[0, t]}\left\|\w_{x}\right\|^{2m}\right] +\mathbb{E}\left[\left(C_{\vartheta_{7}}\int_{0}^{t}\left\|\w\right\|_{2}^{3} \d s\right)^{m}\right],\notag
\end{align}
and
\begin{align}
  & \mathbb{E}\left[\left(\int_{0}^{t}\left|\int_{0}^{1}\frac{j^{3}}{\left(\bar{\rho}+\sigma\right)^{2}}\tilde{r} \left|\w_{xx}\right|^{2}\d x\right|^{2}\d s\right)^{\frac{m}{2}}\right]\notag\\
\ls & \mathbb{E}\left[\left(C\int_{0}^{t}\left\|j\right\|_{\infty}^{6}\left\|\w_{xx}\right\|^{4}\d s\right)^{\frac{m}{2}} \right]\\
\ls & \mathbb{E}\left[\left(C\sup_{s\in[0, t]}\left\|j\right\|_{\infty}^{6}\right)^{\frac{m}{2}}
     \left(\int_{0}^{t}\left\|\w_{xx}\right\|^{4}\d s\right)^{\frac{m}{2}} \right]\notag\\
\ls & \left(\frac{\vartheta_{7}}{3}\right)^{m}\mathbb{E}\left[\sup_{s\in[0, t]}\left\|\w_{x}\right\|^{2m}\right] +\mathbb{E}\left[\left(C_{\vartheta_{7}}\int_{0}^{t}\left\|\w\right\|_{2}^{3} \d s\right)^{m}\right].\notag
\end{align}
Let $\vartheta_{7}$ be such that $\vartheta_{7}^{m}\mathbb{E}\left[\sup_{s\in[0, t]}\left\|\w_{x}\right\|^{2m}\right]$ can be bounded by the first-order estimates.
The covariation in \eqref{2 order balance} is
\begin{align}
&\tilde{r}^{-1}\left\langle 2\tilde{r}_{x}\tilde{r}^{-1}\left(\tilde{r}_{x}\mathbb{F}+\tilde{r}\mathbb{F}_{x}\right)\d  W - 2 \tilde{r}_{x}\mathbb{F}_{x}\d  W -\tilde{r}\mathbb{F}_{xx}\d  W \right\rangle\notag\\
=&\tilde{r}^{-1}\left|2 \tilde{r}_{x}\tilde{r}^{-1}\left(\tilde{r}_{x}\mathbb{F}+\tilde{r}\mathbb{F}_{x} \right)\right|^{2}\d t-8\tilde{r}_{x}\tilde{r}^{-1}\left(\tilde{r}_{x}\mathbb{F}+\tilde{r}\mathbb{F}_{x}\right)\tilde{r}_{x}\mathbb{F}_{x} \d t\notag\\
 & - 4\tilde{r}_{x}\tilde{r}^{-1}\left(\tilde{r}_{x}\mathbb{F}+\tilde{r}\mathbb{F}_{x}\right)\tilde{r}\mathbb{F}_{xx} \d t -4\tilde{r}_{x}\mathbb{F}_{x}\tilde{r}\mathbb{F}_{xx} \d t+ 4\tilde{r}_{x}^{2}\mathbb{F}_{x}^{2} \d t + \tilde{r}^{2}\mathbb{F}_{xx}^{2} \d t.
\end{align}
We estimate $-\int_{0}^{1}8\tilde{r}_{x}\tilde{r}^{-1}\left(\tilde{r}_{x}\mathbb{F}+\tilde{r}\mathbb{F}_{x}\right)\tilde{r}_{x}\mathbb{F}_{x} \d x \d t$, by \eqref{condition for F},
\begin{align}
\left|-\int_{0}^{1}8\tilde{r}_{x}\tilde{r}^{-1}\left(\tilde{r}_{x}\mathbb{F}+\tilde{r}\mathbb{F}_{x}\right)\tilde{r}_{x}\mathbb{F}_{x} \d x \d t\right|\ls C\left( \left\|\w\right\|_{2}^{3} +\left| \bar{J} \right| \left\|\w\right\|_{2}^{2}\right)\d t.
\end{align}
Other terms are dealt with similarly. Therefore, there holds
\begin{align}
&\int_{0}^{1}\tilde{r}^{-1}\left\langle 2\tilde{r}_{x}\tilde{r}^{-1}\left(\tilde{r}_{x}\mathbb{F}+\tilde{r}\mathbb{F}_{x}\right)\d  W - 2 \tilde{r}_{x}\mathbb{F}_{x}\d  W -\tilde{r}\mathbb{F}_{xx}\d  W \right\rangle \d x \ls C \left\|\w\right\|_{2}^{3} \d t.
\end{align}
Next we consider the stochastic terms in $2\w_{xx}\cdot \tilde{\mathcal{D}}\d  \w_{xx}$, i.e.,
\begin{align}
&4\tilde{r}_{x}\tilde{r}^{-1}\tilde{r}_{x}\mathbb{F} j_{xx}\d  W-4\tilde{r}_{x}\tilde{r}^{-1}\left(\tilde{r}_{x}\mathbb{F}+\tilde{r}\mathbb{F}_{x}\right) j_{xx}\d  W-2\tilde{r}_{xx}\mathbb{F}j_{xx}\d  W\\
&+2\tilde{r}_{xx}\mathbb{F} j_{xx}\d  W +4 \tilde{r}_{x}\mathbb{F}_{x} j_{xx}\d  W+ 2\tilde{r}\mathbb{F}_{xx}j_{xx}\d  W= 2\tilde{r}\mathbb{F}_{xx}j_{xx}\d  W.\notag
\end{align}
Since \eqref{condition for F}, by Soblev's inequality, it holds that
\begin{align} 
&\mathbb{E}\left[\left|\int_{0}^{t}\int_{0}^{1}2\tilde{r}\mathbb{F}_{xx} j_{xx} \d x \d  W\right|^{m}\right]\ls \mathbb{E}\left[\left(\int_{0}^{t}\left|\int_{0}^{1}C \tilde{r}\left( \left|J_{x}\right|^{2}+\left|J\right|\left|J_{xx}\right|\right)j_{xx} \d x\right|^{2} \d s\right)^{\frac{m}{2}}\right]\notag\\
\ls &\mathbb{E}\left[\left(\int_{0}^{t}C\tilde{r}\left(\left\|J_{x}\right\|^{2}\left\|J_{x}\right\|_{\infty}^{2}\left\|j_{xx}\right\|^{2}
                     +\left\|J\right\|_{\infty}^{2}\left\|j_{xx}\right\|^{4} \right) \d s\right)^{\frac{m}{2}}\right]\\
\ls &\mathbb{E}\left[\left(\int_{0}^{t}C\tilde{r}\left\|J_{x}\right\|^{2}\left\|j_{xx}\right\|^{4}  \d s\right)^{\frac{m}{2}}\right]\notag\\
\ls & \mathbb{E} \left[\left\|\w_{0} \right\|^{2m} \right]+ \mathbb{E}\left[\left(\frac{\vartheta_{8}}{2}\right)^{m}\sup_{s\in[0,t]}\left\|j_{x}\right\|^{2m}\right]
+\mathbb{E}\left[\left(\int_{0}^{t}C_{\vartheta_{8}}\left( \left\|\w\right\|_{2}^{3} +\left| \bar{J} \right| \left\|\w\right\|_{2}^{2}\right)\d s\right)^{m}\right].\notag
\end{align}
Combining the above estimates, we have
\begin{align}\label{2 order estimate 1}
&\mathbb{E}\left[\left(\int_{0}^{1} \left(\left|\sigma_{xx}\right|^{2}+\left|j_{xx}\right|^{2}\right)(t)\right)^{m}\right]+ \mathbb{E}\left[\left(\zeta_{10}\int_{0}^{t}\int_{0}^{1}j_{xx}^{2} \d s\right)^{m}\right]\notag\\
\ls &\mathbb{E}\left[\left(\int_{0}^{1} \left(\left|\sigma_{xx}\right|^{2}+\left|j_{xx}\right|^{2}\right)(0)\right)^{m}\right]+C\mathbb{E} \left[\left\|\w_{0} \right\|^{2m} \right]+C\mathbb{E} \left[\left\|\left(\w_{0}\right)_{x} \right\|^{2m} \right]\\
&+ \mathbb{E}\left[\left(\int_{0}^{t}C\left(\left|\bar{J}\right|\left\|\w\right\|_{2}^{2}+\left\|\w\right\|_{2}^{3}\right)\d s\right)^{m}\right].\notag
\end{align}
Similarly as in the first-order estimates, we multiply $-\tilde{\mathcal{A}}\w_{xxx}$ with $\left[\begin{array}{c}0\\ \sigma_{x}\end{array}\right]$, and integrate with respect to $x$ to give the estimate of $\int_{0}^{t}\int_{0}^{1}\left|\sigma_{xx}\right|^{2} \d x \d s$. In fact, by It\^o's formula, we know
\begin{align}
\tilde{\mathcal{D}}\d \left(\w_{xx}\cdot\left[\begin{array}{c}0  \\ \sigma_{x}\end{array}\right]\right)= \tilde{\mathcal{D}}\d \w_{xx}\cdot\left[\begin{array}{c}0  \\ \sigma_{x}\end{array}\right]+\tilde{\mathcal{D}} \w_{xx} \cdot\left[\begin{array}{c}0  \\ -j_{xx}\end{array}\right]\d t.
\end{align}
Then, we have
\begin{align}
&-\tilde{\mathcal{D}}\d  \left(\w_{xx}\cdot\left[\begin{array}{c}0  \\ \sigma_{x}\end{array}\right]\right) + \tilde{\mathcal{D}} \w_{xx} \cdot\left[\begin{array}{c}0  \\ -j_{xx}\end{array}\right]\d t - 2 \tilde{\mathcal{D}}_{x}\tilde{\mathcal{D}}^{-1}\tilde{\mathcal{D}}_{x}\left(\mathcal{A}\w_{x}+\mathcal{B}\w + \mathcal{C}\right)\cdot\left[\begin{array}{c}0\\ \sigma_{x}\end{array}\right] \d t\notag\\
&+2\tilde{\mathcal{D}}_{x}\tilde{\mathcal{D}}^{-1}\tilde{\mathcal{D}}_{x}\left(\mathcal{N} \d t + \left[\begin{array}{c} 0  \\ \mathbb{F}\d  W\end{array}\right]\right)\cdot\left[\begin{array}{c}0\\ \sigma_{x}\end{array}\right]
 +2\tilde{\mathcal{D}}_{x}\tilde{\mathcal{D}}^{-1}\tilde{\mathcal{A}}\w_{xx}\cdot\left[\begin{array}{c}0\\ \sigma_{x}\end{array}\right]\d t\notag \\
 &+2\tilde{\mathcal{D}}_{x}\tilde{\mathcal{D}}^{-1}\left(\left(\tilde{\mathcal{A}}_{x}+\tilde{\mathcal{B}}\right)\w_{x}+\tilde{\mathcal{B}}_{x}\w + \tilde{\mathcal{C}}_{x}\right)\cdot\left[\begin{array}{c}0\\ \sigma_{x}\end{array}\right] \d t- \tilde{\mathcal{N}}_{x}\cdot\left[\begin{array}{c}0\\ \sigma_{x}\end{array}\right] \d t  \notag \\
& - 2\tilde{\mathcal{D}}_{x}\tilde{\mathcal{D}}^{-1}\left( \tilde{\mathcal{D}}_{x}\left[\begin{array}{c}0  \\ \mathbb{F}\d  W\end{array}\right] + \tilde{\mathcal{D}} \left[\begin{array}{c}0  \\ \mathbb{F}_{x}\d  W\end{array}\right]\right)\cdot\left[\begin{array}{c}0\\ \sigma_{x}\end{array}\right] \\
&- \tilde{\mathcal{D}}_{xx} \left(-\mathcal{A}\w_{x}-\mathcal{B}\w - \mathcal{C} + \mathcal{N}\right)\cdot\left[\begin{array}{c}0\\ \sigma_{x}\end{array}\right] \d t - \tilde{\mathcal{D}}_{xx}\left[\begin{array}{c}0  \\ \mathbb{F}\d  W\end{array}\right] \cdot\left[\begin{array}{c}0\\ \sigma_{x}\end{array}\right] \notag \\
 &-\left(\tilde{\mathcal{A}}_{x}\w_{xx}+\tilde{\mathcal{A}}\w_{xxx} + \left(\tilde{\mathcal{A}}_{xx}+\tilde{\mathcal{B}}_{x}\right)\w_{x} +\left(\tilde{\mathcal{A}}_{x}+\tilde{\mathcal{B}}\right)\w_{xx}+\tilde{\mathcal{B}}_{xx}\w + \tilde{\mathcal{B}}_{x}\w_{x} +\tilde{\mathcal{C}}_{xx}\right)\cdot\left[\begin{array}{c}0\\ \sigma_{x}\end{array}\right] \d t \notag \\
=&- \tilde{\mathcal{N}}_{xx}\cdot\left[\begin{array}{c}0\\ \sigma_{x}\end{array}\right] \d t  + \tilde{\mathcal{D}}_{xx} \left[\begin{array}{c}0  \\ \mathbb{F}\d  W\end{array}\right]\cdot\left[\begin{array}{c}0\\ \sigma_{x}\end{array}\right] + 2\tilde{\mathcal{D}}_{x} \left[\begin{array}{c}0  \\ \mathbb{F}_{x}\d  W\end{array}\right]\cdot\left[\begin{array}{c}0\\ \sigma_{x}\end{array}\right] \notag \\
&+ \tilde{\mathcal{D}}\left[\begin{array}{c}0  \\ \mathbb{F}_{xx}\d  W\end{array}\right]\cdot\left[\begin{array}{c}0\\ \sigma_{x}\end{array}\right],  \notag
\end{align}
where the first term holds
\begin{align}
-\tilde{\mathcal{D}}\d  \left(\w_{xx}\cdot\left[\begin{array}{c}0  \\ \sigma_{x}\end{array}\right]\right) =-\d \left(\tilde{r}j_{xx}\sigma_{x}\right),
 \end{align}
 the second term holds
 \begin{align}
\tilde{\mathcal{D}} \w_{xx} \cdot\left[\begin{array}{c}0  \\ -j_{xx}\end{array}\right]\d t=\tilde{r} j_{xx}^{2}\d t.
  \end{align}
We consider the stochastic terms $\tilde{\mathcal{D}}\left[\begin{array}{c}0  \\ \mathbb{F}_{xx}\d  W\end{array}\right]\cdot\left[\begin{array}{c}0\\ \sigma_{x}\end{array}\right]$,
\begin{align}
& \mathbb{E} \left[\left| \int_{0}^{t}\int_{0}^{1} \tilde{r}\mathbb{F}_{xx}\sigma_{x} \d x \d  W \right|^{m}\right]
\ls  \mathbb{E} \left[ \left(\int_{0}^{t}\left|\int_{0}^{1} \tilde{r} \mathbb{F}_{xx}\sigma_{x} \d x \right|^{2} \d s \right)^{\frac{m}{2}} \right] \\
\ls & \mathbb{E} \left[  \left(\frac{\vartheta_{8}}{3}\sup\limits_{s\in[0, t]}\int_{0}^{1}\left|j_{xx}\right|^{2}\d x\right)^{m} \right]+\mathbb{E} \left[\left\|\w_{0} \right\|^{2m} \right]+C_{\vartheta_{8}}\mathbb{E} \left[\left( \int_{0}^{t} \int_{0}^{1} C\left( \left\|\w\right\|_{2}^{3} +\left| \bar{J} \right| \left\|\w\right\|_{2}^{2}\right)\d s \right)^{m}\right],\notag
\end{align}
where $\mathbb{E} \left[\left(\frac{\vartheta_{8}}{3} \sup\limits_{s\in[0, t]}\int_{0}^{1}\left|j_{xx}\right|^{2}\d x\right)^{m} \right]$ can be balanced.
Taking $ 2\tilde{\mathcal{D}}_{x}\tilde{\mathcal{D}}^{-1} \tilde{\mathcal{D}} \left[\begin{array}{c}0  \\ \mathbb{F}_{x}\d  W\end{array}\right]\cdot\left[\begin{array}{c}0\\ \sigma_{x}\end{array}\right]$ for example, we estimate
\begin{align}
& \mathbb{E} \left[\left| \int_{0}^{t}\int_{0}^{1} \tilde{r}_{x} \mathbb{F}_{x}\sigma_{x} \d x \d  W \right|^{m}\right]
\ls  \mathbb{E} \left[ \left(\int_{0}^{t}\left|\int_{0}^{1} \tilde{r}_{x} \mathbb{F}_{x}\sigma_{x} \d x \right|^{2} \d s \right)^{\frac{m}{2}} \right] \\
\ls & \mathbb{E} \left[\left\|\w_{0} \right\|^{2m} \right]+\mathbb{E} \left[ \left(C \sup\limits_{s\in[0, t]}\int_{0}^{1}\left|\sigma_{x}\right|^{2}\d x\right)^{m} \right] + \mathbb{E} \left[\left( \int_{0}^{t} \int_{0}^{1} C\left( \left\|\w\right\|_{2}^{3} +\left| \bar{J} \right| \left\|\w\right\|_{2}^{2}\right) \d s \right)^{m}\right],\notag
\end{align}
while the estimates for other terms are also bounded by
 $$\mathbb{E} \left[\left\|\w_{0} \right\|^{2m} \right]+ \mathbb{E} \left[  \left(C\sup\limits_{s\in[0, t]}\int_{0}^{1}\left|\sigma_{x}\right|^{2}\d x\right)^{m} \right]
+\mathbb{E} \left[\left( \int_{0}^{t} \int_{0}^{1} C\left( \left\|\w\right\|_{2}^{3} +\left| \bar{J} \right| \left\|\w\right\|_{2}^{2}\right) \d s \right)^{m}\right].$$
We calculate the key term
 \begin{align}
&-\tilde{\mathcal{A}}\w_{xxx} \cdot \left[\begin{array}{c} 0 \\ \sigma_{x}\end{array}\right] \d t \notag \\
= & -\left(\tilde{\mathcal{A}}\w_{xx}\cdot \left[\begin{array}{c} 0 \\ \sigma_{x}\end{array}\right]\right)_{x}\d t+ \tilde{\mathcal{A}}_{x}\w_{xx}\cdot \left[\begin{array}{c} 0 \\ \sigma_{x}\end{array}\right]\d t+ \tilde{\mathcal{A}}\w_{xx}\cdot  \left[\begin{array}{c} 0 \\ \sigma_{xx}\end{array}\right] \d t\\
 = & -\left(\tilde{\mathcal{A}}\w_{xx}\left[\begin{array}{c} 0 \\ \sigma_{x} \end{array}\right]\right)_{x}\d t +  \left( -\left(\frac{\bar{J}+j}{\bar{\rho}+\sigma}\right)^{2}+P'\left(\bar{\rho}+\sigma\right)\right)r \sigma_{xx}^{2}\d t+  2\frac{\bar{J}+j}{\bar{\rho}+\sigma}rj_{xx}\sigma_{x}\d t \notag\\
 &+\left(\left( -\left(\frac{\bar{J}+j}{\bar{\rho}+\sigma}\right)^{2}+P'\left(\bar{\rho}+\sigma\right)\right)r\right)_{x}\sigma_{xx}\sigma_{x}\d t+ 2\left(\frac{\bar{J}+j}{\bar{\rho}+\sigma}r\right)_{x} j_{xx}\sigma_{x}\d t.\notag
 \end{align}
 We integrate the above formula in $[0,1]$ with respect to $x$. Then, the first term holds
  \begin{align}
  &-\int_{0}^{t}\int_{0}^{1}\left(\tilde{\mathcal{A}}\w_{xx}\left[\begin{array}{c} 0 \\ \sigma_{x} \end{array}\right]\right)_{x}\d x \d s \\ =&\int_{0}^{t}\int_{0}^{1}2\frac{\bar{J}+j}{\bar{\rho}+\sigma}r \left(\sigma_{x}^{2}\right)_{t} \d x \d s\ls C\sup\limits_{s\in[0, t]}\int_{0}^{1}\left|\sigma_{x}\right|^{2}\d x.\notag
   \end{align}
 The second term is at leat $\zeta_{11}\sigma_{xx}^{2}$.
 The other three terms are bounded by
 \begin{align}
 C\left(\bar{J}+j\right)\w_{xx}^{2}+C\left| \w_{xx}\cdot \w_{x} \right|\ls  C\left(\bar{J}+j\right)\w_{xx}^{2}+\frac{\vartheta_{8}}{3}\w_{xx}^{2}+C_{\vartheta_{8}}\w_{x}^{2},
\end{align}
for $ \vartheta_{8}\ls \min\left\{\frac{\zeta_{10}}{4},\frac{\zeta_{11}}{4}\right\}$.
Therefore, we have
\begin{align}
 &\mathbb{E} \left[\left|\int_{0}^{t}\int_{0}^{1}\zeta_{12}\left(\sigma_{xx}^{2}+j_{xx}^{2}\right)\d x \d s \right|^{m}\right]\notag\\
  \ls
& \mathbb{E} \left[\left( \int_{0}^{t} \int_{0}^{1} C\left|\bar{J}\right|\left\|\w\right\|_{2}^{2} \d s \right)^{m}\right]+ \mathbb{E} \left[\left( \int_{0}^{t} \int_{0}^{1} C\left\|\w\right\|_{2}^{3} \d s \right)^{m}\right]\\
&+\mathbb{E} \left[ \left(\frac{\vartheta_{8}}{3} \sup\limits_{s\in[0, t]}\int_{0}^{1}\left|j_{xx}\right|^{2}\d x\right)^{m} \right]. \notag
\end{align}
In conclusion, together with \eqref{2 order estimate 1}, the second-order estimates can be written as
\begin{align}
&\mathbb{E} \left[\left|\sup\limits_{s\in[0,t]} \int_{0}^{1} \left(j_{xx}^{2}+\sigma_{xx}^{2}\right)(s)\d x \right|^{m}\right] + \mathbb{E} \left[\left|\zeta_{13} \int_{0}^{t}\int_{0}^{1}\left(j_{xx}^{2} +\sigma_{xx}^{2}\right) \d x \d s  \right|^{m}\right]\notag\\
\ls & \mathbb{E} \left[\left| \int_{0}^{1} \left(j_{xx}^{2}+\sigma_{xx}^{2}\right)(0)\d x \right|^{m}\right]+C\mathbb{E} \left[\left\|\w_{0} \right\|^{2m} \right]\\
&+\mathbb{E} \left[\left|C_{6}\int_{0}^{t} \left| \bar{J} \right| \left\|\w\right\|_{2}^{2} \d s +C_{6}\int_{0}^{t}\left\|\w\right\|_{2}^{3}\d s \right|^{m}\right]. \notag
\end{align}

 \section{Global existence}


 To give the global existence, we first establish the local existence of solutions by Banach's fixed point theorem. In the following estimates, the constants in the onto mapping estimates can depend on a given fixed time $T$, demanding less on the fineness of energy estimates compared to the estimates in the last section.  
 
 \subsection{Local existence}

For given $\sigma$ and $j$, we consider the following linearized system with $\hat{\w}$ known,
 \begin{align}\label{linearized system}
\d  \w+\left(\mathcal{A}\left(\hat{\w}\right)\w_{x}+\mathcal{B}\left(\bar{\w}\right)\w + \mathcal{C}(\w)\right)\d t= \mathcal{N}\left(\hat{\w},\w\right)\d t+ \left[\begin{array}{c}0 \\ \mathbb{F}\left(\hat{\w}\right)\d  W\end{array}\right].
 \end{align}
We denote $M=\sup\limits_{t\in[0,T]} \left\|\hat{\sigma}, \hat{J}\right\|_{2}$.

\begin{flushleft}
\textbf{Step 1: Uniform bound.}
\end{flushleft}
 \subsubsection{Zeroth-order estimates for linearized system}
 For the zeroth-order estimates, we just multiply system \eqref{linearized system} by $\w$, then we integrate it with respect to $x$ and $t$. From \eqref{mathcal A}, the expression of $\mathcal{A}$, we estimate
 \begin{align}
& \int_{0}^{1}\mathcal{A}\left(\hat{\w}\right)\w_{x}\cdot \w \d x \d t \ls C \left\| \hat{\w}\right\|_{1}\left\|\w\right\|^{2}\d t. \notag
 \end{align}
 By \eqref{mathcal B} and smoothness of steady state, there holds
\begin{align}
  \int_{0}^{1} \mathcal{B}\left(\bar{\w}\right)\w \cdot \w \d x \d t \ls C\left\|\w\right\|^{2}\d t.
\end{align}
Due to \eqref{mathcal C}, it holds that
\begin{align}
  \int_{0}^{1} \mathcal{C} \cdot \w \d x \d t \ls C \left\|\tilde{e}\right\|\left\|\w\right\|\d t.
\end{align}
 By \eqref{mathcal N} and the smoothness of steady state, we have
\begin{align}
 \int_{0}^{1} \mathcal{N}\left(\hat{\w},\w\right) \cdot \w \d x \d t \ls C\left(\left\|\hat{\w}\right\|_{1}+1\right)\left\|\w\right\|^{2}\d t.
\end{align}
For the stochastic term, taking the $m$-expectation, we estimate as follows
\begin{align}
 & \mathbb{E}\left[\left|\int_{0}^{t} \int_{0}^{1} \left[\begin{array}{c}0 \\ \mathbb{F}\left(\hat{\w}\right)\d  W\end{array}\right]\cdot \w \d x \right|^{m}\right]
\ls  \mathbb{E}\left[\left|\int_{0}^{t} \int_{0}^{1} \mathbb{F}\left(\hat{\w}\right) j \d x \d  W\right|^{m}\right]\notag \\
\ls & \mathbb{E}\left[\left(C\int_{0}^{t} \left|\int_{0}^{1} \mathbb{F}\left(\hat{\w}\right) j \d x \right|^{2}\d s \right)^{\frac{m}{2}}\right]\ls \mathbb{E}\left[\left(C\int_{0}^{t} \left\|\hat{J}\right\|_{1}^{4}\left\|j\right\|^{2}  \d s \right)^{\frac{m}{2}}\right]\\
 \ls & \mathbb{E}\left[\left( C\sup_{s\in[0,t]}\left\|\hat{J}\right\|_{1}^{4}\int_{0}^{t}\left\|j\right\|^{2}  \d s \right)^{\frac{m}{2}}\right] \notag\\
 \ls &\frac{1}{2^{m}} \mathbb{E}\left[\left( \sup_{s\in[0,t]}\left\|\w\right\|^{2} \right)^{m}\right]+ \mathbb{E}\left[\left(C\int_{0}^{t}\left\|\hat{J}\right\|_{1}^{4} \d s \right)^{m}\right]. \notag
\end{align}
By It\^o's formula, we have
\begin{align}
\d \left|\w\right|^{2}=2\d \w \cdot \w + \langle  \mathbb{F}\left(\hat{\w}\right)\d  W,  \mathbb{F}\left(\hat{\w}\right)\d  W\rangle=  2\d \w \cdot \w +\left|\mathbb{F}\left(\hat{\w}\right)\right|^{2} \d t.
\end{align}
We directly estimate
\begin{align}
\int_{0}^{1}\left|\mathbb{F}\left(\hat{\w}\right)\right|^{2}\d x \ls C \left\|\bar{J}+\hat{j}\right\|^{2}\left\|\bar{J}+\hat{j}\right\|_{\infty} \ls C\left(\left\|\hat{\w}\right\|_{1}^{2}+\left\|\hat{\w}\right\|_{1}+1\right)\left(\left\|\hat{\w}\right\|^{2}+\left\|\hat{\w}\right\|+1\right).
\end{align}
Combining the above estimates, we have
\begin{align}
&\mathbb{E}\left[\left( \sup_{s\in[0,t]}\int_{0}^{1}\left|\w(s)\right|^{2}\d x\right)^{m}\right] \notag\\
\ls & \mathbb{E}\left[\left(C\sup\limits_{s\in[0,t]}\left(\left\| \hat{\w}\right\|_{1}\right) \int_{0}^{t}\int_{0}^{1} \left|\w\right|^{2}\d x \d s\right)^{m}\right]\notag\\
&+\mathbb{E}\left[\left(C\int_{0}^{t}\left(\left\|\hat{\w}\right\|_{1}^{4}+\left\|\hat{\w}\right\|_{1}^{3}+\left\|\hat{\w}\right\|_{1}^{2}+\left\|\hat{\w}\right\|_{1} +1\right)\d s \right)^{m}\right]  \\
 &+\mathbb{E}\left[\left( C\int_{0}^{t}\int_{0}^{1} \left|\w\right|^{2}\d x \d s + \int_{0}^{1}\left|\w(0)\right|^{2}\d x\right)^{m}\right]\notag\\
\ls &C_{M,m}\mathbb{E}\left[\left(\int_{0}^{t}\int_{0}^{1} \left(\left|\w_{x}\right|^{2}+\left|\w\right|^{2}\right)\d x \d s\right)^{m}\right]+ C_{M,m}t^{m}+C\mathbb{E}\left[\left( \int_{0}^{1}\left|\w_{0}\right|^{2}\d x\right)^{m}\right] ,\notag
\end{align}
where $C_{M,m}$ depends on $m$ and $M=\sup\limits_{t\in[0,T]} \left\|\hat{\sigma}, \hat{J}\right\|_{2}$.

\subsubsection{Higher-order estimates for linearized system}

We multiply \eqref{linearized system} on the left by $\mathcal{D}$ to obtain
\begin{align}\label{symmetrized sto Euler-Poisson system contact boundary}
\d  \left(\mathcal{D} \w\right)+\left(\bar{\mathcal{A}} \left(\hat{\w}\right)\w_{x}+\bar{\mathcal{B}}\left(\bar{\w}\right)\w + \bar{\mathcal{C}}\right)\d t= \bar{\mathcal{N}}\left(\hat{\w},\w\right)\d t+ \mathcal{D} \left[\begin{array}{c}0 \\ \mathbb{F}\left(\hat{\w}\right)\d  W\end{array}\right],
\end{align}
where $\mathcal{D}$ is defined in \eqref{symmetrizing matrix of first order estim}, $\bar{\mathcal{A}}=\mathcal{D}\mathcal{A}$, $\bar{\mathcal{B}}=\mathcal{D}\mathcal{B}$, $ \bar{\mathcal{C}}=\mathcal{D}\mathcal{C}$, $\bar{\mathcal{N}} =\mathcal{D}\mathcal{N}$.
Differentiating \eqref{symmetrized sto Euler-Poisson system contact boundary} with respect to $x$, we get
\begin{align}\label{1 order derivative of symmetrized system 1}
&\d  \left(\mathcal{D} \w_{x}\right) - \left(\mathcal{D}_{x} \left(\mathcal{A}\left(\hat{\w}\right) \w_{x}+\mathcal{B}\left(\bar{\w}\right)\w + \mathcal{C}- \mathcal{N}\left(\hat{\w},\w\right)\right)\right)\d t \notag \\
 &+ \left(\bar{\mathcal{A}}\left(\hat{\w}\right)\w_{xx} + \left(\left(\bar{\mathcal{A}}\left(\hat{\w}\right)\right)_{x}+\bar{\mathcal{B}}\left(\bar{\w}\right)\right)\w_{x}+\left(\bar{\mathcal{B}}\left(\bar{\w}\right)\right)_{x}\w + \bar{\mathcal{C}}_{x}\right) \d t \\
=& \left(\bar{\mathcal{N}}\left(\hat{\w},\w\right)\right)_{x} \d t + \mathcal{D}_{x} \left[\begin{array}{c}0 \\ \mathbb{F}\left(\hat{\w}\right)\d  W\end{array}\right] + \mathcal{D} \left[\begin{array}{c}0\\ \left(\mathbb{F}\left(\hat{\w}\right)\right)_{x} \d  W\end{array}\right]. \notag
\end{align}
By It\^o's formula, it holds that
\begin{align}\label{equality of fiist order estim after Ito}
 & \d  \left( \mathcal{D}\frac{\left|\w_{x}\right|^{2}}{2}\right)
 = \mathcal{D} \d  \w_{x}\cdot \w_{x} + \frac{r^{-1}\left|r \left(\mathbb{F}\left(\hat{\w}\right)\right)_{x} + r_{x}\mathbb{F}\left(\hat{\w}\right) \right|^{2}}{2}\d t  \notag\\
=& \w_{x}\mathcal{D}_{x}\left(\mathcal{A}\left(\hat{\w}\right) \w_{x}+\mathcal{B}\left(\bar{\w}\right)\w + \mathcal{C}- \mathcal{N}\left(\hat{\w},\w\right)\right)\d t- \bar{\mathcal{A}}\left(\hat{\w}\right) \w_{xx}\w_{x}\d t \notag\\
 & - \left(\left(\bar{\mathcal{A}}\left(\hat{\w}\right)\right)_{x}+ \bar{\mathcal{B}}\left(\bar{\w}\right)\right) \w_{x}\cdot \w_{x} \d t
 - \left(\bar{\mathcal{B}}\left(\bar{\w}\right)\right)_{x}\w \w_{x}\d t - \bar{\mathcal{C}}_{x} \w_{x}\d t\\
 &+ \left(\bar{\mathcal{N}}\left(\hat{\w},\w\right)\right)_{x} \w_{x}\d t + r_{x} \mathbb{F}\left(\hat{\w}\right) j_{x} \d  W + r \left(\mathbb{F}\left(\hat{\w}\right)\right)_{x}j_{x} \d  W+\frac{r^{-1}\left|r\mathbb{F}_{x}+r_{x}\mathbb{F}\right|^{2}\left(\hat{\w}\right)}{2}\d t.\notag
\end{align}
By the boundary condition $j_{x}(0)=j_{x}(1)=0$, we have
\begin{align}
\int_{0}^{1}\bar{\mathcal{A}}\left(\hat{\w}\right) \w_{xx} \cdot \w_{x} \d x = -\int_{0}^{1} \frac{1}{2}\bar{\mathcal{A}}_{x} \w_{x} \cdot \w_{x} \d x .
\end{align}
Now we collect all the terms containing $\w_{x}^{2}$, for \eqref{mathcal A} and \eqref{mathcal B},
\begin{align}
&\int_{0}^{1} \left(-\mathcal{D}_{x}\mathcal{A}\left(\hat{\w}\right)+\left(\bar{\mathcal{A}}\left(\hat{\w}\right)\right)_{x}+\bar{\mathcal{B}}\left(\bar{\w}\right)
-\frac{1}{2}\left(\bar{\mathcal{A}}\left(\hat{\w}\right)\right)_{x}\right)\w_{x} \cdot \w_{x} \d x \\
= &\int_{0}^{1} \mathcal{Q}\left(\hat{\w}\right) \w_{x}\cdot \w_{x} \d x, \notag
\end{align}
where
$\mathcal{Q}\left(\hat{\w}\right)=\left(\frac{1}{2}\mathcal{D}\mathcal{A}_{x}-\frac{1}{2}\mathcal{D}_{x}\mathcal{A}+\mathcal{D}\mathcal{B}\right)\left(\hat{\w}\right)
=\left[\begin{array}{cc}
q_{11} & q_{12} \\
 q_{11} & q_{22}
\end{array}\right]$.
$r$ is chosen such that
\begin{align}
\left|q_{12}+q_{21}\right|=&O\left(\left|\hat{\w}\right|+\left|\hat{\w}_{x}\right|\right), \\
q_{22}=&\frac{1}{2}r\left\{2\frac{\bar{J}+\hat{j}}{\bar{\rho}+\sigma}\right\}_{x}-\frac{1}{2}r_{x}\left(2\frac{\bar{J}+\hat{j}}{\bar{\rho}+\sigma}\right) +r \left(1-2\frac{\bar{J}}{\bar{\rho}^{2}}\bar{\rho}_{x}\right)\\
=& r-\frac{\bar{J}}{\bar{\rho}}r_{x}-3\frac{\bar{J}}{\bar{\rho}^{2}}\bar{\rho}_{x}r+O\left(|\hat{\w}|+|\hat{\w}_{x}|\right).\notag
\end{align}
\eqref{r satisfies in first order estimates} admits a positive solution $r(x)$.
Thus, for small $\bar{J}$, there is a constant $\eta_{1}>0$, such that
\begin{align}
\int_{0}^{1}\mathcal{Q}\w_{x}\cdot\w_{x} \d x \geqslant \eta_{1}\int_{0}^{1} j_{x}^{2} \d x + \int_{0}^{1} O\left(\left|\hat{\w}\right|+\left|\hat{\w}_{x}\right|\right)\left|\w_{x}\right|^{2}\d x,
\end{align}
where it holds that
\begin{align}
\int_{0}^{1} O\left(\left|\hat{\w}\right|+\left|\hat{\w}_{x}\right|\right)\left|\w_{x}\right|^{2}\d x \ls \left\|\hat{\w}_{x}\right\|_{\infty}\left\|\w_{x}\right\|^{2} \ls  \left\|\hat{\w}\right\|_{2}\left\|\w_{x}\right\|^{2}.
\end{align}
For the terms containing $\w\cdot \w_{x}$, we estimate
\begin{align}
&\int_{0}^{1} \left(-\mathcal{D}_{x}\mathcal{B}\left(\bar{\w}\right)+\left(\bar{\mathcal{B}}\left(\bar{\w}\right)\right)_{x}\right)\w \cdot \w_{x}\d x =\int_{0}^{1} \mathcal{D}\left(\mathcal{B}\left(\bar{\w}\right)\right)_{x} \w \cdot \w_{x}\d x\\
 =& \int_{0}^{1} O\left(\left|\w\right|\left|\w_{x}\right|\right)\d x\ls C \int_{0}^{1} \left|\w_{x}\right|^{2}\d x+C \int_{0}^{1}\left|\w\right|^{2}\d x,\notag
\end{align}
with the formula of $\mathcal{B}$ \eqref{mathcal B}.
For the terms concerning $\w_{x}$, from \eqref{mathcal C} and \eqref{mathcal N}, we have
\begin{align}
 & \left(-\mathcal{D}_{x}\mathcal{C}+\bar{\mathcal{C}}_{x}\right) \cdot \w_{x}\d t= \mathcal{D}\mathcal{C}_{x}\cdot \w_{x}\d t= - r\left(\bar{\rho} \tilde{e} \right)_{x}j_{x}\d t= r\left(\bar{\rho}\sigma+\bar{\rho}_{x}\tilde{e}\right)\sigma_{t}\d t \\
=& \d \left(\frac{1}{2}r\bar{\rho}\sigma^{2} + r\bar{\rho}_{x}\tilde{e}\sigma \right)-r\bar{\rho}_{x}\tilde{e}_{t}\sigma\d t
= \d \left(\frac{1}{2}r\bar{\rho}\sigma^{2} + r\bar{\rho}_{x}\tilde{e}\sigma \right) + O\left(\sigma^{2}+j^{2}\right)\d t, \notag
\end{align}
and
\begin{align}
\left(\bar{\mathcal{N}}\left(\hat{\w},\w\right)\right)_{x}\w_{x}=O\left(\left|\hat{\w}\right|+\left|\hat{\tilde{e}}\right|+1\right)\left(\left|\w\right|^{2}+\left|\w_{x}\right|^{2}\right).
\end{align}
Integrating them over $x$, due to $\tilde{e}_{x}=\sigma$, we have
\begin{align}
&\int_{0}^{1}\left(-\mathcal{D}_{x}\mathcal{C}+\bar{\mathcal{C}}_{x}\right) \cdot \w_{x}\d x\ls \d \int_{0}^{1} \left(\frac{1}{2}r\bar{\rho}\sigma^{2} + r\bar{\rho}_{x}\tilde{e}\sigma \right)\d x+ C\int_{0}^{1} \left(\sigma^{2}+j^{2}\right)\d x \d t\notag\\
\ls&  \d \int_{0}^{1}C \left(\sigma^{2}+\tilde{e}^{2}\right)\d x + C\int_{0}^{1} \left(\sigma^{2}+j^{2}\right)\d x\d t\\
\ls & C\d  \left(\left\|\sigma\right\|^{2}+\left\|\tilde{e}\right\|_{1}^{2}\right) + C\int_{0}^{1} \left(\sigma^{2}+j^{2}\right)\d x \d t
= C\d \left\|\sigma\right\|^{2} + C\left\|\sigma\right\|^{2} \d t , \notag
\end{align}
and
\begin{align}
\int_{0}^{1}\left(\bar{\mathcal{N}}\left(\hat{\w},\w\right)\right)_{x}\w_{x}\d x \ls C\left(\left\|\hat{\w}\right\|_{1}+\left\|\hat{\tilde{e}}\right\|_{1}+1\right)\int_{0}^{1}\left(\left|\w\right|^{2}\d x +\int_{0}^{1}\left|\w_{x}\right|^{2}\right)\d x.
\end{align}
Under \eqref{condition for F}, the assumption for $\mathbb{F}$,  we estimate
\begin{align}
\int_{0}^{1}  \frac{r^{-1}\left|r\mathbb{F}_{x}+r_{x}\mathbb{F}\right|^{2}\left(\hat{\w}\right)}{2}\d x \d t
\ls C\int_{0}^{1}\left(\left|\hat{\w}\right|^{2}+\left|\hat{\w}_{x}\right|^{2}\right)^{2}\d x \d t
\ls C\left\|\hat{\w}\right\|_{1}^{4}\d t,
\end{align}
\begin{align}
\int_{0}^{1}\frac{r^{-1}\left|r_{x}\mathbb{F}\right|^{2}\left(\hat{\w}\right)}{2}\d x \d t \ls C\int_{0}^{1}\left|\hat{\w}\right|^{4}\d x \d t\ls C\left\|\hat{\w}\right\|_{1}^{4}\d t.
\end{align}
For the stochastic term, since $\bar{J}$ is sufficiently small, we have
\begin{align}
& \mathbb{E} \left[\left| \int_{0}^{t}\int_{0}^{1} r_{x} \mathbb{F}\left(\hat{\w}\right) j_{x} \d x \d  W \right|^{m}\right] \notag\\
\ls &\mathbb{E} \left[ \left(C\int_{0}^{t}\left|\int_{0}^{1} r_{x} \mathbb{F}\left(\hat{\w}\right) j_{x} \d x \right|^{2} \d s \right)^{\frac{m}{2}} \right] \notag\\
=&\mathbb{E} \left[ \left(C\int_{0}^{t}\left|\int_{0}^{1} r_{x} \left|\bar{J}+\hat{j}\right|^{2} j_{x} \d x \right|^{2} \d s \right)^{\frac{m}{2}} \right] \\
\ls & \mathbb{E} \left[  \left(\frac{\vartheta_{9}}{4}\sup\limits_{s\in [0,t]}\int_{0}^{1}\left|\hat{j}\right|^{2}\d x\right)^{m} \right]+\mathbb{E} \left[ \left(\frac{\vartheta_{9}}{4}\sup\limits_{s\in [0,t]}\int_{0}^{1}\left|j_{x}\right|^{2}\d x\right)^{m} \right]\notag\\
 &+\mathbb{E} \left[\left(C_{\vartheta_{9}} \int_{0}^{t} \left|\int_{0}^{1} \left|\hat{j}\right|\left|\w_{x}\right| \d x \right|^{2} \d s \right)^{m}\right],\notag
\end{align}
where $\vartheta_{9}$ is chosen such that $\mathbb{E} \left[\left(\frac{\vartheta_{9}}{4}\sup\limits_{s\in [0,t]}\int_{0}^{1}\left|j\right|^{2}\d x\right)^{m} \right]$ can be balanced by the zeroth-order estimates, $\mathbb{E} \left[ \left(\frac{\vartheta_{9}}{4} \sup\limits_{s\in [0,t]}\int_{0}^{1}\left|j_{x}\right|^{2}\d x\right)^{m} \right]$ can be balanced by the left hand side of this first-order estimates. Further, there holds
\begin{align}
\mathbb{E} \left[\left( \int_{0}^{t} \left|\int_{0}^{1} \left|\hat{j}\right|\left|\w_{x}\right| d x \right|^{2} \d s \right)^{m}\right]\ls\mathbb{E} \left[\left( \int_{0}^{t} C\left\|\hat{\w}\right\|_{2}^{2}\left\|\w\right\|_{2}^{2} \d s \right)^{m}\right]. \notag
\end{align}
Combining the above terms, along with the zeroth-order estimates , we obtain
\begin{align}
& \mathbb{E} \left[\left|\sup\limits_{s\in [0,t]} \int_{0}^{1} \left(\sigma_{x}^{2}+j_{x}^{2}\right)(s)\d x +\eta_{2} \int_{0}^{t}\int_{0}^{1}j_{x}^{2} \d x \d s  \right|^{m}\right] \notag\\
\ls &  \mathbb{E} \left[\left| C\left\|\sigma(t)\right\|^{2}+ C\int_{0}^{t} \left(\left\|\hat{\w}\right\|_{2}^{2}+1\right)\left\|\w\right\|^{2}\ \d s + C \int_{0}^{t} \left(\left\|\hat{\w}\right\|_{2}+1\right)\left\|\w_{x}\right\|^{2}\d s \right|^{m}\right] \\
& + \mathbb{E} \left[\left| C\left(\left\|\hat{\w}\right\|_{1}+\left\|\hat{\tilde{e}}\right\|_{1}+1\right)\int_{0}^{t}\int_{0}^{1} \left(\left|\w_{x}\right|^{2}+\left|\w\right|^{2}\right)\d x \d s \right|^{m}\right] + \mathbb{E}\left[\left|C\int_{0}^{t}\left\|\hat{\w}\right\|_{1}^{4}\d s \right|^{m}\right] \notag\\
\ls & C_{M,m} \mathbb{E}\left[\left(\int_{0}^{t}\int_{0}^{1} \left(\left|\w_{x}\right|^{2}+\left|\w\right|^{2}\right)\d x \d s\right)^{m}\right]+ C_{M,m}t^{m}+C\mathbb{E}\left[\left( \int_{0}^{1}\left(\left|\w_{0,x}\right|^{2}+\left|\w_{0}\right|^{2}\right)\d x\right)^{m}\right], \notag
\end{align}
where $C_{M,m}$ depends on $m$ and $M=\sup\limits_{t\in[0,T]} \left\|\hat{\sigma}, \hat{\u}\right\|_{2}$, $\eta_{2}$ is a positive constant.

 Similarly, the second-order estimates of uniform upper bound can be obtained similarly to boundary estimates \eqref{important claim on boundary}:
\begin{align}
 &\mathbb{E} \left[\left|\sup\limits_{s\in [0,t]} \int_{0}^{1} \left(\sigma_{xx}^{2}+j_{xx}^{2}\right)(s)\d x +\eta_{3} \int_{0}^{t}\int_{0}^{1}j_{xx}^{2} \d x \d s  \right|^{m}\right]\notag\\
 \ls & C_{M,m} \mathbb{E}\left[\left(\int_{0}^{t}\int_{0}^{1} \left(\left|\w_{xx}\right|^{2}+\left|\w_{x}\right|^{2}+\left|\w\right|^{2}\right)\d x \d s\right)^{m}\right]+ C_{M,m}t^{m}\\
 &+C\mathbb{E}\left[\left( \int_{0}^{1}\left(\left|\w_{0,xx}\right|^{2}+\left|\w_{0,x}\right|^{2}+\left|\w_{0}\right|^{2}\right)\d x\right)^{m}\right], \notag
\end{align}
 where $C_{M,m}$ depends on $m$ and $M=\sup\limits_{t\in[0,T]} \left\|\hat{\sigma}, \hat{\u}\right\|_{2}$, $\eta_{3}$ is a positive constant.
In summary of the estimates, there holds
\begin{align}
   & \mathbb{E}\left[\left(\sup\limits_{s\in[0,t]}\int_{0}^{s} \d_{\tau }\left(\left\|\w\right\|_{2}^{2} \right) \right)^{m}\right]
  \ls C_{M,m}\left(t^{m}+\mathbb{E}\left[\left(\int_{0}^{t} \left(\left\|\w\right\|_{2}^{2}  \right) \d s \right)^{m}\right]\right).
\end{align}
By Gr\"onwall's inequality, we have $\w\in L^{m}\left(\Omega; C\left([0,T];H^{2}\left(U\right)\right)\right)$. Actually, it holds that
\begin{align}\label{local estimates}
\mathbb{E}\left[\left(\sup\limits_{s\in[0,t]}\left(\left\|\w\right\|_{2}^{2}\right) (s) \right)^{m}\right]
\ls &\mathbb{E}\left[\left(\left(\left\|\w\right\|_{2}^{2}  \right) (0)\right)^{m}\right]+C_{M,m}t^{m} \notag\\
&+ \int_{0}^{t}\left(\mathbb{E}\left[\left(\left\|\w\right\|_{2}^{2}  (0) \right)^{m}\right]+C_{M,m}t^{m}\right)C_{M,m}e^{\int_{0}^{s}C_{M,m}\d \tau}\d s\\
\ls & \left(\mathbb{E}\left[\left(\left\|\w_{0}\right\|_{2}^{2}  \right)^{m}\right]+C_{M,m}t^{m}\right) e^{C_{M,m}t}.\notag
\end{align}
From \eqref{local estimates}, by applying Kolmogorov-Centov's theorem, following the standard argument in stochastic analysis \cite{BreitFeireislHofmanova-book2018}, we deduce the time continuity of $\w$ up to a modification in a completed probability space $\left(\tilde{\Omega}, \tilde{\mathfrak{F}}, \tilde{\mathbb{P}}\right)$, and so we omit the details.

The iteration scheme is
 \begin{align}\label{iteration system}
&\d  \w_{n} +\left(\mathcal{A}\left(\w_{n-1}\right)\left(\w_{n}\right)_{x}+\mathcal{B}\left(\bar{\w}\right)\w_{n} + \mathcal{C}\right)\d t\\
&= \mathcal{N}\left(\w_{n-1},\w_{n} \right)\d t+ \left[\begin{array}{c}0 \\ \mathbb{F}\left(\w_{n-1}\right)\d  W\end{array}\right].\notag
 \end{align}
By energy estimates \eqref{local estimates}, we take $T_{0}$ such that
\begin{align}
e^{C_{M,m}T_{0}}\ls 2, \quad C_{M,m}T_{0} \ls \mathbb{E}\left[\left(\left\|\w\right\|_{2}^{2}(0) \right)^{m}\right].
\end{align}
If 
\begin{align}
\mathbb{E}\left[\left\| \w_{n-1}\right\|_{2}^{2m}\right] \ls 4\mathbb{E}\left[\left(\left\|\w_{0}\right\|_{2}^{2}\right)^{m}\right] ,
\end{align}
holds, then there holds
 \begin{align}
 \mathbb{E}\left[\left\| \w_{n}\right\|_{2}^{2m}\right] \ls 4\mathbb{E}\left[\left(\left\|\w_{0}\right\|_{2}^{2}\right)^{m}\right].
 \end{align} 
\begin{flushleft}
\textbf{Step 2: Contraction}
\end{flushleft}
 We shall prove the contraction and the local existence of classical solutions $\sigma, j$ in the space $L^{m}\left(\Omega; C\left([0,T]; H^{2}\left([0,1]\right)\right)\right)$.
$\left(\w_{n}-\w_{n-1}\right)$ satisfies
\begin{align}\label{contraction system}
& \left(\d  \w_{n}-\d  \w_{n-1}\right)+ \mathcal{A}\left(\w_{n-1}\right)\left(\left(\w_{n}\right)_{x}-\left(\w_{n-1}\right)_{x}\right)\d t\notag\\
&+\left(\mathcal{A}\left(\w_{n-1}\right)-\mathcal{A}\left(\w_{n-2}\right)\right)\left(\w_{n-1}\right)_{x}\d t +\mathcal{B}\left(\bar{\w}\right) \left(\w_{n}-\w_{n-1}\right) \d t +\mathcal{C}\left(\tilde{e}_{n}\right)-\mathcal{C}\left(\tilde{e}_{n-1}\right) \d t \\
 =& \left(\mathcal{N}\left(\w_{n-1},\w_{n} \right)-\mathcal{N}\left(\w_{n-2},\w_{n-1} \right)\right)\d t + \left[\begin{array}{c}0 \\ \left(\mathbb{F}\left(\w_{n-1}\right)-\mathbb{F}\left(\w_{n-2}\right)\right)\d  W\end{array}\right]. \notag
\end{align}
Then we multiply \eqref{contraction system} with $ \left(\w_{n}-\w_{n-1}\right)$ and integrate over $x$. By It\^o's formula, there holds
\begin{align}
\d \left|\w_{n}- \w_{n-1}\right|^{2}= 2\left(\w_{n}- \w_{n-1}\right)\cdot \left( \d  \w_{n}-\d  \w_{n-1}\right)+ \left|\mathbb{F}\left(\w_{n-1}\right)-\mathbb{F}\left(\w_{n-2}\right)\right|^{2}\d t.
\end{align}
Recalling \eqref{mathcal A}, we estimate by integration by parts
\begin{align}
 &\int_{0}^{1} \mathcal{A}\left(\w_{n-1}\right)\left(\left(\w_{n}\right)_{x}-\left(\w_{n-1}\right)_{x}\right)\cdot \left(\w_{n}-\w_{n-1}\right)\d x \d t
 \ls C \left\| \w_{n-1}\right\|_{1} \left\|\w_{n}-\w_{n-1}\right\|^{2}\d t.
\end{align}
We directly estimate
\begin{align}
&\int_{0}^{1} \left(\mathcal{A}\left(\w_{n-1}\right)-\mathcal{A}\left(\w_{n-2}\right)\right)\left(\w_{n-1}\right)_{x}\cdot \left(\w_{n}-\w_{n-1}\right)\d x \d t\\
 \ls & C \left\| \w_{n-1}\right\|_{1} \left\|\w_{n-1}-\w_{n-2}\right\| \left\|\w_{n}-\w_{n-1}\right\|\d t.\notag
\end{align}
With \eqref{mathcal B}, we have
\begin{align}
\int_{0}^{1} \mathcal{B}\left(\bar{\w}\right) \left(\w_{n}-\w_{n-1}\right)\cdot \left(\w_{n}-\w_{n-1}\right)\d x \d t \ls C\left\|\w_{n}-\w_{n-1}\right\|^{2}\d t.
\end{align}
Due to $\tilde{e}_{x}=\sigma$ and \eqref{mathcal C}, it holds that
\begin{align}
&\int_{0}^{1} \mathcal{C}\left(\tilde{e}_{n}\right)-\mathcal{C}\left(\tilde{e}_{n-1}\right)\cdot \left(\w_{n}-\w_{n-1}\right)\d x  \d t\\
 \ls &C\left\|\tilde{e}_{n}-\tilde{e}_{n-1}\right\| \left\|\w_{n}-\w_{n-1}\right\| \d t \ls C\left\|\w_{n}-\w_{n-1}\right\|^{2}\d t.\notag
\end{align}
Since \eqref{mathcal N} and
\begin{align}
&\mathcal{N}\left(\w_{n-1},\w_{n} \right)-\mathcal{N}\left(\w_{n-2},\w_{n-1} \right)\\
=&\mathcal{N}\left(\w_{n-1},\w_{n} \right)-\mathcal{N}\left(\w_{n-1},\w_{n-1} \right)+\mathcal{N}\left(\w_{n-1},\w_{n-1} \right)-\mathcal{N}\left(\w_{n-2},\w_{n-1} \right), \notag
\end{align}
 we estimate
\begin{align}
&\int_{0}^{1} \left(\mathcal{N}\left(\w_{n-1},\w_{n} \right)-\mathcal{N}\left(\w_{n-2},\w_{n-1} \right)\right) \cdot \left(\w_{n}-\w_{n-1}\right)\d x  \d t \notag\\
\ls &C\int_{0}^{1} \left(\left(\w_{n}-\w_{n-1}\right)+\left(\w_{n-1}-\w_{n-2}\right)\right) \cdot \left(\w_{n}-\w_{n-1}\right)\d x  \d t\\
\ls &C \left\|\w_{n}-\w_{n-1}\right\|^{2}\d t+\left\|\w_{n-1}-\w_{n-2}\right\|\left\|\w_{n}-\w_{n-1}\right\|\d t  .\notag
\end{align}
For the stochastic integral, under the assumption for $\mathbb{F}$ \eqref{condition for F}, we estimate
\begin{align}
&\mathbb{E}\left[\left|\int_{0}^{t}\int_{0}^{1}\left(\mathbb{F}\left(\w_{n-1}\right)-\mathbb{F}\left(\w_{n-2}\right)\right) \left(j_{n}-j_{n-1}\right)\d x\d  W \right|^{m}\right]\notag\\
\ls &  \mathbb{E}\left[\left|C\int_{0}^{t}\left|\int_{0}^{1}\left(\mathbb{F}\left(\w_{n-1}\right)-\mathbb{F}\left(\w_{n-2}\right)\right) \left(j_{n}-j_{n-1}\right)\d x\right|^{2}\d s \right|^{\frac{m}{2}}\right] \\
\ls & \mathbb{E}\left[\left|C\int_{0}^{t} \left\|J_{n-1}-J_{n-2}\right\|^{2}\left\|j_{n}-j_{n-1}\right\|^{2}\d s \right|^{\frac{m}{2}}\right] \notag\\
\ls &\frac{1}{4^{m}}\mathbb{E}\left[\left|\sup\limits_{s\in [0,t]}\left\|j_{n}-j_{n-1}\right\|^{2} \right|^{m}\right]+ \mathbb{E}\left[\left|C\int_{0}^{t} \left\|j_{n-1}-j_{n-2}\right\|^{2}\d s\right|^{m}\right].\notag
\end{align}
Immediately, there holds
\begin{align}
&\mathbb{E}\left[\left|\int_{0}^{t}\int_{0}^{1}\left|\mathbb{F}\left(\w_{n-1}\right)-\mathbb{F}\left(\w_{n-2}\right)\right|^{2}\d x\d s \right|^{m}\right]\ls \mathbb{E}\left[\left|C\int_{0}^{t}\left\|j_{n-1}-j_{n-2}\right\|^{2}\d s \right|^{m}\right]\\
\ls & \mathbb{E}\left[\left|C\int_{0}^{t}\left\|\w_{n-1}-\w_{n-2}\right\|^{2}\d s \right|^{m}\right].\notag
\end{align}
Combining the above estimates, for some $m\geqslant 2$, we have
\begin{align}
&\mathbb{E} \left[\left(\int_{0}^{t} \d  \left\|\w_{n}-\w_{n-1}\right\|^{2} \right)^{m}\right] \\
\ls &\mathbb{E}\left[\left|\int_{0}^{t} C \left(\left\|\w_{n}-\w_{n-1}\right\|^{2}+\left\|\w_{n-1}-\w_{n-2}\right\|^{2}+ \left\|\w_{n}-\w_{n-1}\right\|\left\|\w_{n-1}-\w_{n-2}\right\|\right) \d s \right|^{m}\right]\notag\\
&+\frac{1}{4^{m}}\mathbb{E}\left[\left|\sup\limits_{s\in [0,t]}\left\|j_{n}-j_{n-1}\right\|^{2} \right|^{m}\right],\notag
\end{align}
where $C$ depends on $M$.
By Cauchy's inequality and Jensen's inequality, we have
\begin{align}
&\mathbb{E} \left[ \left(\sup_{s\in[0,t]} \left\|\w_{n}-\w_{n-1}\right\|^{2}\right)^{m}\right]\notag\\
\ls &\mathbb{E} \left[ \left(\int_{0}^{t} C \left(\left\|\w_{n}-\w_{n-1}\right\|^{2}+\left\|\w_{n-1}-\w_{n-2}\right\|^{2}+ \left\|\w_{n}-\w_{n-1}\right\|\left\|\w_{n-1}-\w_{n-2}\right\|\right) \d s \right)^{m}\right]\notag\\
\ls & \mathbb{E} \left[ \left(\int_{0}^{t} C \left(\left\|\w_{n}-\w_{n-1}\right\|^{2}+\left\|\w_{n-1}-\w_{n-2}\right\|^{2}\right) \d s\right)^{m}\right]\\
\ls & \int_{0}^{t}\left(\mathbb{E} \left[\left(C_{7} \left\|\w_{n}-\w_{n-1}\right\|^{2}\right)^{m}\right]+ \mathbb{E} \left[\left(C_{7}\left\|\w_{n-1}-\w_{n-2}\right\|^{2}\right)^{m}\right]\right)\d s.\notag
\end{align}
 The higher order contraction estimates are proved similarly to zeroth-order, with the symmetrizing matrix and the boundary estimates \eqref{important claim on boundary}, and so the detailed proof is omitted here.

 In summary, we have
\begin{align}
&\mathbb{E} \left[ \left(\sup_{s\in[0,t]} \left\|\w_{n}-\w_{n-1}\right\|_{2}^{2}\right)^{m}\right] \\
\ls & \int_{0}^{t}\left(\mathbb{E} \left[\left(C_{8} \left\|\w_{n}-\w_{n-1}\right\|_{2}^{2}\right)^{m}\right]+ \mathbb{E}
\left[\left(C_{8}\left\|\w_{n-1}-\w_{n-2}\right\|_{2}^{2}\right)^{m}\right]\right)\d s. \notag
\end{align}
By Gr\"onwall's inequality, we have
\begin{align}
&\mathbb{E} \left[ \left( \sup_{s\in[0,t]} \left\|\w_{n}-\w_{n-1}\right\|_{2}^{2}\right)^{m}\right]\\
\ls & \mathbb{E} \left[ \left(\sup_{s\in[0,t]}\left\|\w_{n-1}-\w_{n-2}\right\|_{2}^{2}\right)^{m}\right]C_{8}^{m}t+\int_{0}^{t}\mathbb{E} \left[ \left(\sup_{s\in[0,\tau]}\left\|\w_{n-1}-\w_{n-2}\right\|_{2}^{2} \right)^{m}\right] \tau C_{8}^{2r} e^{C_{8}^{m}\tau} \d \tau \notag\\
\ls & 3 C_{8}^{m} \mathbb{E} \left[ \left(\sup_{s\in[0,\tau]}\left\|\w_{n-1}-\w_{n-2}\right\|_{2}^{2} \right)^{m}\right]t. \notag
\end{align}
Setting $T_{1}\ls T_{0}$ and $3C_{8}^{m}T_{1}<1$, $e^{C_{8}^{m}T_{1}}\ls 2 $, it holds that
\begin{align}\label{Contraction conclusion}
\mathbb{E} \left[ \left( \sup\limits_{s\in[0,t]} \left\|\w_{n}-\w_{n-1}\right\|_{2}\right)^{m}\right]\ls a ~\mathbb{E} \left[ \left( \sup\limits_{s\in[0,t]}\left\|\w_{n-1}-\w_{n-2}\right\|_{2}\right)^{m}\right], \quad a < 1,
 \end{align}
where $\w_{n}$ is a Cauchy sequence, $a=3C_{8}^{m}T_{1}$. 
By Banach's fixed point theorem, there exists a unique solution $\w$ in $L^{2m}\left(\Omega; C\left([0,T]; H^{2}\left([0,1]\right)\right)\right)$. Since $\phi_{xx}=\sigma$ is the unique solution under the boundary condition \eqref{boundary condtion for perturbations}, and $\phi \in L^{2m}\left(\Omega; C\left([0,T]; H^{4}\left([0,1]\right)\right)\right)$, it follows that $\tilde{e}=\phi_{x}$ is a unique solution, $\tilde{e}\in L^{2m}\left(\Omega; C\left([0,T]; H^{3}\left([0,1]\right)\right)\right)$.

By the proof of Theorem 5.2.9 in \cite{Karatzas1988}, $(\rho,J,\Phi)$ is the unique strong solutions to SEP, where $\rho,J\in C\left([0,T_{1}]; H^{2}\left(U\right)\right)$ and $\Phi \in C\left([0,T_{1}]; H^{4}\left(U\right)\right)$ hold $\mathbb{P}$ a.s. Although it is smooth in PDE, it is a strong solution in stochastic analysis. We give the definition of the local strong solution.

\begin{definition}
Let $\left(\Omega, \mathfrak{F}, \mathbb{P}\right)$ be a fixed stochastic basis with a complete right-continuous filtration $\mathfrak{F}=\left(\mathfrak{F}_{s}\right)_{s\geqslant 0}$ and $W$ be the fixed Wiener process. $\left(\rho, J, \Phi\right)$ is called a strong solution to initial and boundary problem \eqref{1-d Euler Poisson}-\eqref{condition for F}-\eqref{1-D contact boundary of rho}-\eqref{1-D contact boundary of Phi}-\eqref{initial conditions}-\eqref{Initial formula of Phi}, if:
\begin{enumerate}
  \item $\left(\rho, J, \Phi\right)$ is adapted to the filtration $\left(\mathfrak{F}_{s}\right)_{s\geqslant 0}$;
  \item $\mathbb{P}[\{\left(\rho(0), J(0), \Phi(0)\right)=\left(\rho_{0}, J_{0}, \Phi_{0}\right) \}]=1$;
  \item the equation of continuity
\begin{align}
   \rho(t)
  =  \rho_{0}- \int_0^t J_{x}  \d s,\notag
\end{align}
  holds $\mathbb{P}$ {\rm a.s.}, for any $t\in [0, T_{1}]$;
  \item The momentum equation
  \begin{align}
J(t)=& J_{0}  -\int_{0}^{t}\left(\frac{J^{2}}{\rho}+P(\rho)\right)_{x}\d s - \int_{0}^{t} J \d s +\int_{0}^{t} \rho \Phi_{x} \d s +\int_{0}^{t} \mathbb{F}\left(\rho,J\right) \d  W(s) ,
\end{align}
 holds $\mathbb{P}$ {\rm a.s.}, for any $t\in [0, T_{1}]$;
 \item The electrostatic potential equation
 \begin{align}
\Phi_{xx}  =\rho-b,
\end{align}
 holds $\mathbb{P}$ {\rm a.s.} for any $t\in [0, T_{1}]$.
\end{enumerate}
\end{definition}

\begin{remark}
Reviewing the above proof, the onto mapping and contraction \eqref{Contraction conclusion} hold for general stochastic forces with linear growth. Thus, the existence holds without \eqref{condition for F} and \eqref{assumption for initial data}. 
\end{remark}

 \subsection{Global existence}

Combining the zeroth-order, first-order, and second-order energy estimates together, and considering that $|\bar{J}|$ is sufficiently small, we have:
\begin{align}
&\mathbb{E} \left[\left|\sup\limits_{s\in[0,t]}\left(\left\|\w\right\|_{2}^{2}+\left\|\tilde{e}\right\|^2\right)(s)\right|^{m}\right] + \mathbb{E} \left[\left|\zeta \int_{0}^{t}\left(\left\|\w\right\|_{2}^{2}+\left\|\tilde{e}\right\|^2\right) (s)\d s  \right|^{m}\right]\\
\ls & \mathbb{E} \left[\left(\left\|\w_{0}\right\|_{2}^{2}+\left\|\tilde{e}_{0}\right\|^2\right)^{m}\right] +C\mathbb{E} \left[\left\|\w_{0} \right\|^{2m} \right]+ \mathbb{E} \left[\left| C\int_{0}^{t} \left| \bar{J} \right| \left\|\w\right\|_{2}^{2} \d s+C\int_{0}^{t}\left\|\w\right\|_{2}^{3}\d s \right|^{m}\right], \notag
\end{align}
where $\zeta$ is a small positive constant with $\zeta\ls \zeta_{i}$, $1\ls i\ls 13$.
Hence, it holds that
\begin{align}
&\mathbb{E} \left[\left|\sup\limits_{s\in[0,t]}\left(\left\|\w\right\|_{2}^{2}+\left\|\tilde{e}\right\|^2\right) (s)+\zeta \int_{0}^{t}\left(\left\|\w\right\|_{2}^{2}+\left\|\tilde{e}\right\|^2\right) (s)\d s\right|^{m}\right] \\
\ls &\mathbb{E} \left[\left(\left\|\w_{0}\right\|_{2}^{2}+\left\|\tilde{e}_{0}\right\|^2+\left\|\w_{0}\right\|^{2}\right)^{m}\right]+\mathbb{E} \left[\left|\int_{0}^{t} C \left(\left\|\w\right\|_{2}^{3}+\left|\bar{J}\right| \left\|\w\right\|_{2}^{2} \right)\d s \right|^{m}\right]. \notag
\end{align}
For $\left\|\w\right\|_{2} $ and $\left|\bar{J}\right|$ being small, we have
\begin{align}
&\mathbb{E} \left[\left|\sup\limits_{s\in[0,t]}\left(\left\|\w\right\|_{2}^{2}+\left\|\tilde{e}\right\|^2\right)(s)+\zeta \int_{0}^{t}\left(\left\|\w\right\|_{2}^{2}+\left\|\tilde{e}\right\|^2\right)(s)\d s -C\int_{0}^{t}\left(\left\|\w\right\|_{2}^{3}+\left|\bar{J}\right| \left\|\w\right\|_{2}^{2} \right)\d s\right|^{m}\right] \notag \\
 \ls &\mathbb{E} \left[\left(C\left(\left\|\w_{0}\right\|_{2}^{2}+\left\|\tilde{e}_{0}\right\|^2+\left\|\w_{0}\right\|^{2}\right)\right)^{m}\right],
\end{align}
and
\begin{align}\label{uniform estimates for 1-D}
\mathbb{E} \left[\left|\sup\limits_{s\in[0,t]}\left(\left\|\w\right\|_{2}^{2}+\left\|\tilde{e}\right\|^2\right)(s)\right|^{m}\right] \ls &C\mathbb{E} \left[\left(\left\|\w_{0}\right\|_{2}^{2}+\left\|\tilde{e}_{0}\right\|^2+\left\|\w_{0}\right\|^{2}\right)^{m}\right].
\end{align}

 Based on energy estimates \eqref{uniform estimates for 1-D} uniformly in $t$ and local existence, we have the global estimates by following the standard bootstrap, see also \cite{Kawashima2003LargeTimeBO, Zhang-Li-Mei2024}. The global existence of $\rho, J\in L^{m}\left(\Omega; C\left([0,T]; H^{2}\left(U\right)\right)\right)$ holds in probability space $\left(\Omega, \mathfrak{F}, \mathbb{P}\right)$.

 \section{Asymptotic stability of 1-D solutions}

 The {\it a priori} estimates is insufficient for investigating the decay rate since the {\it a priori} estimates are already in the form of time integrals rather than a differential inequality. The asymptotic decay of solution is established in this section based on the weighted decay estimates.

 \subsection{Weighted decay estimates}

 \subsubsection{Zeroth-order weighted decay estimates}\label{Zeroth-order weighted decay estimates for 1-D}
We multiply \eqref{equality for 1-D zeroth-order} with $e^{\zeta t}$. Then, we have
\begin{align}\label{equality for 1-D zeroth-order weighted estim}
& e^{\zeta t} \d \int_{0}^{1}\left(\frac{\left(\bar{\rho}j-\bar{J}\sigma\right)^{2}}{2\left(\bar{\rho}+\sigma\right)\bar{\rho}^{2}}+G\left(\bar{\rho}+\sigma\right)
 -G\left(\bar{\rho}\right)-G'\left(\bar{\rho}\right)\sigma+\frac{\tilde{e}^{2}}{2}\right)\d x \notag\\
& + e^{\zeta t}\int_{0}^{1} \frac{\left(\bar{\rho}j-\bar{J}\sigma\right)^{2}}{2\left(\bar{\rho}+\sigma\right)\bar{\rho}^{2}} \d x \d t +e^{\zeta t}\left.\left(\frac{3\bar{J}j^2}{\bar{\rho}^{2}}+O\left(\left(\left|j\right|^{3}+\left|\sigma\right|^{3}\right)\right)\right) \right|_{0}^{1}\d t  \notag\\
&+ e^{\zeta t}\int_{0}^{1}O\left(\left|\bar{J}\right|\left\|\w\right\|_{2}^{2} \right)\d x \d t + e^{\zeta t}O\left(\left\|\w\right\|_{2}^{3}\right)\d t \notag\\
&+e^{\zeta t}\int_{0}^{1}\frac{\mathbb{F}'\left(\bar{J}\right)^{2}+\mathbb{F}\left(\bar{J}\right)\mathbb{F}''\left(\bar{J}\right) }{2\bar{\rho}}j^{2} \d x \d t -e^{\zeta t}\int_{0}^{1}\frac{\mathbb{F}\left(\bar{J}\right)\mathbb{F}''\left(\bar{J}\right)}{\bar{\rho}^{2}} j\sigma \d x \d t\\
& - e^{\zeta t}\int_{0}^{1}2\frac{\left|\mathbb{F}\left(\bar{J}\right)\right|^{2}}{\bar{\rho}^{3}}\sigma^{2} \d x \d t + e^{\zeta t}\int_{0}^{1}O\left(\left(\left|j\right|^{3}+\left|\sigma\right|^{3}\right)\left|\bar{J}\right|\right)\d x \d t\notag\\
&+e^{\zeta t}\int_{0}^{1}\left(\frac{2\mathbb{F}'\left(\bar{J}\right)}{\bar{\rho}}+\frac{\bar{J}\mathbb{F}''\left(\bar{J}\right)}{\bar{\rho}}\right)\frac{1}{2}j^{2}\d x \d  W+e^{\zeta t}\int_{0}^{1}\frac{\bar{J}\mathbb{F}\left(\bar{J}\right)}{\bar{\rho}^{3}}\sigma^{2}\d x \d  W\notag\\
&+e^{\zeta t}\frac{\bar{J}\mathbb{F}'\left(\bar{J}\right)+\mathbb{F}\left(\bar{J}\right)}{\bar{\rho}^{2}}j\sigma\d  W  +e^{\zeta t}\int_{0}^{1} O\left(\left(\left|j\right|^{3}+\left|\sigma\right|^{3}\right)\right)\d x\d  W =0. \notag
\end{align}
We estimate the stochastic term in the following way:
\begin{align}
&\mathbb{E}\left[\left|\int_{0}^{t} e^{\zeta s}\int_{0}^{1}\left(\frac{2\mathbb{F}'\left(\bar{J}\right)}{\bar{\rho}}+\frac{\bar{J}\mathbb{F}''\left(\bar{J}\right)}{\bar{\rho}}\right)\frac{1}{2}j^{2}\d x\d  W \right|^{m}\right]\notag\\
\ls &  \mathbb{E}\left[\left(C\int_{0}^{t}e^{ 2 \zeta s}\left|\int_{0}^{1}\left(\frac{2\mathbb{F}'\left(\bar{J}\right)}{\bar{\rho}}+\frac{\bar{J}\mathbb{F}''\left(\bar{J}\right)}{\bar{\rho}}\right)\frac{1}{2}j^{2}\d x\right|^{2}\d s\right)^{\frac{m}{2}}\right]\notag\\
\ls & \mathbb{E}\left[\left(C\int_{0}^{t}e^{2\zeta s}\left|\bar{J}\right|^{2}\left\| j \right\|_{2}^{4} \d s\right)^{\frac{m}{2}}\right] \\
\ls & \mathbb{E}\left[\left(e^{\zeta t}\sup_{s\in [0,t]} \left|\bar{J}\right| \left\| j \right\|_{2}^{2} \right)^{\frac{m}{2}}\left(C\int_{0}^{t} e^{\zeta s}\left|\bar{J}\right|\left\| j \right\|_{2}^{2} \d s\right)^{\frac{m}{2}}\right]\notag\\
\ls & e^{\zeta m t}\mathbb{E}\left[\left(\left|\bar{J}\right|\sup_{s\in [0,t]}\left\|\w\right\|_{2}^{2} \right)^{m}\right]
+ \mathbb{E}\left[\left(C \int_{0}^{t}e^{\zeta s}\left|\bar{J}\right|\left\| j \right\|_{2}^{2} \d s\right)^{m}\right]. \notag
\end{align}
Similarly, we have
\begin{align}
& \mathbb{E}\left[\left|\int_{0}^{t}e^{\zeta s}\int_{0}^{1}\frac{\bar{J}\mathbb{F}\left(\bar{J}\right)}{\bar{\rho}^{3}}\sigma^{2}\d x \d  W \right|^{m}\right]\notag\\
 \ls & e^{\zeta m t}\mathbb{E}\left[\left(\sup_{s\in [0,t]}\left|\bar{J}\right| \left\| \sigma \right\|_{2}^{2} \right)^{m}\right] + \mathbb{E}\left[\left(C\int_{0}^{t}e^{\zeta s}\left|\bar{J}\right|\left\| \sigma \right\|_{2}^{2} \d s \right)^{m}\right]\\
 \ls & e^{\zeta m t}\mathbb{E}\left[\left(\left|\bar{J}\right|\sup_{s\in [0,t]}\left\|\w\right\|_{2}^{2} \right)^{m}\right]
+ \mathbb{E}\left[\left(C\int_{0}^{t}e^{\zeta s}\left|\bar{J}\right|\left\| \sigma \right\|_{2}^{2} \d s\right)^{m}\right], \notag
\end{align}
\begin{align}
& \mathbb{E}\left[\left|\int_{0}^{t}e^{\zeta s}\int_{0}^{1}\frac{\bar{J}\mathbb{F}'\left(\bar{J}\right)+\mathbb{F}\left(\bar{J}\right)}{\bar{\rho}^{2}}j\sigma\d  W\right|^{m}\right]\notag\\
 \ls &  e^{\zeta m t} \mathbb{E}\left[\left(\sup_{s\in [0,t]}\left|\bar{J}\right| \left\| \sigma \right\|_{2}^{2} \right)^{m}\right]+ \mathbb{E}\left[\left(C\int_{0}^{t}e^{\zeta s}\left|\bar{J}\right|\left\| j \right\|_{2}^{2} \d s \right)^{m}\right]\\
 \ls &  e^{\zeta m t}\mathbb{E}\left[\left(\left|\bar{J}\right|\sup_{s\in [0,t]}\left\|\w\right\|_{2}^{2} \right)^{m}\right]+ \mathbb{E}\left[\left(C\int_{0}^{t}e^{\zeta s}\left|\bar{J}\right|\left\| j \right\|_{2}^{2} \d s \right)^{m}\right] ,\notag
\end{align}
and
\begin{align}
&\mathbb{E}\left[\left|\int_{0}^{t}e^{\zeta s}\int_{0}^{1} O\left(\left(\left|j\right|^{3}+\left|\sigma\right|^{3}\right)\right)\d x\d  W \right|^{m}\right]\\
\ls &e^{\zeta m t}\mathbb{E}\left[\left( \sup\limits_{s\in [0, t]}\left\|\w\right\|_{2}^{3}\right)^{m}\right]+ \mathbb{E}\left[\left( C\int_{0}^{t}e^{\zeta s}\left\|\w\right\|_{2}^{3} \d s\right)^{m}\right]. \notag
\end{align}
From the zeroth-order estimates in subsection \ref{Zeroth-order estimates on w}, there holds
\begin{align}
&\mathbb{E}\left[\left|\int_{0}^{t}e^{\zeta s}\d \int_{0}^{1}\left(\frac{\left(\bar{\rho}j-\bar{J}\sigma\right)^{2}}{2\left(\bar{\rho}+\sigma\right)\bar{\rho}^{2}}+G\left(\bar{\rho}+\sigma\right)
 -G\left(\bar{\rho}\right)-G'\left(\bar{\rho}\right)\sigma+\frac{\tilde{e}^{2}}{2}\right)(s)\d x \right|^{m}\right]\notag\\
\geqslant &\mathbb{E}\left[\left|\int_{0}^{t}e^{\zeta s}\d \left( \zeta_{1}\int_{0}^{1}\left(j^{2}+\sigma^{2}+\tilde{e}^{2}\right)\d x \right)\right|^{m}\right].
\end{align}
Therefore, there holds
\begin{align}\label{zero order estimate of j summary in weighted estim}
& \mathbb{E}\left[\zeta_{1}\left|\int_{0}^{t}e^{\zeta s}\d \int_{0}^{1}\left(j^{2}+\sigma^{2}+\tilde{e}^{2}\right)\d x\right|^{m}\right]\notag\\
& + \mathbb{E}\left[\left|\int_{0}^{t} e^{\zeta s}\int_{0}^{1} \frac{\left(\bar{\rho}j-\bar{J}\sigma\right)^{2}}{2\left(\bar{\rho}+\sigma\right)\bar{\rho}^{2}} \d x \d s\right|^{m}\right] \notag \\
\ls &\mathbb{E}\left[\left|\int_{0}^{t}e^{\zeta s}\left.\left(\frac{3\bar{J}j^2}{2\bar{\rho}^{2}}+ O\left(\left(\left|j\right|^{3}+\left|\sigma\right|^{3}\right)\right)\right) \right|_{0}^{1}\d s \right|^{m}\right]+ \mathbb{E}\left[\left|\int_{0}^{t} e^{\zeta s} O\left(\left|\bar{J}\right|\left\|\w\right\|_{2}^{2} \right)\d s \right|^{m}\right]\\
& +\mathbb{E}\left[\left|\int_{0}^{t} e^{\zeta s} \int_{0}^{1} O\left(\left(\left|j\right|^{3}+\left|\sigma\right|^{3}\right)\left|\bar{J}\right|\right)\d x\d s\right|^{m}\right] \notag \\
& +   e^{\zeta m t}\mathbb{E}\left[\left(\left|\bar{J}\right|\sup_{s\in [0,t]}\left\|\w\right\|_{2}^{2} \right)^{m}\right]
+ \mathbb{E}\left[\left(C\int_{0}^{t}e^{\zeta s}\left|\bar{J}\right|\left\|\w\right\|_{2}^{2} \d s\right)^{m}\right] \notag\\
& + e^{\zeta m t}\mathbb{E}\left[\left( \sup\limits_{s\in [0, t]}\left\|\w\right\|_{2}^{3}\right)^{m}\right]+ \mathbb{E}\left[\left(C \int_{0}^{t}e^{\zeta s}\left\|\w\right\|_{2}^{3} \d s\right)^{m}\right]\notag\\
\ls &  e^{\zeta m t}\mathbb{E}\left[\left(\left|\bar{J}\right|\sup_{s\in [0,t]}\left\|\w\right\|_{2}^{2} \right)^{m}\right]
+ \mathbb{E}\left[\left(C \int_{0}^{t}e^{\zeta s}\left|\bar{J}\right|\left\|\w\right\|_{2}^{2} \d s\right)^{m}\right]\notag\\
& +  \mathbb{E}\left[\left(C \int_{0}^{t}e^{\zeta s}\left\|\w\right\|_{2}^{3} \d s\right)^{m}\right]. \notag
\end{align}
Additionally, the estimate for the dissipation in terms of $\int_{0}^{1}\left(\sigma^{2}+\tilde{e}^{2}\right)\d x$ is needed.
From the zero-th energy estimates in subsection \ref{Zeroth-order estimates on w}, we multiply \eqref{equation of E} with $ e^{\zeta t}\tilde{e}$, and then integrate it over $x$. Since $\tilde{e}_{x}=\sigma=0$ at $x=0$ and $x=1$, we have
\begin{align}\label{balance of E in weighted estim}
& e^{\zeta t}\int_{0}^{1} \frac{P'\left(\bar{\rho}\right)}{\bar{\rho}} \tilde{e}_{x}^{2} \d x \d t +  e^{\zeta t}\int_{0}^{1}\tilde{e}^{2} \d x \d t \notag\\
=& e^{\zeta t}\d  \left(\int_{0}^{1}\frac{j\tilde{e}}{\bar{\rho}} \d x\right) +  e^{\zeta t}\int_{0}^{1}\frac{j^{2}}{\bar{\rho}} \d x \d t- e^{\zeta t}\int_{0}^{1}j\d x \int_{0}^{1}\frac{j}{\bar{\rho}} \d x \d t+ e^{\zeta t}\int_{0}^{1}\frac{j\tilde{e}}{\bar{\rho}} \d x \d t  \\
& +  e^{\zeta t}\int_{0}^{1}\frac{\mathbb{F} \tilde{e}}{\bar{\rho}} \d x \d  W +  e^{\zeta t}\int_{0}^{1}O\left(\left|\bar{J}\right|\tilde{e} \left(\left|\w\right|+\left|\w_{x}\right|\right)\right)\d x \d t \notag\\
&+ e^{\zeta t}\int_{0}^{1}O\left(\left|\w\right|^{2} + \left|\w_{x}\right|^{2} + \tilde{e}^{2} \right)\tilde{e} \d x \d t. \notag
\end{align}
We deal with the right-hand side by H\"older's inequality:
\begin{align}
& -e^{\zeta t}\int_{0}^{1}j\d x \int_{0}^{1}\frac{j}{\bar{\rho}} \d x \d t + e^{\zeta t}\int_{0}^{1}\frac{j^{2}}{\bar{\rho}} \d x \d t + e^{\zeta t}\int_{0}^{1}\frac{j\tilde{e}}{\bar{\rho}} \d x \d t \\
\ls & \tilde{C} e^{\zeta t}\int_{0}^{1}j^{2}\d x \d t + e^{\zeta t}\int_{0}^{1}\frac{\tilde{e}^{2}}{2} \d x \d t,\notag
\end{align}
where $\tilde{C}$ depends on $\bar{\rho}$.
For the stochastic term, we estimate it as follows
\begin{align}
&\mathbb{E}\left[\left|\int_{0}^{t}e^{\zeta s}\int_{0}^{1}\frac{\mathbb{F}\tilde{e}}{\bar{\rho}}\d x \d  W\right|^{m}\right] \notag \\
\ls  & e^{\zeta m t}\mathbb{E}\left[\left(\left|\bar{J}\right|^{2}\sup_{s\in [0,t]}\left\|\w\right\|_{2} \right)^{m}\right] +  \mathbb{E}\left[\left(C\int_{0}^{t}e^{\zeta s}\left|\bar{J}\right|^{2} \left\|\w\right\|_{2}\d s \right)^{m}\right]\\
& +\mathbb{E}\left[\left|C\int_{0}^{t}e^{\zeta s} \left\|\w\right\|_{2}^{3}\d s\right|^{m}\right]. \notag
\end{align}
Multiplying \eqref{balance of E in weighted estim} with some small constant, adding \eqref{zero order estimate of j summary in weighted estim}, we have
\begin{align}\label{conclusion of zero order weighted estimate}
&\mathbb{E}\left[\left| \int_{0}^{t} e^{\zeta s} \d  \int_{0}^{1} \left(j^{2}+\sigma^{2}+ \tilde{e}^{2}\right)(t) \d x + \zeta \int_{0}^{t} e^{\zeta s}\int_{0}^{1} \left(j^{2}+\sigma^{2}+ \tilde{e}^{2}\right) \d x \d s \right|^{m}\right] \notag\\
\ls & e^{\zeta m t}\mathbb{E}\left[\left(\left|\bar{J}\right|\sup_{s\in [0,t]}\left\|\w\right\|_{2}^{2} \right)^{m}\right] + \mathbb{E}\left[\left(C\int_{0}^{t} e^{\zeta s}\left|\bar{J}\right| \left\|\w\right\|_{2}^{2} \d s \right)^{m}\right] \\
 &+\mathbb{E}\left[\left(C\int_{0}^{t} e^{\zeta s}\left\|\w\right\|_{2}^{3}\d s\right)^{m}\right]+e^{\zeta m t}\mathbb{E}\left[\left(\left|\bar{J}\right|^{2}\sup_{s\in [0,t]}\left\|\w\right\|_{2} \right)^{m}\right]\notag\\
  &+  \mathbb{E}\left[\left(C\int_{0}^{t}e^{\zeta s}\left|\bar{J}\right|^{2} \left\|\w\right\|_{2}\d s \right)^{m}\right]. \notag
\end{align}

\subsubsection{First-order weighted estimates}

For the first-order weighted estimates, we multiply \eqref{equality of fiist order estim after Ito} with $e^{\zeta t}$, and integrate it on $U$. For the stochastic term, we estimate
\begin{align}
& \mathbb{E} \left[\left| \int_{0}^{t}e^{\zeta s}\int_{0}^{1} r_{x} \mathbb{F} j_{x} \d x \d  W \right|^{m}\right] \notag\\
\ls & \mathbb{E} \left[ \left(C\int_{0}^{t}e^{\zeta s}\left|\int_{0}^{1} r_{x} \mathbb{F} j_{x} \d x \right|^{2} \d s \right)^{\frac{m}{2}} \right]
= \mathbb{E} \left[\left(C\int_{0}^{t}e^{\zeta s}\left|\int_{0}^{1} r_{x} \left|\bar{J}+j\right|^{2} j_{x} \d x \right|^{2} \d s \right)^{\frac{m}{2}} \right]\notag\\
= &\mathbb{E} \left[\left(C\int_{0}^{t}\left|e^{\zeta s}\int_{0}^{1} r_{x} \left(\bar{J}^{2}+2\bar{J}j+j^{2}\right) j_{x} \d x \right|^{2} \d s \right)^{\frac{m}{2}} \right] \\
\ls &\mathbb{E} \left[\left(C\left|\bar{J}\right|^{2}e^{\zeta t}\sup_{s\in[0,t]}\left\|\w\right\|_{1}\right)^{m}\right] +\mathbb{E} \left[\left(C\int_{0}^{t}e^{\zeta s}\left|\bar{J}\right|^{2}\left\|\w\right\|_{1} \d s\right)^{m} \right] \notag\\
  &+\mathbb{E} \left[\left( C \int_{0}^{t} \left|\bar{J}\right|e^{\zeta s} \left\|\w\right\|_{1}^{2} \d s \right)^{m}\right]+\mathbb{E} \left[\left(C \int_{0}^{t} \left\|\w\right\|_{1}^{3} \d s \right)^{m}\right].\notag
\end{align}
Repeating the argument in subsection \ref{Zeroth-order weighted decay estimates for 1-D}, we have
\begin{align}
&\mathbb{E}\left[\left| \int_{0}^{t}  e^{\zeta s} \d  \left( \int_{U}\left|\nabla \w \right|^{2}\right)\right|^{m}\right] + \mathbb{E}\left[\left|\int_{0}^{t} \zeta e^{\zeta s}\int_{U}\left|\nabla \w \right|^{2}\d x \d s \right|^{m}\right] \notag\\
\ls  &\mathbb{E} \left[\left(\left|\bar{J}\right|^{2}e^{\zeta t}\sup_{s\in[0,t]}\left\|\w\right\|_{1}\right)^{m}\right] +\mathbb{E} \left[\left(C\int_{0}^{t}e^{\zeta s}\left|\bar{J}\right|^{2}\left\|\w\right\|_{1} \d s\right)^{m} \right]  \\
 &+e^{\zeta m t}\mathbb{E}\left[\left(C\left|\bar{J}\right|\sup_{s\in [0,t]}\left\|\w\right\|_{2}^{2} \right)^{m}\right]+ \mathbb{E}\left[\left(C\int_{0}^{t} e^{\zeta s}\left|\bar{J}\right| \left\|\w\right\|_{2}^{2} \d s \right)^{m}\right] +\mathbb{E}\left[\left|C\int_{0}^{t} e^{\zeta s}\left\|\w\right\|_{2}^{3}\d s\right|^{m}\right]. \notag
\end{align}


\subsubsection{Second-order weighted estimates}

Similarly, we multiply \eqref{2 order balance} with  $e^{\zeta t}$, and integrate it on $U$. Repeating the procedure in subsection \ref{Zeroth-order weighted decay estimates for 1-D}, we have
\begin{align}
 &\mathbb{E}\left[\left| \int_{0}^{t}  e^{\zeta s}\d  \left(\int_{U} \left| \partial^{2} \w \right|^{2} \d x\right) \right|^{m}\right]
  + \mathbb{E}\left[\left| \int_{0}^{t} \zeta e^{\zeta s} \int_{U}\left|\partial^{2} \w \right|^{2}\d x \d s\right|^{m}\right]\notag\\
\ls  &\mathbb{E} \left[\left(C\left|\bar{J}\right|^{2}e^{\zeta t}\sup_{s\in[0,t]}\left\|\w\right\|_{2}\right)^{m}\right] +\mathbb{E} \left[\left(C\int_{0}^{t}e^{\zeta s}\left|\bar{J}\right|^{2}\left\|\w\right\|_{2} \d s\right)^{m} \right]  \\
 &+e^{\zeta m t}\mathbb{E}\left[\left(\left|\bar{J}\right|\sup_{s\in [0,t]}\left\|\w\right\|_{2}^{2} \right)^{m}\right]+ \mathbb{E}\left[\left(C\int_{0}^{t} e^{\zeta s}\left|\bar{J}\right| \left\|\w\right\|_{2}^{2} \d s \right)^{m}\right] +\mathbb{E}\left[\left|C\int_{0}^{t} e^{\zeta s}\left\|\w\right\|_{2}^{3}\d s\right|^{m}\right]. \notag
\end{align}

\subsubsection{Asymptotic stability}

Combining the weighted estimates in the previous subsections, we obtain
 \begin{align}
 &\mathbb{E}\left[\left| \int_{0}^{t} e^{\zeta s}\d  \left(\left\|\w\right\|_{2}^{2}+\left\|\tilde{e}\right\|^2\right) +\int_{0}^{t} \zeta e^{\zeta s}\left(\left\|\w\right\|_{2}^{2}+\left\|\tilde{e}\right\|^2\right) \d s \right|^{m}\right]\notag\\
  \ls   &\mathbb{E} \left[\left(C\left|\bar{J}\right|^{2}e^{\zeta t}\sup_{s\in[0,t]}\left\|\w\right\|_{2}\right)^{m}\right] +\mathbb{E} \left[\left(C\int_{0}^{t}e^{\zeta s}\left|\bar{J}\right|^{2}\left\|\w\right\|_{2} \d s\right)^{m} \right]  \\
   &+e^{\zeta m t}\mathbb{E}\left[\left(\left|\bar{J}\right|\sup_{s\in [0,t]}\left\|\w\right\|_{2}^{2} \right)^{m}\right] + \mathbb{E}\left[\left(C\int_{0}^{t} e^{\zeta s}\left|\bar{J}\right| \left\|\w\right\|_{2}^{2} \d s \right)^{m}\right]  \notag\\
 &+\mathbb{E}\left[\left|C\int_{0}^{t} e^{\zeta s}\left\|\w\right\|_{2}^{3}\d s\right|^{m}\right]. \notag
 \end{align}
Therefore, we have
 \begin{align}
 &\mathbb{E}\left[\left|  e^{\zeta t}\left( \left\|\w\right\|_{2}^{2}\right) \right|^{m}\right] \notag\\
  \ls &\mathbb{E} \left[\left(\left\|\w_{0}\right\|_{2}^{2}+\left\|\tilde{e}_{0}\right\|^2\right)^{m}\right]+\mathbb{E} \left[\left(C\left|\bar{J}\right|^{2}e^{\zeta t}\sup_{s\in[0,t]}\left\|\w\right\|_{2}\right)^{m}\right] \notag\\
  &+\mathbb{E} \left[\left(C\int_{0}^{t}e^{\zeta s}\left|\bar{J}\right|^{2}\left\|\w\right\|_{2} \d s\right)^{m} \right]+ e^{\zeta m t}\mathbb{E}\left[\left(\left|\bar{J}\right|\sup_{s\in [0,t]}\left\|\w\right\|_{2}^{2} \right)^{m}\right] \\
  & + \mathbb{E}\left[\left(C\int_{0}^{t} e^{\zeta s}\left|\bar{J}\right| \left\|\w\right\|_{2}^{2} \d s \right)^{m}\right]+\mathbb{E}\left[\left|C\int_{0}^{t} e^{\zeta s}\left\|\w\right\|_{2}^{3}\d s\right|^{m}\right]. \notag
 \end{align}
 Since $\left\|\w\right\|_{2}$ is small, we have
 \begin{align}
&\mathbb{E} \left[\left| e^{\zeta t}\left\|\w\right\|_{2}^{2}- C \left|\bar{J}\right| e^{\zeta t}\sup_{s\in [0,t]}\left\|\w\right\|_{2}^{2}- C\left|\bar{J}\right|^{2} e^{\zeta t}\sup_{s\in [0,t]}\left\|\w\right\|_{2} - \int_{0}^{t}C\left|\bar{J}\right|e^{\zeta s} \left\|\w\right\|_{2}^{2} \d s\right.\right.\notag\\
&\qquad  \left.\left.- \int_{0}^{t}C\left|\bar{J}\right|^{2}e^{\zeta s} \left\|\w\right\|_{2} \d s -\int_{0}^{t} C e^{\zeta s} \left\|\w\right\|_{2}^{3} \d s\right|^{m}\right] \\
\ls &\mathbb{E} \left[\left(\left\|\w_{0}\right\|_{2}^{2}+\left\|\tilde{e}_{0}\right\|^2\right)^{m}\right] .\notag
\end{align}
We estimate
\begin{align}
\int_{0}^{t} e^{\zeta s} \left\|\w\right\|_{2}^{3}\d s
\ls  \sup\limits_{s\in[0,t]}e^{\frac{3\zeta s}{2}}\left\|\w\right\|_{2}^{3}\int_{0}^{t}e^{-\frac{3\zeta s}{2}}\d s\ls \left(\sup\limits_{s\in[0,t]}\left(e^{\zeta s}\left\|\w\right\|_{2}^{2}\right)\right)^{\frac{3}{2}}, \\
\int_{0}^{t}  e^{\zeta s} \left|\bar{J}\right| \left\|\w\right\|_{2}^{2} \d s \ls  \left|\bar{J}\right| \sup\limits_{s\in[0,t]} \left\|\w\right\|_{2}^{2}\int_{0}^{t} e^{\zeta s} \d s \ls \left|\bar{J}\right| \sup\limits_{s\in[0,t]} \left\|\w\right\|_{2}^{2}e^{\zeta t}.
\end{align}
If we denote $y(t) = e^{\zeta t}\sup\limits_{s\in[0,t]} \left\|\w\right\|_{2}^{2}$, then there holds
\begin{align}
&- e^{\zeta t}\left|\bar{J}\right|^{2}\sup\limits_{s\in [0,t]}\left\|\w\right\|_{2}= - \left|\bar{J}\right|^{2} \frac{e^{\zeta t}\sup\limits_{s\in [0,t]}\left\|\w\right\|_{2}^{2}}{\sup\limits_{s\in [0,t]}\left\|\w\right\|_{2}}= -y(t) \frac{\left|\bar{J}\right|^{2}}{\sup\limits_{s\in [0,t]}\left\|\w\right\|_{2}}\\
& \geqslant -y(t) \frac{\left|\bar{J}\right|^{2}}{\inf\sup\limits_{s\in [0,t]}\left\|\w\right\|_{2}} \geqslant -y(t) \frac{\left|\bar{J}\right|^{2}}{\left\|\w_{0}\right\|_{2}}=-\left|\bar{J}\right| y(t),\notag
\end{align}
and
\begin{align}
&\int_{0}^{t}e^{\zeta s}\left|\bar{J}\right|^{2} \left\|\w\right\|_{2} \d s \ls \left|\bar{J}\right|^{2} \sup\limits_{s\in [0,t]}\left\|\w\right\|_{2}\int_{0}^{t}e^{\zeta s}\d s=\left|\bar{J}\right|^{2} \frac{\sup\limits_{s\in[0,t]} \left\|\w\right\|_{2}^{2}}{\sup\limits_{s\in[0,t]} \left\|\w\right\|_{2}}e^{\zeta t}\\
\ls &\left|\bar{J}\right| \sup\limits_{s\in[0,t]} \left\|\w\right\|_{2}^{2}e^{\zeta t}=\left|\bar{J}\right| y(t). \notag
\end{align}
On the account of the smallness of $\left(\left\|\w_{0}\right\|_{2}^{2}+\left\|\tilde{e}_{0}\right\|^2\right)$, we have
\begin{align}
\mathbb{E} \left[\left| \left(1-C\bar{J}\right) y(t)\left(1-\sqrt{y(t)}\right)\right|^{m}\right] \ls \mathbb{E} \left[\left(\left\|\w_{0}\right\|_{2}^{2}+\left\|\tilde{e}_{0}\right\|^2\right)^{m}\right].
\end{align}
Since $\bar{J}$ is small, we obtain the asymptotic decay estimate
\begin{align}
\mathbb{E} \left[\left|\sup\limits_{s\in[0,t]}\left\|\w\right\|_{2}^{2} \right|^{m}\right] \ls C^{m} e^{-\zeta m t}\mathbb{E} \left[\left(\left\|\w_{0}\right\|_{2}^{2}+\left\|\tilde{e}_{0}\right\|^2\right)^{m}\right],
\end{align}
where $m\geqslant 2$, and $C$ is independent of $t$.

\section{Invariant measure and stationary solutions}

In this section, the existence of invariant measure requires \eqref{condition for F} and \eqref{assumption for initial data}, which is unnecessary in the global existence.

The law generated by the initial data $\z_{0}:=\left(\rho_{0}, J_{0}, E_{0}\right)$ in probability space $\left(\Omega, \mathfrak{F}, \mathbb{P}\right)$ is denoted by $\mathcal{L}\left(\z_{0}\right)$. The system \eqref{1-d Euler Poisson}, with initial data $\z_{0}$ satisfying \eqref{assumption for initial data}, and forced by \eqref{condition for F}, admits a unique strong solution
\begin{align}
\z(t,x,\omega):=\left(\rho(t,x,\omega),J(t,x,\omega),E(t,x,\omega)\right),
\end{align}
in $C\left([0,T]; H^{3}\left(U\right)\right)\times  C\left([0,T]; H^{3}\left(U\right)\right)\times C\left([0,T]; H^{5}\left(U\right)\right)$. We denote $\z(t)$ for short.
Let $\mathcal{S}_t$ be the transition semigroup \cite{Da-Prato-Zabczyk2014}:
\begin{align}\label{def trans semigroup}
\mathcal{S}_t \psi(\z_{0})=\mathbb{E}[\psi\left(\z((t, \z_{0}))\right)],  \quad t \geqslant 0,
\end{align}
where $ \psi$ is the bounded function on $\mathscr{H}:=H^{2}\left(U\right)\times H^{2}\left(U\right)\times H^{4}\left(U\right)$, i.e., $\psi \in C_b(\mathscr{H})$.
 $\mathcal{S}(t, \z_{0}, \Gamma)$ is the transition function:
\begin{align}
\qquad \mathcal{S}(t, \z_{0}, \Gamma) := \mathcal{S}_t (\z_{0}, \Gamma)=\mathcal{S}_t \mathds{1}_{\Gamma}(\z_{0})=\mathscr{L}\left( \z(t, \z_{0})\right)(\Gamma), ~ \z_{0} \in \mathscr{H}, ~ \Gamma \in \mathscr{B}(\mathscr{H}), ~t \geqslant 0.
\end{align}

For $\v_{0}:=\left(\rho_{0}-\bar{\rho}, J_{0}-\bar{J}, E_{0}-\bar{E}\right)$ in probability space $\left(\Omega, \mathfrak{F}, \mathbb{P}\right)$, the perturbed system \eqref{1-D sto semiconductor around teady state} admits a unique strong solution
\begin{align}
\v(t,x,\omega):=\left(\w(t,x,\omega),\phi(t,x,\omega)\right):=\left(\rho_{0}-\bar{\rho}, \u_{0}, \Phi_{0}-\bar{\Phi}\right),
\end{align}
on $C\left([0,T]; H^{2}\left(U\right)\right)\times  C\left([0,T]; H^{2}\left(U\right)\right)\times C\left([0,T]; H^{4}\left(U\right)\right)$. We denote $\v(t)$ for short.
Let $\tilde{\mathcal{S}}_t$ be the transition semigroup for \eqref{1-D sto semiconductor around teady state}:
\begin{align}\label{def trans semigroup}
\tilde{\mathcal{S}}_t \psi(\v_{0})=\mathbb{E}[\psi\left(\v((t, \v_{0}))\right)],  \quad t \geqslant 0,
\end{align}
where $ \psi$ is the bounded function on $\mathscr{H}:=H^{2}\left(U\right)\times H^{2}\left(U\right)\times H^{4}\left(U\right)$. We denote $\psi \in C_b(\mathscr{H})$ and
denote $\tilde{\mathcal{S}}(t, \v_{0}, \Gamma)$ as the transition function for \eqref{1-D sto semiconductor around teady state}.

We introduce the definition of a stationary solution for \eqref{1-d Euler Poisson}.
 \begin{definition}
 A strong solution $\left(\rho; J; \Phi; W\right)$ to system \eqref{1-d Euler Poisson} and initial boundary conditions \eqref{1-D contact boundary of rho}-\eqref{1-D contact boundary of Phi}-\eqref{initial conditions}, is called stationary, provided that the
joint law of the time shift
 \begin{align}
 \left(\mathcal{S}_{\tau} \rho, \mathcal{S}_{\tau} J, \mathcal{S}_{\tau} \Phi, \mathcal{S}_\tau W\right)
 \end{align}
  on
\begin{align}
C\left([0,T]; H^{2}\left(U\right)\right)\times  C\left([0,T]; H^{2}\left(U\right)\right)\times C\left([0,T]; H^{4}\left(U\right)\right)\times C\left([0,T]; \mathcal{H}\right)
\end{align}
is independent of $\tau \geqslant 0$.
  \end{definition}
Let $\mathscr{M}\left(\mathscr{H}\right)$ be the space of all bounded measures on $\left(\mathscr{H}, \mathscr{B}\left(\mathscr{H}\right)\right)$. For any $\psi \in C_b\left(\mathscr{H}\right)$ and any $\mu \in \mathscr{M}\left(\mathscr{H}\right)$, we set the action
\begin{align}
 \mu \left( \psi \right) =\int_{\mathscr{H}} \psi(x) \mu(\d x).
\end{align}
For $t \geqslant 0$, $\mu \in \mathscr{M}\left(\mathscr{H}\right)$, $\mathcal{S}_t^*$ acts on $\mathscr{M}\left(\mathscr{H}\right)$ by
\begin{align}
\mathcal{S}_t^* \mu(\Gamma)=\int_{\mathscr{H}} \mathcal{S}(t, x, \Gamma) \mu(\d x), \quad \Gamma \in \mathscr{B}\left(\mathscr{H}\right) .
\end{align}
And there holds
\begin{align}
\left(\mathcal{S}_t^* \mu \right) \left( \psi \right)=\mu\left( \mathcal{S}_t \psi\right), \quad \forall \psi \in C_b(\mathscr{H}), \quad\mu \in \mathscr{M}\left(\mathscr{H}\right).
\end{align}
Particularly, for SEP system \eqref{1-d Euler Poisson} and $\z$ in probability space $\left(\Omega, \mathfrak{F}, \mathbb{P}\right)$, after the action of the transition semigroup $\mathcal{S}_{t}$ for \eqref{1-d system with contact boundary}, there holds
\begin{align}\label{act of semigroup St}
\mathcal{S}_t^* \mathscr{L}(\z_{0}) =\mathscr{L}(\z(t)).
\end{align}
That is,
\begin{align}
\left(\mathcal{S}_{t}\psi\right) \mathcal{L}\left(\z_{0}\right)=\mathbb{E}\left[\psi\left(\z(t)\right)\right],
\end{align}
 where $\psi\in C_{b}(\mathscr{H})$, $\mathcal{L}\left(\z_{0}\right)$ is the law generated by the initial data $\z_{0}$.

\begin{definition}
A measure $\mu$ in $\mathscr{M}(\mathscr{H})$ is said to be an invariant (stationary) measure, if
\begin{align}
\mathcal{S}_t^* \mu=\mu, \quad \forall ~t>0 .
\end{align}
\end{definition}

In particular, Dirac measure centered at the steady state $\left(\bar{\rho}, \bar{J}, \bar{E}\right)$ is the invariant measure for \eqref{steady state for 1-D contact}, since it keeps the same after the action of the transition semigroup for \eqref{steady state for 1-D contact}.

For $\z_{0} \in \mathscr{H}$ and $T>0$, the formula
\begin{align}
\frac{1}{T} \int_0^T \mathcal{S}_t(\z_{0}, \Gamma) \d t=R_T(\z_{0}, \Gamma), \quad \Gamma \in  \mathscr{B}\left(\mathscr{H}\right),
\end{align}
defines a probability measure. For any $\nu \in \mathscr{M}(\mathscr{H}), R_T^* \nu$ is defined as follows:
\begin{align}
R_T^* \nu(\Gamma)=\int_{\mathscr{H}} R_T(x, \Gamma) \nu(\d x), \quad \Gamma \in \mathscr{B}\left(\mathscr{H}\right).
\end{align}
For any $\psi \in C_b(\mathscr{H})$, there holds
\begin{align}
\left( R_T^* \nu\right)\left( \psi \right)=\frac{1}{T} \int_0^T \left( \mathcal{S}_t^* \nu\right)\left( \psi \right) \d t .
\end{align}

 $\mathcal{S}_t, t \geq 0$, is a Feller semigroup, if for arbitrary $\psi \in C_b(\mathscr{H})$, the function
\begin{align}
[0,+\infty) \times \mathscr{H}, \quad(t, x) \mapsto \mathcal{S}_t \psi(x)
\end{align}
is continuous.

The method of constructing an invariant measure described in the following theorem is due to Krylov-Bogoliubov \cite{KRYLOV-BOGOLIUBOV1937}.
\begin{theorem}
If for some $\nu \in \mathscr{M}(\mathscr{H})$ and some sequence $T_n \uparrow+\infty, R_{T_n}^* \nu \rightarrow \mu$ weakly as $n \rightarrow \infty$, then $\mu$ is an invariant measure for Feller semigroup $\mathcal{S}_t, t \geq 0$.
\end{theorem}


  The following lemma is obtained similarly to \cite{Breit-Feireisl-Hofmanova-Maslowski2019}, and we provide a proof for the convenience of the readers. $\v_{t}^{\v_{0}}$ represents the stochastic process initiated from $\v_{0}$ for the sake of expediency.
\begin{lemma}
The perturbed SEP system \eqref{1-D sto semiconductor around teady state} defines a Feller-Markov process, i.e., $\tilde{\mathcal{S}}_{t}: C_{b}(\mathscr{H})\ra C_{b}(\mathscr{H})$, and
\begin{align}
\mathbb{E}\left[\left.\psi\left(\v_{t+s}^{\v_{0}}\right)\right|\mathfrak{F}_{t}\right]=\left(\tilde{\mathcal{S}}_{s}\psi\right)\left(\v_{t}^{\v_{0}}\right), \quad \forall ~\v_{0}\in \mathscr{H}, \quad \psi\in C_{b}(\mathscr{H}), \quad \forall ~t,s>0.
\end{align}
\end{lemma}
\begin{proof}
From the continuity of solutions, it is easy to see the Feller property that $\tilde{\mathcal{S}}_{t}: C_{b}(\mathscr{H})\ra C_{b}(\mathscr{H})$ is continuous. For the Markov property, it suffices to prove
\begin{align}
\mathbb{E}\left[\psi\left(\v_{t+s}^{\v_{0}}\right) Z\right]=\mathbb{E}\left[\tilde{\mathcal{S}}_{s}\psi\left(\v_{t}^{\v_{0}}\right)Z\right],
\end{align}
where $Z\in \mathfrak{F}_{t}$.

Let $\mathbf{D}$ be any $\mathcal{F}_t$-measurable random variable. We denote $\mathbf{D}_{n}=\sum\limits_{i=1}^n \mathbf{D}^i \mathbf{1}_{\Omega^i}$, where $\mathbf{D}^i \in \mathscr{H}$ are deterministic and $\left(\Omega^i\right) \subset \mathcal{F}_t$ is a collection of disjoint sets such that $\bigcup\limits_i \Omega^i=\Omega$. $\mathbf{D}_n \rightarrow \mathbf{D}$ in $\mathscr{H}$ implies $\tilde{\mathcal{S}}_t \varphi\left(\mathbf{D}_n\right) \rightarrow \tilde{\mathcal{S}}_t \varphi(\mathbf{D})$ in $\mathscr{H}$. For every deterministic $\mathbf{D} \in \mathcal{F}_t$, the random variable $\v_{t, t+s}^{\mathbf{D}}$ depends only on the increments of the Brownian motion $W_{t+s}-W_{t}$, and hence it is independent of $\mathcal{F}_t$. Therefore,
\begin{align}
\mathbb{E}\left[\psi\left(\v_{t, t+s}^{\mathbf{D}}\right) Z\right]=\mathbb{E}\left[\psi\left(\v_{t, t+s}^{\mathbf{D}}\right)\right] \mathbb{E}[Z], \quad \forall~\mathbf{D}\in \mathcal{F}_t.
\end{align}
Since $\v_{t, t+s}^{\mathbf{D}}$ has the same law as $\v_s^{\mathbf{D}}$ by uniqueness, we have
\begin{align}
\mathbb{E}\left[\psi\left(\v_{t, t+s}^{\mathbf{D}}\right) Z\right] = \mathbb{E}\left[\psi\left(\v_s^{\mathbf{D}}\right)\right] \mathbb{E}[Z]=\tilde{S}_s \psi(\mathbf{D}) \mathbb{E}[Z] = \mathbb{E}\left[ \tilde{S}_s \psi(\mathbf{D}) Z \right].
\end{align}
So
\begin{align}
\mathbb{E}\left[\varphi\left(\v_{t, t+s}^{\mathbf{D}}\right) Z\right]=\mathbb{E}\left[\left(\tilde{S}_s \varphi\right)(\mathbf{D}) Z\right]
\end{align}
holds for every $\mathbf{D}$. By uniqueness,
\begin{align}
\v_{t+s}^{\v_{0}} = \v_{t, t+s}^{\v_{t}}, \quad  \mathbb{P}\quad {\rm a.s.},
\end{align}
which completes the proof. \hfill $\square$
\end{proof}
\smallskip
We prove the tightness of the law
\begin{align}\label{required tightness of the law}
\left\{\frac{1}{T}\int_{0}^{T} \mathcal{L}\left(\v(t)\right) \d t,\quad T>0\right\},
\end{align}
so as to apply Krylov-Bogoliubov's theorem.
\begin{theorem}
There exists an invariant measure for the system \eqref{1-D sto semiconductor around teady state}.
\end{theorem}
\begin{proof}
From the energy estimates of global existence, we know that
\begin{align}\label{energy in proof of inv meas}
\mathbb{E}\left[\left\|\w(t)\right\|_{2}^{2m}\right]\ls C\mathbb{E}\left[\left(\left\|\w_{0}\right\|_{2} + \left\|\tilde{e}(0)\right\|^{2}\right)^{m}\right],
\end{align}
where $C$ is independent of $T$.
Consequently, for compact sets
\begin{align}
B_{ L}=\left\{\left.\w(t)\in H^{2}(U)\right|\left\|\w(t)\right\|_{2}\ls L\right\},\quad L>0,
\end{align}
in $C^{1}(U)$,
there holds
\begin{align}
\frac{1}{T}\int_{0}^{T} \mathcal{L}\left(\w(t)\right)\left(B_{L}^{c}\right) \d t
=&\frac{1}{T}\int_{0}^{T}\mathbb{P}\left[\left\{\left\|\w(t)\right\|_{2}>L\right\}\right]\d t \\
\ls &\frac{1}{L^{2m}T}\int_{0}^{T}\mathbb{E}\left[\left\|\w(t)\right\|_{2}^{2m}\right]\d t \notag \\
\ls &\frac{1}{L^{2m}}C\mathbb{E}\left[\left\|\w_{0}\right\|_{2}^{2m}\right]\notag \\
 &\ra 0, \text{ as } L\ra +\infty.
\end{align}
The tightness of $\frac{1}{T}\int_{0}^{T} \mathcal{L}\left(\tilde{e}(t)\right)\d t$ is obtained similarly due to the energy estimate
\begin{align}
\mathbb{E} \left[\left| \left\|\tilde{e}(t)\right\|^{2} \right|^{m}\right] \ls C\mathbb{E}\left[\left|\left(\left\|\w_{0}\right\|_{3}^{2} + \left\|\tilde{e}(0)\right\|^{2} \right) \right|^{m}\right].
\end{align}
 Hence, the tightness of \eqref{required tightness of the law} holds. Thus, there exists an invariant measure by Krylov-Bogoliubov's theorem.   \hfill $\square$
\end{proof}
\eqref{1-d Euler Poisson} also define a Feller-Markov process, similarly to \eqref{1-D sto semiconductor around teady state}. Since $\left(\bar{\rho}, \bar{J}, \bar{E}\right)$ is smooth, by the uniqueness of solutions, $\frac{1}{T}\int_{0}^{T} \mathcal{L}\left(\rho\right)\times\mathcal{L}\left(J\right)\times\mathcal{L}\left(E\right) \d t$ is also a tight measure, which generates an invariant measure as $T\ra +\infty$. In fact, for compact sets
\begin{align}
B_{\rho,L}=\left\{\left.\rho\in H^{2}(U)\right|\left\|\rho\right\|_{2}\ls L\right\},\quad L>0,
\end{align}
in $C^{1}(U)$, there holds
\begin{align}
\frac{1}{T}\int_{0}^{T} \mathcal{L}\left(\rho_{t}^{\rho_{0}}\right)\left(B_{\rho,L}^{c}\right) \d t
=&\frac{1}{T}\int_{0}^{T}\mathbb{P}\left[\left\{\left\|\rho\right\|_{2}>L\right\}\right]\d t \\
\ls &\frac{1}{L^{2m}T}\int_{0}^{T}\mathbb{E}\left[\left\|\rho\right\|_{2}^{2m}\right]\d t \notag \\
\ls &\frac{1}{L^{2m}}C\left(\mathbb{E}\left[\left\|\rho_{0}\right\|_{2}^{2m}\right]
+\mathbb{E}\left[\left\|\bar{\rho}\right\|_{2}^{2m}\right]\right)\notag \\
 &\ra 0, \text{ as } L\ra +\infty.
\end{align}

\begin{remark}
In the above proof, we require that the constant in energy estimate \eqref{energy in proof of inv meas} is independent of $T$. That is the reason why we assume \eqref{condition for F} and \eqref{assumption for initial data}.
\end{remark} 

We care about what holds for the limit of $\frac{1}{T}\int_{0}^{T} \mathcal{L}\left(\rho\right)\times\mathcal{L}\left(J\right)\times\mathcal{L}\left(E\right)\d t$.
\begin{theorem}
The invariant measure generated by $\frac{1}{T}\int_{0}^{T} \mathcal{L}\left(\rho\right)\times\mathcal{L}\left(J\right)\times\mathcal{L}\left(E\right)\d t$, for system \eqref{1-d Euler Poisson}, is the Dirac measure of the steady state $\left(\bar{\rho}, \bar{J}, \bar{E}\right)$.
that is, the limit
\begin{align}\lim\limits_{T\ra +\infty} \frac{1}{T}\int_{0}^{T} \mathcal{L}\left(\rho\right)\times\mathcal{L}\left(J\right)\times\mathcal{L}\left(E\right)\d t=\delta_{\bar{\rho}}\times\delta_{\bar{J}}\times \delta_{\bar{E}}
\end{align}
 holds weakly.
\end{theorem}
\begin{proof}
For any $\psi\in C_{b}\left(\mathscr{H}\right)$, we have
\begin{align}
\lim_{T\ra +\infty}\frac{1}{T} \int_{0}^{T} \left(\mathcal{L}\left(\rho\right)\right)\left( \psi \right) \d t = &\lim_{T\ra +\infty} \frac{1}{T}  \int_{0}^{T} \mathbb{E}\left[ \psi\left(\rho\right) \right] \d t\\
=& \lim_{T\ra +\infty} \frac{1}{T} \int_{0}^{T} \left( \mathbb{E}\left[ \psi\left(\rho\right)-\psi\left(\bar{\rho}\right)  \right] + \mathbb{E}\left[ \psi\left(\bar{\rho}\right) \right]\right) \d t.\notag
\end{align}
We claim that $\lim_{T\ra +\infty} \frac{1}{T}\int_{0}^{T}  \mathbb{E}\left[ \psi\left(\rho\right)-\psi\left(\bar{\rho}\right)  \right] \d t=0$. Actually,
separating $\Omega$ into
$\Omega_{t}=\left\{\psi\left(\rho\right)-\psi\left(\bar{\rho}\right)\ls \frac{1}{\sqrt{s}}\right\}$, $s>0$, and $\Omega_{t}^{c}$, there holds
\begin{align}
 \mathbb{E}\left[ \psi\left(\rho\right) - \psi\left(\bar{\rho}\right) \right]=& \int_{\Omega}\left( \psi\left(\rho\right) - \psi\left(\bar{\rho}\right)\right) \mathbb{P}\left(\d \omega\right)\notag\\
= & \int_{\Omega\cap \Omega_{t}}\left( \psi\left(\rho\right) - \psi\left(\bar{\rho}\right)\right) \mathbb{P}\left(\d \omega\right)+\int_{\Omega\cap \Omega_{t}^{c}}\left( \psi\left(\rho\right) - \psi\left(\bar{\rho}\right)\right) \mathbb{P}\left(\d \omega\right) \\
=&I_{1}+ I_{2}.\notag
\end{align}
 For $I_{1}$, it holds that
\begin{align}
\lim_{T\ra +\infty} \frac{1}{T}\int_{0}^{T}  \int_{\Omega\cap \Omega_{t}}\left( \psi\left(\rho\right) - \psi\left(\bar{\rho}\right)\right) \mathbb{P}\left(\d \omega\right)\d t
\ls \lim_{T\ra +\infty} \frac{1}{T}\int_{0}^{T} \frac{1}{\sqrt{s}}\d t =0.
\end{align}
For $I_{2}$, by the weighted energy estimates and Chebyshev's inequality, with a approximation of the bounded function $\psi$ by polynomials (Weierstrass Approximation Theorem), there holds
\begin{align}
 &\int_{\Omega\cap \Omega_{t}^{c}}\left( \psi\left(\rho\right) - \psi\left(\bar{\rho}\right)\right) \mathbb{P}\left(\d \omega\right)
\ls C\mathbb{P}\left[\left\{\psi\left(\rho\right)-\psi\left(\bar{\rho}\right)> \frac{1}{\sqrt{s}}\right\}\right]\\
\ls & C\frac{\mathbb{E}\left[\left| \psi\left(\rho\right) - \psi\left(\bar{\rho}\right)\right|^{2m}\right]}{\left(\frac{1}{\sqrt{s}}\right)^{2m}}\ls Cs^{m}e^{-\gamma m s}\mathbb{E}\left[\left\| \rho_{0} - \bar{\rho}\right\|_{2}^{2m}\right].\notag
\end{align}
So we have
\begin{align}
\lim_{T\ra +\infty} \frac{1}{T}\int_{0}^{T} \int_{\Omega\cap \Omega_{t}^{c}}\left( \psi\left(\rho\right) - \psi\left(\bar{\rho}\right)\right) \mathbb{P}\left(\d \omega\right)\d t \ls \lim_{T\ra +\infty}  C\frac{1}{T}\int_{0}^{T}t^{m}e^{-\gamma m t}\d t =0.
\end{align}
Therefore, there holds
\begin{align}
\lim_{T\ra +\infty}\frac{1}{T} \int_{0}^{T} \left(\mathcal{L}\left(\rho\right)\right)\psi  \d t
=& \lim_{T\ra +\infty} \frac{1}{T} \int_{0}^{T}  \mathbb{E}\left[ \psi\left(\bar{\rho}\right) \right]\d t=\mathbb{E}\left[ \psi\left(\bar{\rho}\right) \right] = \delta_{\bar{\rho}}\psi.
\end{align}
A similar calculation shows that
\begin{align}
\lim\limits_{T\ra +\infty} \int_{0}^{T}  \frac{1}{T} \mathcal{L}\left(J\right)\d t=\delta_{\bar{J}};
\end{align}
and
\begin{align}
\lim\limits_{T\ra +\infty} \int_{0}^{T}  \frac{1}{T} \mathcal{L}\left(E\right)\d t=\delta_{\bar{E}}.
\end{align}
This completes the proof by the tightness of a joint distributions.
 \hfill $\square$
\end{proof}

\smallskip
\smallskip

\section{Appendix}\label{append}
We provide an overview of the fundamental theory concerning stochastic analysis.
Let $E$ be a separable Banach space and $\mathscr{B}(E)$ be the $\sigma$-field of its Borel subsets, respectively. Let $\left(\Omega, \mathfrak{F}, \mathbb{P}\right)$ be a stochastic basis. A filtration $\mathcal{F}=\left(\mathfrak{F}_{t}\right)_{t \in \mathbf{T}}$ is a family of $\sigma$-algebras on $\Omega$ indexed by $\mathbf{T}$ such that $\mathfrak{F}_{s} \subseteq$ $\mathfrak{F}_{t} \subseteq \mathcal{F}$, $s \ls t$, $s, t \in \mathbf{T}$. $\left(\Omega, \mathfrak{F}, \mathbb{P}\right)$ is also called a filtered space.
 We first list some definitions.
\begin{enumerate}
  \item [1.] {\bf $E$-valued random variables. \cite{Da-Prato-Zabczyk2014}}
 For $(\Omega, \mathscr{F})$ and $(E, \mathscr{E})$ being two measurable spaces, a mapping $X$ from $\Omega$ into $E$ such that the set $\{\omega \in \Omega: X(\omega) \in A\}=\{X \in A\}$ belongs to $\mathscr{F}$ for arbitrary $A \in \mathscr{E}$, is called a measurable mapping or a random variable from $(\Omega, \mathscr{F})$ into $(E, \mathscr{E})$ or an $E$-valued random variable.

  \item [2.]{\bf Strongly measurable operator valued random variables. \cite{Da-Prato-Zabczyk2014}} Let $\mathcal{U}$ and $\mathcal{H}$ be two separable Hilbert spaces which can be infinite-dimensional, and denote by $L(\mathcal{U}, \mathcal{H})$ the set of all linear bounded operators from $\mathcal{U}$ into $\mathcal{H}$. A functional operator $\Psi(\cdot)$ from $\Omega$ into $L(\mathcal{U}, \mathcal{H})$ is said to be strongly measurable, if, for arbitrary $X \in \mathcal{U}$, the function $\Psi(\cdot) X$ is measurable, as a mapping from $(\Omega, \mathscr{F})$ into $(\mathcal{H}, \mathscr{B}(\mathcal{H}))$. Let $\mathscr{L}$ be the smallest $\sigma$-field of subsets of $L(\mathcal{U}, \mathcal{H})$ containing all sets of the form
\begin{align}
\{\Psi \in L(\mathcal{U}, \mathcal{H}): \Psi X \in A\}, \quad X \in \mathcal{U}, ~A \in \mathscr{B}(\mathcal{H}).
\end{align}
Then $\Psi: \Omega \rightarrow L(\mathcal{U}, \mathcal{H})$ is a strongly measurable mapping from $(\Omega, \mathscr{F})$ into $(L(\mathcal{U}, \mathcal{H}), \mathscr{L})$.
\item  [3.]{\bf Law of a random variable.} For an $E$-valued random variable
 $X:(\Omega, \mathcal{F}) \rightarrow (E, \mathscr{E})$, 
we denote by $\mathcal{L}[X]$ the law of $X$ on $E$, that is, $\mathcal{L}[X]$ is the probability measure on $(E, \mathscr{E})$ given by
\begin{equation}
\mathcal{L}[X](A)=\mathbb{P}[X \in A], \quad A \in \mathscr{E}.\\
\end{equation} 
%

 \item [4.]
{\bf Stochastic process. \cite{Da-Prato-Zabczyk2014}}
A stochastic process $X$ is defined as an arbitrary family $X = \{X_t\}_{t\in \mathbf{T}}$ of $E$-valued random variables $X_t$, $t \in \mathbf{T}$. $X$ is also regarded as a mapping from $\Omega$ into a Banach space like $C([0, T] ; E)$ or $L^p=L^p(0, T ; E), ~1 \leq p<+\infty$, by associating $\omega \in \Omega$ with the trajectory $X(\cdot, \omega)$.
\item [5.]{\bf Cylindrical Wiener Process valued in Hilbert space.} \cite{Da-Prato-Zabczyk2014} 
 A $\mathcal{U}$-valued stochastic process $W(t), t \geqslant 0$, is called a cylindrical Wiener process, if
\begin{itemize}
 \item  $W(0)=0$;
 \item $W$ has continuous trajectories;
 \item $W$ has independent increments;
 \item the distribution of $(W(t)-W(s))$ is $\mathscr{N}(0,(t-s)), \quad 0\ls s < t$.
\end{itemize}

  \item [6.]
{\bf Adapted stochastic process.} A stochastic process $X$ is $\mathcal{F}$-adapted, if $X_{t}$ is $\mathfrak{F}_{t}$-measurable for every $t \in \mathbf{T}$; 

\item [7.]{\bf Martingale.} The $E$-valued process $X$ is called integrable provided $\mathbb{E}\left[\|X_t\|\right]<+\infty$ for every $t \in \mathbf{T}$. An integrable and adapted $E$-valued process $X_t, t \in \mathbf{T}$, is a martingale if
    \begin{itemize}
      \item  $X$ is adapted;
      \item $X_{s}=\mathbb{E}\left[X_{t} \mid \mathfrak{F}_{s}\right]$, for arbitrary $t, s \in \mathbf{T},~ 0\ls s \ls  t$.
    \end{itemize}

\item [8.]\label{stopping time def}{\bf Stopping time.} On $\left(\Omega, \mathfrak{F}, \mathbb{P}\right)$, a random time is a measurable mapping $\tau: \Omega \rightarrow \mathbf{T} \cup \infty$. A random time is a stopping time if $\{\tau \leq t\} \in \mathfrak{F}_{t}$ for every $t \in \mathbf{T}$.
For a process $X$ and a subset $V$ of the state space, we define the hitting time of $X$ in $V$ as
\begin{align}
\tau_{V}(\omega)=\inf \left\{\left.t \in \mathbf{T}\right| X_{t}(\omega) \in V\right\}.
\end{align}
If $X$ is a continuous adapted process, and $V$ is closed, then $\tau_{V}$ is a stopping time.

 \item [9.]
{\bf Modification.}
 A stochastic process $Y$ is called a modification or a version of $X$ if
\begin{align}
\mathbb{P}[\{\omega \in \Omega: X(t, \omega) \neq Y(t, \omega)\}]=0 \quad \text { for all } t \in \mathbf{T}.
\end{align}

\item [10.]
{\bf Progressive measurability.} In $\left(\Omega, \mathfrak{F}, \mathbb{P}\right)$, stochastic process $X$ is progressively measurable or simply progressively measurable, if, for $\omega\in \Omega$, $(\omega, s) \mapsto X(s,\omega),~ s \ls t$ is $\mathfrak{F}_{t} \otimes \mathscr{B}(\mathbf{T} \cap[0, t])$-measurable for every $t \in \mathbf{T}$.

\item [11.]
{\bf Progressive measurability of continuous functions.} Let $X(t), t \in[0, T]$, be a stochastically continuous and adapted process with values in a separable Banach space $E$. Then, $X$ has a progressively measurable modification.

  \item [12.]{\bf Cross quadratic variation.} Fixing a number $T>0$, we denote by $\mathcal{M}_T^2(E)$ the space of all $E$-valued continuous, square integrable martingales $M$, such that $M(0)=0$.
If $M \in \mathcal{M}_T^2\left(\mathbb{R}^1\right)$ then there exists a unique increasing predictable process $\langle M(\cdot)\rangle$, starting from $0$, such that the process
\begin{align}
M^2(t)-\langle M(\cdot)\rangle, \quad t \in[0, T]
\end{align}
is a continuous martingale. The process $\langle M(\cdot)\rangle$ is called the quadratic variation of $M$. If $M_1, M_2 \in \mathcal{M}_T^2\left(\mathbb{R}^1\right)$ then the process
\begin{align}
\left\langle M_1(t), M_2(t)\right\rangle=\frac{1}{4}\left[\left\langle\left(M_1+M_2\right)(t)\right\rangle-\left\langle\left(M_1-M_2\right)(t)\right\rangle\right]
\end{align}
is called the cross quadratic variation of $M_1, M_2$. It is the unique, predictable process with trajectories of bounded variation, starting from 0 such that
\begin{align}
M_1(t) M_2(t)-\left\langle M_1(t), M_2(t)\right\rangle, \quad t \in[0, T]
\end{align}
is a continuous martingale.\\
For $M\in\mathcal{M}_T^2(\mathcal{H})$, where $\mathcal{H}$ is Hilbert space,
the quadratic variation is defined by
\begin{align}
\langle M(t)\rangle=\sum_{i, j=1}^{\infty}\left\langle M_i(t), M_j(t)\right \rangle  e_i \otimes e_j , \quad t \in[0, T],
\end{align}
as an integrable adapted process, where $M_i(t)$ and $M_j(t)$ are in $\mathcal{M}_T^2\left(\mathbb{R}^1\right)$. If $a \in \mathcal{H}_1, b \in \mathcal{H}_2$, then $a \otimes b$ denotes a linear operator from $\mathcal{H}_2$ into $\mathcal{H}_1$ given by the formula
\begin{align}
(a \otimes b) x=a\langle b, x\rangle_{\mathcal{H}_2}, ~x \in \mathcal{H}_2.
\end{align}
We define a cross quadratic variation for $M^1 \in \mathcal{M}_T^2\left(\mathcal{H}_1\right)$, $M^2 \in \mathcal{M}_T^2\left(\mathcal{H}_2\right)$ where $\mathcal{H}_1$ and $\mathcal{H}_2$ are two Hilbert spaces. Namely we define
\begin{align}
\left\langle M^1(t), M^2(t)\right\rangle=\sum_{i, j=1}^{\infty}\left\langle M_i^1(t), M_j^2(t)\right\rangle e_i^1 \otimes e_j^2, \quad t \in[0, T],
\end{align}
where $\left\{e_i^1\right\}$ and $\left\{e_j^2\right\}$ are complete orthonormal bases in $\mathcal{H}_1$ and $\mathcal{H}_2$ respectively.

\item [13.]{\bf Stochastic integral.} Let $W$ be the Wiener process. Let $\Psi(t), t \in [0, T ]$, be a measurable Hilbert--Schmidt operators in $L(\mathcal{U}, \mathcal{H})$, which is set in the space $\mathcal{L}_{2}$ such that
    \begin{align}
    \mathbb{E}\left[\int_{0}^{t}\|\Psi(s)\|_{\mathcal{L}_{2}}^2  \d s\right]:=\mathbb{E} \int_0^t \langle\Psi(s),\Psi^{\star}(s)\rangle_{\mathcal{H}}\d s <+\infty,
    \end{align}
where $\langle\cdot,\cdot\rangle_{\mathcal{H}}$ means the inner product in $\mathcal{H}$.
 For the stochastic integral $\int_{0}^{t}\Psi \d W$, there holds
\begin{equation}
\mathbb{E}\left[\left(\int_{0}^{t}\Psi \d W\right)^{2}\right]=\mathbb{E}\left[\int_{0}^{t}\|\Psi(s)\|_{\mathcal{L}_{2}}^2 \d s\right].
\end{equation}
Furthermore, the following properties hold
\begin{itemize}
                                             \item Linearity: $\int(a \Psi_{1}+b \Psi_{2}) \d W=a \int \Psi_{1} \d W+b \int \Psi_{2} \d W$ for constants $a$ and $b$;
                                             \item Stopping property: $\int 1_{\{\cdot \leq \tau\}} \Psi \d W=\int \Psi \d M^{\tau}=\int_{0}^{\cdot \wedge \tau} \Psi \d W$;
                                             \item It\^o-isometry: for every $t$,
\begin{equation}
\mathbb{E}\left[\left(\int_{0}^{t} \Psi \d W\right)^{2}\right]=\mathbb{E}\left[\int_{0}^{t} \|\Psi(s)\|_{\mathcal{L}_{2}}^2 \d s\right].
\end{equation}
\end{itemize}

\item [14.] {\bf Dirac measure}. Let $(E, \mathscr{B}(E))$ be a measurable space. Given $x \in E$, the Dirac measure $\delta_x$ at $x$ is the measure defined by
\begin{equation}
\delta_x(A):= \begin{cases}1, & x \in A \\ 0, & x \notin A\end{cases}
\end{equation}
for each measurable set $A \subseteq E$. In this paper, there holds
$$\delta_{\bar{\rho}}=\mathcal{L}[\bar{\rho}](A) = \mathbb{P}\left[\left\{\omega\in \Omega~|~\bar{\rho}(x)\in A \right\}\right]=1.$$

\item [15.]{\bf Tightness of measures.} \cite{Billingsley2013tightness}
Let $E$ be a Hausdorff space, and let $\mathscr{E}$ be a $\sigma$-algebra on $E$. Let $\mathscr{M}$ be a collection of measures defined on $\mathscr{E}$. The collection $\mathscr{M}$ is called tight if, for any $\eps>0$, there is a compact subset $K_{\eps}$ of $E$ such that, for all measures $\mu \in \mathscr{M}$,
\begin{equation}\label{tightness def 1}
|\mu|\left(E \backslash K_{\eps}\right)<\eps,
\end{equation}
where $|\mu|$ is the total variation measure of $\mu$. More specially, for probability measures $\mu$, \eqref{tightness def 1} can be written as
\begin{equation}\label{tightness def 2}
\mu\left(K_{\eps}\right)>1-\eps.
\end{equation}

\end{enumerate}
\smallskip

 We list some important theorems in stochastic analysis.

\begin{enumerate}
\item [1.]{\bf It\^o's formula.} \cite{Ito1944,Da-Prato-Zabczyk2014} Assume that $\Psi$ is an $\mathcal{L}_{2}$-valued process stochastically integrable in $[0, T], \varphi$ being a $\mathcal{H}$-valued predictable process Bochner integrable on $[0, T], \mathbb{P}$-a.s., and $X(0)$ being a $\mathscr{F}_{0}$-measurable $\mathcal{\mathcal{H}}$-valued random variable. Then the following process
\begin{align}\label{form of X}
X(t)=X(0)+\int_0^t \varphi(s) d s+\int_0^t \Psi(s) \d W(s), \quad t \in[0, T]
\end{align}
is well defined. Assume that a function $F:[0, T] \times \mathcal{H} \rightarrow \mathbb{R}^1$ and its partial derivatives $F_t, F_x, F_{x x}$, are uniformly continuous on bounded subsets of $[0, T] \times \mathcal{H}$. Under the above conditions, $\mathbb{P}$-a.s., for all $t \in[0, T]$,
\begin{align}
~\qquad  F(t, X(t))= & F(0, X(0))+\int_0^t \left\langle F_x(s, X(s)), \Psi(s) \d W(s)\right\rangle_{\mathcal{H}} \\
& +\int_0^t\left\{F_t(s, X(s))+\left\langle F_x(s, X(s)), \varphi(s)\right\rangle_{\mathcal{H}}
 +\frac{1}{2} F_{x x}(s, X(s))\|\Psi(s)\|_{\mathcal{L}_{2}}^2 \right\} \d s. \notag
\end{align}
Applying the above formula for $F=\langle x, x \rangle_{\mathcal{H}}$, we have It\^o's formula for $\langle X, X \rangle_{\mathcal{H}}$. Then by
\begin{align}
\langle X, Y \rangle_{\mathcal{H}} =\frac{\langle X+Y, X+Y \rangle_{\mathcal{H}} -\langle X-Y, X-Y \rangle_{\mathcal{H}}}{4}
\end{align}
in Hilbert space,
 the following It\^o's formula holds for $X$ and $Y$ in form of \eqref{form of X},
\begin{equation}\label{Ito for XY}
\begin{aligned}
\langle X, Y \rangle_{\mathcal{H}} &= \langle X_{0}, Y_{0}\rangle_{\mathcal{H}} +\int \langle X, \d Y \rangle_{\mathcal{H}}\d s+ \int  \langle Y, \d X \rangle_{\mathcal{H}}\d s+\int \d\left\langle~\langle X,Y \rangle,  \langle X,Y \rangle~\right\rangle_{\mathcal{H}}^{\frac{1}{2}}\\
&=\langle X_{0}, Y_{0}\rangle_{\mathcal{H}} +\int \langle X, \d Y \rangle_{\mathcal{H}}\d s+ \int  \langle Y, \d X \d s \rangle_{\mathcal{H}}+\left\langle~\langle X,Y \rangle,  \langle X,Y \rangle~\right\rangle_{\mathcal{H}}^{\frac{1}{2}},\\
\end{aligned}
\end{equation}
where $\langle X,Y\rangle$ means the cross quadratic  variation of $X$ and $Y$ defined above. In this paper, for the convenience of notations, we still use  $\langle X,Y\rangle$ to denote $\left\langle~\langle X,Y \rangle,  \langle X,Y \rangle~\right\rangle_{\mathcal{H}}^{\frac{1}{2}}$.

\item [2.]{\bf Chebyshev's inequality}. Let $Y$ be a random variable in probability space $\left(\Omega, \mathfrak{F}, \mathbb{P}\right)$, $\varepsilon>0$. For every $0<r<\infty$, Chebyshev's inequality reads
\begin{equation}
\mathbb{P}[\left\{|Y| \geq \varepsilon \right\}] \leq \frac{1}{\varepsilon^r}\mathbb{E}\left[|Y|^r\right] .
\end{equation}

\item [3.] {\bf Burkholder-Davis-Gundy's inequality}. \cite{BDG-inequality,Da-Prato-Zabczyk2014} Let $M$ be a continuous local martingale in $\mathcal{H}$. Let $M^{\ast}=\max\limits_{0\ls s\ls t}|M(s)|$, $m \geqslant 1$. $\langle M\rangle_{T}$ denotes the quadratic variation stopped by $T$. Then, there exist constants $K^{m}$ and $K_{m}$ such that
\begin{equation}
 K_{m} \mathbb{E}\left[\left(\langle M\rangle_{T}\right)^{m}\right]\ls \mathbb{E}\left[\left(M^{\ast}_{T}\right)^{2m}\right]\ls K^{m} \mathbb{E}\left[\left(\langle M\rangle_{T}\right)^{m}\right],
\end{equation}
for every stopping time $T$. For $m\geqslant  1$, $K^{m}=\left(\frac{2m}{2m-1}\right)^{\frac{2m(2m-2)}{2}}$, which is equivalent to $e^{m}$ as $m\ra \infty$.
Specifically, for every $m \geqslant 1$, and for every $t \geqslant 0$, there holds
\begin{equation}
\mathbb{E}\left[ \sup _{s \in[0, t]}\left|\int_0^t \Psi(s) \d W(s)\right|^{2m} \right]\leq  K^{m}\left(\mathbb{E}\left[\int_0^{t} \|\Psi(s)\|_{\mathcal{L}_{2}}^2\d s\right]\right)^{m}
\end{equation}

\item [4.] {\bf Stochastic Fubini theorem}. Assume that $(E, \mathscr{E})$ is a measurable space and let
$$
\Psi:(t, \omega, x) \rightarrow \Psi(t, \omega, x)
$$
be a measurable mapping from $\left(\Omega_T \times E, \mathscr{B}(\Omega_T) \times \mathscr{B}(E)\right)$ into $\left(\mathcal{L}^{2}, \mathscr{B}\left(\mathcal{L}^{2}\right)\right)$. Moreover, assume that
\begin{equation}
\int_E \left[\mathbb{E} \int_0^T \langle\Psi(s),\Psi^{\star}(s)\rangle_{\mathcal{H}}\d t\right]^{\frac{1}{2}} \mu(\d x)<+\infty,
\end{equation}
then $\mathbb{P}$-a.s. there holds
\begin{equation}
\int_E\left[\int_0^T \Psi(t, x) \d W(t)\right] \mu(\d x)=\int_0^T\left[\int_E \Psi(t, x) \mu(\d x)\right] \d W(t).
\end{equation}

 \item [5.]\label{Centov thm}{\bf Kolmogorov-Centov's continuity theorem.} \cite{Karatzas1988,Da-Prato-Zabczyk2014} Let $(\Omega, \mathcal{F}, \mathbb{P})$ be a probability space and $\bar{X}$ a process on $[0, T]$ with values in a complete metric space $(E, \mathscr{E})$. Suppose that
\begin{equation}
\mathbb{E}\left[\left|\bar{X}_{t}-\bar{X}_{s}\right|^{a}\right] \leq C|t-s|^{1+b},
\end{equation}
for every $s<t \leq T$ and some strictly positive constants $a, b$ and $C$. Then $\bar{X}$ admits a continuous modification $X$, $\mathbb{P}\left[\left\{X_{t}=\bar{X}_{t}\right\}\right]=1$ for every $t$, and $X$ is locally H\"older continuous for every exponent $0<\gamma<\frac{b }{a},$ namely,
\begin{equation}
\mathbb{P}\left[\left\{\omega: \sum_{0<t-s<h(\omega), t, s \leq T} \frac{\left|X_{t}(\omega)-X_{s}(\omega)\right|}{|t-s|^{\gamma}} \leq \delta\right\}\right]=1,
\end{equation}
where $h(\omega)$ is an strictly positive random variable a.s., and the constant satisfies $\delta>0$.
\end{enumerate}

\smallskip
\smallskip
\section*{Acknowledgments}

The authors would like to express their thanks to Prof. Deng Zhang for the valuable discussions. This work was commenced
when L. Zhang studied in McGill University as a joint Ph.D training student. She would like to
express her gratitude to McGill University. The research of Y. Li was supported in part
by National Natural Science Foundation of China under grants 12371221, 12161141004, and 11831011. Y. Li was also grateful to the supports by the Fundamental Research Funds for the Central Universities and Shanghai Frontiers Science Center of Modern Analysis. The research of M. Mei was supported by NSERC of Canada grant RGPIN 2022-03374 and NNSFC of China grant No. W2431005.


\end{document}